\documentclass[11pt,a4paper]{article}
\usepackage[utf8]{inputenc}
\usepackage[T1]{fontenc}
\usepackage{lmodern}
\usepackage{amsmath,amssymb,amsthm}
\usepackage{mathtools}
\usepackage{hyperref}
\usepackage{geometry}
\usepackage{enumerate}
\usepackage{bm}
\usepackage{graphicx}
\usepackage{cite}
\usepackage{xcolor}
\newcommand{\rev}[1]{#1}
\newcommand{\revblock}[1]{{#1}}
\newif\ifmarkchanges
\markchangesfalse
\newcommand{\edit}[1]{{\ifmarkchanges\color{blue}\fi#1}}
\newif\ifcompactchanges
\compactchangesfalse
\newif\ifbibchanges
\bibchangesfalse
\newcommand{\bibedit}[1]{{\ifbibchanges\color{blue}\fi#1}}
\newcommand{\compactedit}[1]{{\ifcompactchanges\color{blue}\fi#1}}
\newif\ifnotationchanges
\notationchangesfalse
\newcommand{\notationedit}[1]{{\ifnotationchanges\color{blue}\else\color{black}\fi#1}}
\newif\ifoctoberchanges
\octoberchangesfalse
\newcommand{\octoberedit}[1]{{\ifoctoberchanges\color{blue}\fi#1}}
\newif\ifreviewchanges
\reviewchangesfalse
\newcommand{\reviewedit}[1]{{\ifreviewchanges\color{blue}\fi#1}}
\newtheorem{theorem}{Theorem}[section]
\newtheorem{proposition}[theorem]{Proposition}
\newtheorem{lemma}[theorem]{Lemma}
\newtheorem{corollary}[theorem]{Corollary}
\theoremstyle{definition}
\newtheorem{definition}[theorem]{Definition}
\newtheorem{remark}[theorem]{Remark}

\newcommand{\T}{\mathbb{T}}
\newcommand{\Z}{\mathbb{Z}}
\newcommand{\R}{\mathbb{R}}
\newcommand{\C}{\mathbb{C}}
\newcommand{\N}{\mathbb{N}}
\newcommand{\Sph}{\mathbb{S}}
\newcommand{\E}{\mathcal{E}}
\newcommand{\cH}{\mathcal{H}}
\newcommand{\cO}{\mathcal{O}}
\newcommand{\cG}{\mathcal{G}}
\newcommand{\cU}{\mathcal{U}}
\newcommand{\cB}{\mathcal{B}}
\newcommand{\cR}{\mathcal{R}}
\newcommand{\cW}{\mathcal{W}}
\newcommand{\cD}{\mathcal{D}}
\newcommand{\cS}{\mathcal{S}}

\newcommand{\cQ}{\mathcal{Q}}
\newcommand{\frH}{\mathfrak{H}}
\newcommand{\frh}{\mathfrak{h}}

\newcommand{\spec}{\operatorname{spec}}
\newcommand{\sess}{\operatorname{spec}_{\mathrm{ess}}}

\newcommand{\Ran}{\operatorname{Ran}}
\hypersetup{
    colorlinks=true,
    linkcolor=black,
    citecolor=black,
    urlcolor=black
}
\title{Efimov effect in an internal spectral gap of a three-particle lattice Hamiltonian}
\author{%
  Saidakhmat N.~Lakaev\quad\compactedit{\small ORCID: \href{https://orcid.org/0000-0003-4951-9340}{0000-0003-4951-9340}}\\
  \small Department of Mathematics, Samarkand State University, Samarkand, Uzbekistan\\
  \small \texttt{slakaev2019@gmail.com}
}
\date{}
\begin{document}
\maketitle
\begin{abstract}
\reviewedit{We prove the existence of an Efimov sequence in an internal gap
of the essential spectrum of a three-particle lattice Hamiltonian.
The system consists of two identical spinless fermions and one
distinguishable particle on \(\Z^3\), with repulsive pair interactions.
For each fixed fermion-to-third-particle mass ratio
\(\gamma>\gamma_*\approx13.6069657\) and total quasi-momentum near
\(\pi\mathbf1\), the mixed zero-range coupling is tuned to a
simultaneous upper-threshold resonance.
Sufficiently strong nearest-neighbour repulsion between the fermions
separates the upper fermion-pair two-cluster spectrum from the free
three-particle band. Infinitely many discrete eigenvalues accumulate
at the lower endpoint of the resulting gap, which coincides with
the free upper edge. We establish logarithmic eigenvalue-counting
asymptotics with the same positive coefficient as at zero
fermion--fermion coupling. This coefficient also governs accumulation
above the essential spectrum throughout the closed weak-coupling
interval, including its critical endpoint. We classify the channel
spectrum near the free upper edge for all nonnegative fermion
couplings.}
\end{abstract}



\section{Introduction}
\label{sec:intro}
We prove Efimov accumulation at the lower endpoint of an internal
spectral gap for two identical spinless fermions and a distinguishable
particle on a lattice. With critically tuned mixed interactions,
sufficiently strong nearest-neighbour fermion repulsion separates
the upper two-cluster spectrum from the free band while preserving
the leading counting coefficient of the system without that interaction.

\subsection{\edit{Main contributions}}
\label{subsec:intro-novelty}
\edit{The main contributions are the following:}
\begin{itemize}
\item[(N1)] the identification of an internal gap
\((E_{+,\gamma}(K),\tau_{12}^+(K,\lambda))\) in the essential
spectrum, whose lower endpoint carries Efimov accumulation;
\item[(N2)] \reviewedit{the proof that the leading logarithmic counting
coefficient \(\cU_+(K)\) of the compact part of the channel system
is independent of the fermion--fermion coupling for every
\(\lambda\ge0\), and coincides with the coefficient of the discrete
eigenvalue-counting asymptotics in the weak and separated regimes;}
\item[(N3)] a complete classification of the fermion channel
near the free upper edge for every nonnegative \(\lambda\),
\edit{with exact variational characterizations of the onset
and separation thresholds;}
\item[(N4)] \edit{the transfer of these accumulation results to
the attractive counterpart by spectral reflection
(Corollary~\ref{cor:lower-reflection}).}
\end{itemize}

\subsection{Relation to previous results}
\label{subsec:intro-previous}
The Efimov effect, predicted in \cite{Efimov1970,Efimov1971},
is the accumulation of infinitely many three-particle bound
states at a spectral threshold under suitable two-particle
resonance conditions. \octoberedit{Its rigorous continuum theory was
developed for broad classes of short-range pair potentials in
\cite{Yafaev1974,OvchinnikovSigal1979,Tamura1991}}, and logarithmic
eigenvalue-counting asymptotics were established in
\cite{Sobolev1993,Tamura1993}. \octoberedit{Albeverio, H\o egh-Krohn, and Wu
\cite{AlbeverioHoeghKrohnWu} announced, without proof,
exponential asymptotics for the eigenvalues corresponding
to spherically symmetric three-body bound states in
continuum models with pairwise point interactions.} For three-particle Schr\"odinger operators
on lattices, the essential spectrum was shown to be a finite
union of closed bounded intervals in \cite{ALM2007}. Threshold
accumulation and spectral asymptotics were studied in
\cite{Lakaev1993,ALM1998,ALM2004,ALD2007,DMS2011}, while a
general treatment of two-particle threshold effects was given in
\cite{ALMM2006}. Connections between Efimov physics and three-body
forces in atomic and nuclear systems are reviewed in
\cite{HammerNoggaSchwenk2013}.

\octoberedit{On the cubic lattice, on-site pair interactions, of rank one in
the two-particle fibres, can produce Efimov accumulation
\cite{Lakaev1993,ALM2004,ALD2007}. Higher-rank interactions allow
several bound-state branches and richer channel spectra,
including the internal-gap configuration studied here.}

\edit{For three-particle lattice Schr\"odinger operators,
Albeverio, Lakaev, and Muminov \cite{ALM2004} established a
Birman--Schwinger counting identity below the essential spectrum
(Lemmas~6.3 and~6.5) and a spectral equivalence with an auxiliary
integral operator for energies outside the union of channel
spectra (Theorem~6.6). For a lattice model of two identical
fermions and a third particle, a Birman--Schwinger counting
identity below the essential spectrum is formulated in
\cite[Theorem~7.2]{ALD2007}. The relative counting identity used here
inside the gap is formulated in
Proposition~\ref{prop:gap-BS-counting}, within the framework
of \cite{Pushnitski2009}.}

For a system of two identical fermions and a third particle,
Lakaev and Shermatov \cite{LShermatov1999} established the lattice
Efimov effect when the fermion-to-third-particle mass ratio
satisfies \(\gamma>\gamma_*\). Their result holds for total
quasi-momenta near \(3p_0\), where \(p_0\) is the unique
nondegenerate minimum of the one-particle dispersion, at a
quasi-momentum-dependent critical coupling satisfying
condition~\revblock{(5) in \cite{LShermatov1999}}.
\edit{For an attractive zero-range lattice model, Albeverio,
Dell'Antonio, and Lakaev \cite{ALD2007} studied the dependence
of the discrete spectrum on the total quasi-momentum, including
threshold counting asymptotics, with the coupling fixed at its
zero-momentum critical value. This fixed-coupling setting differs
from the fibrewise critical tuning used here. The mass-ratio
condition for the present antisymmetric model is verified
directly by the angular analysis in
Lemma~\ref{lem:upper-angular-comparison} and
Proposition~\ref{prop:upper-supercriticality}.}

In the continuum, Basti and Teta \cite{BastiTeta2017} proved
logarithmic eigenvalue-counting asymptotics for \(m<m_*\) and
finiteness of the discrete spectrum for \(m>m_*\), under their
resonance and spectral assumptions. Here \(m\) is the mass of the
third particle, with both fermion masses normalized to one. In our
notation, \(\gamma=1/m\) and \(\gamma_*=1/m_*\approx13.6069657\).
This critical value agrees with the mass-ratio criterion in
\cite{LShermatov1999} and is characterized by the scalar equation
\eqref{eq:upper-qstar} and the relation \eqref{eq:upper-gammastar}
below. Repulsively bound states for the same particle statistics
in dimensions one and two were established in
\cite{LakaevLakaev2017}.

\edit{The possibility of discrete spectrum in a gap between the
two-particle and three-particle branches of the essential spectrum
was pointed out in the author's doctoral dissertation abstract
\cite{Lakaev1992Thesis}. In particular, the Efimov effect in such
a gap was anticipated: an infinite sequence of eigenvalues may
accumulate at an edge of the three-particle branch.
This possibility was related to the boundedness and possible
separation of the two spectral branches, allowing Efimov
accumulation even in the presence of two-particle bound states.
}

Abdullaev, Khalkhuzhaev, and Kholmatov \cite{AKK2026}
considered the fibre Hamiltonians of a three-particle system on
$\mathbb{Z}^3$ consisting of two identical fermions and a third
particle, with attractive zero-range interactions only between
each fermion and the third particle. For mass ratios above a
quasimomentum-dependent threshold and sufficiently strong
attraction, they proved the existence of at least three eigenvalues,
counted with multiplicities, in the spectral gap near the upper
edge of the two-particle branch of the essential spectrum.

Symmetries between attractive and repulsive lattice systems were
studied in \cite{DeuchertEtAl2012}. For the cosine dispersion,
Remark~\ref{rem:spectral-reflection} and Corollary~\ref{cor:lower-reflection}
transfer our upper-threshold results to the attractive lower threshold
by a quasi-momentum shift and spectral reflection.

Bounded lattice dispersion permits such internal gaps. For
translation-invariant continuum three-particle Hamiltonians with
quadratic kinetic energy and short-range pair potentials satisfying
the HVZ hypotheses, the essential spectrum at fixed total momentum
is a half-line \cite{ReedSimon1978}.

\subsection{\edit{Model, notation, and main result}}
\label{subsec:intro-model}
Let \(\T=\R/(2\pi\Z)\), represented by \((-\pi,\pi]\), and equip
\(\T^3\) with normalized Haar measure \(\eta(dp)=dp/(2\pi)^3\).
Let \(\varepsilon(p)=2\sum_{j=1}^3(1-\cos p_j)\). The fibre
Hamiltonian
\[
H_{\lambda,\mu,\gamma}(K)
=H_\gamma^{\mathrm{free}}(K)+\lambda V_{12}+\mu(V_{13}+V_{23})
\]
\reviewedit{acts on the subspace \(\cH^{\mathrm a}\) of
\(L^2((\T^3)^2,\eta\otimes\eta)\) consisting of functions
antisymmetric in \(p,q\). Section~\ref{sec:model} gives its
construction in coordinate and quasi-momentum space.} The free part is multiplication by
\(\E_\gamma(K;p,q)=\varepsilon(p)+\varepsilon(q)
+\gamma\varepsilon(K-p-q)\). Here \(\gamma>0\) is the
fermion-to-third-particle mass ratio. The mixed interactions have
zero range and strength \(\mu>0\); the fermion--fermion
interaction has nearest-neighbour range and strength \(\lambda\ge0\),
with unscaled pair fibre a rank-three orthogonal projection.

For \(\gamma>\gamma_*\) and \(K\) near \(\pi\mathbf1\), the free
dispersion has a unique nondegenerate maximum at \((k_K^+,k_K^+)\).
Write \(E_+=E_{+,\gamma}(K)=\max_{p,q}\E_\gamma(K;p,q)\). The
coupling \(\mu_{\mathrm c}^+:=\mu_{\mathrm c}^+(K,\gamma)\), defined in
\eqref{eq:three-body-upper-critical}, makes both mixed-channel
fibres with spectator momentum \(k_K^+\) resonant at \(E_+\). The
fermion-channel thresholds satisfy
\(0<\lambda_{\mathrm{crit}}<\lambda_{\mathrm{sep}}\). They mark,
respectively, the onset of channel spectrum above \(E_+\) and the
separation of the upper two-cluster spectrum from the free band;
their exact definitions are
\eqref{eq:fermion-critical-coupling} and
\eqref{eq:fermion-separation-coupling}. Their dependence on
\((K,\gamma)\) is suppressed below. For
\(\lambda>\lambda_{\mathrm{sep}}\), let
\(\tau_{12}^+(K,\lambda)\) be the lower edge of the separated
upper two-cluster spectrum, as in
\eqref{eq:fermion-upper-channel-bottom}, and set
\(\cG_\lambda^+(K)=(E_+,\tau_{12}^+(K,\lambda))\).

\revblock{\noindent\begin{minipage}{\linewidth}\paragraph*{\edit{Notation.}}
\label{subsec:intro-notation-table}
The main notation is collected below.
\begin{center}
\renewcommand{\arraystretch}{1.15}
\begin{tabular}{ll}
\hline
Symbol & Meaning \\
\hline
\(\gamma\) & fermion-to-third-particle mass ratio; fixed \(>0\) \\
\(\gamma_*\) & \edit{critical mass ratio at \(K=\pi\mathbf1\); \(\approx13.6069657\)} \\
\(K\) & total quasi-momentum of the three-particle fibre \\
\(k_K^+\) & common fermion momentum at the free maximum \\
\(\kappa_K^+\) & \(K-k_K^+\); mixed-pair total momentum \\
\(\rho_K^+\) & \(K-2k_K^+\); third-particle momentum at the maximum \\
\(E_{\pm,\gamma}(K)\) & free three-particle band edges \\
\(\mu_{\mathrm c}^+=\mu_{\mathrm c}^+(K,\gamma)\) & critical mixed coupling \\
\(\lambda_{\mathrm{crit}}(K,\gamma)\) & onset of upper channel spectrum \\
\(\lambda_{\mathrm{sep}}(K,\gamma)\) & separation of upper two-cluster branches \\
\(\cG_\lambda^+(K)\) & internal gap \((E_{+,\gamma}(K),\tau_{12}^+(K,\lambda))\) \\
\(\cU_+(K)\) & Efimov counting coefficient (see Definition~\ref{def:upper-U}) \\
\(\Delta_K^+(q,z)\) & mixed-channel determinant \\
\(T_K^+(z)\) & Birman--Schwinger operator for \(\lambda=0\) \\
\(\notationedit{\widetilde T}_{\lambda,K}^+(z)\) & Birman--Schwinger operator with fermion correction \\
\(\cS_K^+(u)\), \(\widehat\cS_K^+(s)\) & logarithmic model kernel and its Fourier symbol \\
\hline
\end{tabular}
\end{center}\end{minipage}}

\begin{theorem}[Efimov accumulation at an internal threshold]
\label{thm:intro-main}
Fix \(\gamma>\gamma_*\). There is a neighbourhood \(\cO_+\) of
\(\pi\mathbf1\) such that, for each fixed \(K\in\cO_+\) and
\(\mu=\mu_{\mathrm c}^+\), the following assertions hold with a coefficient
\(\cU_+(K)\in(0,\infty)\) independent of \(\lambda\).
\begin{enumerate}[(i)]
\item For \(\lambda>\lambda_{\mathrm{sep}}\),
\(\cG_\lambda^+(K)\) is a nonempty internal gap in the
essential spectrum, and both of its endpoints belong
to that spectrum. Infinitely many discrete
eigenvalues accumulate at \(E_+\) from this gap, and for every
fixed
\(\zeta\in\cG_\lambda^+(K)\setminus
\sigma(H_{\lambda,\mu_{\mathrm c}^+,\gamma}(K))\), their number in
\((E_++\delta,\zeta)\), counted with multiplicities, satisfies
\begin{equation}
\label{eq:intro-gap-count}
N_{\mathrm{gap}}^{(\lambda)}(K,E_++\delta;\zeta)
=\cU_+(K)|\log\delta|+o(|\log\delta|),\qquad\delta\downarrow0.
\end{equation}
Here \(0<\delta<\zeta-E_+\).
\item For \(0\le\lambda\le\lambda_{\mathrm{crit}}\), the upper
essential edge is \(E_+\). Let
\(N_+^{(\lambda)}(K,E_++\delta)\) denote the number of
eigenvalues above \(E_++\delta\), counted with multiplicities.
Then
\begin{equation}
\label{eq:intro-upper-count}
N_+^{(\lambda)}(K,E_++\delta)
=\cU_+(K)|\log\delta|+o(|\log\delta|),\qquad\delta\downarrow0.
\end{equation}
\item For
\(\lambda_{\mathrm{crit}}<\lambda\le\lambda_{\mathrm{sep}}\),
the essential spectrum contains an interval immediately to the
right of \(E_+\), excluding discrete accumulation there.
\end{enumerate}
The coefficient \(\cU_+(K)\) is characterized in
Definition~\ref{def:upper-U}. All threshold limits are taken with
\(K\), \(\gamma\), \(\lambda\) and (in the gap assertions) \(\zeta\)
fixed.
\end{theorem}

The precise counting functions and the full statement appear in
Theorem~\ref{thm:main-upper}, proved in
Section~\ref{sec:upper-efimov}.

\begin{corollary}[Exponential accumulation in the gap]
\label{cor:intro-exponential}
Under the gap assumptions of Theorem~\ref{thm:intro-main},
enumerate the eigenvalues in \((E_+,\zeta)\) in nonincreasing
order, with multiplicities, as \(E_n^+(K,\lambda;\zeta)\). Then
\[
E_n^+(K,\lambda;\zeta)-E_+
=\exp[-n/\cU_+(K)+o(n)],\qquad n\to\infty.
\]
\end{corollary}

This follows from
Proposition~\ref{prop:counting-exponential-equivalence} in
Appendix~\ref{app:abstract-counting}. The exact separation
threshold satisfies
\(\lambda_{\mathrm{sep}}(K,\gamma)\le12(1+\gamma)\), with equality
at \(K=\pi\mathbf1\). At this quasi-momentum the gap is
\((24+12\gamma,\lambda+12)\) for \(\lambda>12(1+\gamma)\). For
general \(K\), Theorem~\ref{thm:tau12-asymptotic} gives
\[
\tau_{12}^+(K,\lambda)=\lambda+12+\frac{b_*(K)}{\lambda}
+O(\lambda^{-2}),\qquad\lambda\to\infty,
\]
with explicit \(b_*(K)\ge0\).

\subsection{Experimental context}
\label{subsec:intro-experiment}
\octoberedit{Experimental motivation comes from Efimov resonances
in ultracold gases
\cite{Kraemer2006,Huang2014,Zaccanti2009,NaidonEndoReview},
repulsively bound atom pairs in optical lattices
\cite{WinklerEtAl2006}, and tunable few-fermion systems
\cite{Zurn2012,Wenz2013}. Related lattice theory concerns
interacting bound pairs \cite{PetrosyanEtAl2007} and three-body
bound states \cite{ValientePetrosyanSaenz2010}.
Realizing the present regime would require resonant mixed
interactions and sufficiently strong nearest-neighbour
fermion--fermion repulsion to separate the upper two-cluster
component from the free band.}

\subsection{Proof strategy and organisation}
\label{subsec:intro-organisation}
\reviewedit{For \(\lambda=0\), localisation near the critical spectator momentum
separates the principal singular
operator from a compact remainder norm-continuous up to \(E_+\).
Homogeneous approximation and logarithmic radial coordinates give
a truncated convolution operator whose Fourier symbol determines
\(\cU_+(K)\).

Compact perturbations norm-continuous up to \(E_+\) preserve this
coefficient for the compact channel system at every fixed
\(\lambda\ge0\), and for the effective Birman--Schwinger operator
in the weak and separated regimes. The corresponding counting
identities then give the discrete eigenvalue asymptotics.}

\compactedit{Section~\ref{sec:model} constructs the model and its fibre decompositions.
Section~\ref{sec:mixed} gives the pair and channel spectra, threshold
estimates, and essential-spectrum classification. Section~\ref{sec:upper-efimov}
proves the accumulation laws, their stability, the strong-coupling
expansion, and the attractive counterpart; Section~\ref{sec:conclusion}
discusses extensions. Appendices~\ref{app:abstract-counting} and
\ref{app:technical-lemmas} give the counting-to-accumulation equivalence
and technical threshold estimates.}

\begin{remark}[Notation and conventions]
\label{rem:intro-notation}
\reviewedit{In the Efimov analysis, the mass ratio
\(\gamma>\gamma_*\) is fixed.} The neighbourhood \(\cO_+\) of
\(\pi\mathbf1\) is chosen sufficiently small and may be reduced
when needed. Unless stated otherwise, threshold limits are taken
at fixed \(K\), \(\gamma\), and coupling constants. We write
\(\cU_+(K)=\cU_K^+(1)\), suppressing the dependence on \(\gamma\).
\end{remark}

\section{Hamiltonians and direct-integral decompositions}
\label{sec:model}
\compactedit{We define the two-particle and three-particle Hamiltonians in
coordinate space, pass to the quasi-momentum representation,
and decompose them with respect to total quasi-momentum.}

\subsection{Coordinate-space formulation}
\label{subsec:coordinate-model}

\subsubsection{Physical space and one-particle Hamiltonian}
Let
\[
\widehat{\frH}=\ell^2((\Z^3)^3).
\]
Particles \(1\) and \(2\) are identical fermions, so the
physical Hilbert space is
\begin{equation}
\label{eq:position-antisymmetric-space}
\widehat{\frH}^{\mathrm a}
=
\bigl\{
\widehat f\in\widehat{\frH}:
\widehat f(x_2,x_1,x_3)=-\widehat f(x_1,x_2,x_3)
\text{ for all }x_1,x_2,x_3\in\Z^3
\bigr\}.
\end{equation}
The exchange operator is
\[
(\widehat P_{12}\widehat f)(x_1,x_2,x_3)
=
\widehat f(x_2,x_1,x_3).
\]
The one-particle kinetic energy is
\begin{equation}
\label{eq:one-particle-coordinate-hamiltonian}
(\widehat h_0u)(x)
=
\sum_{j=1}^3
\bigl[2u(x)-u(x+e_j)-u(x-e_j)\bigr],
\qquad u\in\ell^2(\Z^3),
\end{equation}
where \(e_1,e_2,e_3\) are the coordinate unit vectors.
This operator is bounded and self-adjoint, with
\[
\sigma(\widehat h_0)=[0,12],
\qquad
\|\widehat h_0\|=12.
\]
The third particle has kinetic energy \(\gamma\widehat h_0\),
where \(\gamma>0\) is the fermion-to-third-particle mass
ratio. The symbol \(I\) denotes the identity on the space
specified by the context.

\subsubsection{Two-particle operators}
On
\[
\widehat{\frh}_2
=
\ell^2(\Z^3)\otimes\ell^2(\Z^3),
\]
define the contact interaction by
\begin{equation}
\label{eq:coordinate-contact-pair}
(\widehat v\widehat u)(x,y)
=
\delta_{xy}\widehat u(x,y).
\end{equation}
Here \(\delta_{xy}\) is the Kronecker delta.
The operator \(\widehat v\) is the orthogonal projection
onto functions supported on \(x=y\). The mixed-pair
Hamiltonian is
\[
\widehat h_{\mu,\gamma}
=
\widehat h_0\otimes I
+\gamma I\otimes\widehat h_0+\mu\widehat v.
\]
Although \(\widehat v\) has infinite rank, its fibres
at fixed total pair quasi-momentum have rank one.
For the identical fermions, set
\[
\widehat{\frh}_2^{\mathrm a}
=
\bigl\{
\widehat u\in\widehat{\frh}_2:
\widehat u(x,y)=-\widehat u(y,x)
\text{ for all }x,y\in\Z^3
\bigr\}.
\]
The contact interaction vanishes on this space.
We therefore use the nearest-neighbour interaction
\begin{equation}
\label{eq:coordinate-nearest-neighbour}
(\widehat w\widehat u)(x,y)
=
\sum_{j=1}^3
\bigl[\delta_{x-y,e_j}+\delta_{x-y,-e_j}\bigr]
\widehat u(x,y).
\end{equation}
Its multiplier is the indicator of the nearest-neighbour
configurations. Hence \(\widehat w\) is an orthogonal
projection and preserves \(\widehat{\frh}_2^{\mathrm a}\).
The two-fermion Hamiltonian is
\[
\widehat h_{12,\lambda}
=
\left.
\bigl(\widehat h_0\otimes I+I\otimes\widehat h_0
+\lambda\widehat w\bigr)
\right|_{\widehat{\frh}_2^{\mathrm a}}.
\]

\subsubsection{Three-particle Hamiltonian}
The free Hamiltonian on the physical space is
\begin{equation}
\label{eq:free-three-coordinate-hamiltonian}
\widehat H_\gamma^{\mathrm{free}}
=
\left.
\bigl(
\widehat h_0\otimes I\otimes I
+I\otimes\widehat h_0\otimes I
+\gamma I\otimes I\otimes\widehat h_0
\bigr)
\right|_{\widehat{\frH}^{\mathrm a}}.
\end{equation}
\edit{On the full space \(\widehat{\frH}\), define the pair
interactions by
\[
\begin{aligned}
(\widehat V_{13}\widehat f)(x_1,x_2,x_3)
&=\delta_{x_1x_3}\widehat f(x_1,x_2,x_3),\\
(\widehat V_{23}\widehat f)(x_1,x_2,x_3)
&=\delta_{x_2x_3}\widehat f(x_1,x_2,x_3),\\
(\widehat V_{12}\widehat f)(x_1,x_2,x_3)
&=\mathbf1_{\{|x_1-x_2|_1=1\}}\widehat f(x_1,x_2,x_3),
\end{aligned}
\]
where \(|x|_1=|x_1|+|x_2|+|x_3|\). Their momentum-space fibre
representations are given in
\eqref{eq:fiber-V13}--\eqref{eq:fiber-fermion-interaction}.
They satisfy
\(\widehat P_{12}\widehat V_{13}\widehat P_{12}=\widehat V_{23}\).
The sum \(\widehat V_{13}+\widehat V_{23}\) preserves
\(\widehat{\frH}^{\mathrm a}\), although the individual operators
do not, and \(\widehat V_{12}\) commutes with \(\widehat P_{12}\).}
For fixed \(\lambda,\mu\ge0\), define
\begin{equation}
\label{eq:full-coordinate-model-with-lambda}
\widehat H_{\lambda,\mu,\gamma}
=
\widehat H_\gamma^{\mathrm{free}}
+\lambda\widehat V_{12}|_{\widehat{\frH}^{\mathrm a}}
+\mu(\widehat V_{13}+\widehat V_{23})
|_{\widehat{\frH}^{\mathrm a}}.
\end{equation}
It is bounded and self-adjoint on
\(\widehat{\frH}^{\mathrm a}\). Positive couplings describe
repulsion; \(\lambda=0\) removes the direct
fermion--fermion interaction.
Each full Hamiltonian has fixed couplings. In the threshold
results, we select a fibre \(K\) and choose
\(\mu=\mu_{\mathrm c}^+(K,\gamma)\) for that fibre.

\subsection{Quasi-momentum representation}
\label{subsec:quasimomentum-model}
Let \(\T=\R/(2\pi\Z)\), represented by \((-\pi,\pi]\),
and equip \(\T^3\) with normalized Haar measure
\[
\eta(dp)=\frac{dp_1\,dp_2\,dp_3}{(2\pi)^3}.
\]
All quasi-momentum operations are taken modulo \(2\pi\);
all torus \(L^2\) spaces use Haar measure or its products.
The Fourier transform
\[
\mathcal F:\ell^2(\Z^3)\longrightarrow L^2(\T^3,\eta)
\]
is the unitary extension of
\begin{equation}
\label{eq:one-particle-Fourier-transform}
(\mathcal Fu)(p)
=
\sum_{x\in\Z^3}u(x)e^{-i(p,x)}
\end{equation}
from finitely supported sequences. Its inverse is
\begin{equation}
\label{eq:one-particle-inverse-Fourier-transform}
(\mathcal F^{-1}g)(x)
=
\int_{\T^3}g(p)e^{i(p,x)}\,\eta(dp).
\end{equation}
One has
\begin{equation}
\label{eq:one-particle-momentum-hamiltonian}
\mathcal F\widehat h_0\mathcal F^{-1}=h_0,
\end{equation}
where \(h_0\) is multiplication by
\begin{equation}
\label{eq:cosine-dispersion-model}
\varepsilon(p)=2\sum_{j=1}^3(1-\cos p_j).
\end{equation}
Its unique minimum and maximum on the torus are
\(\varepsilon(0)=0\) and \(\varepsilon(\pi\mathbf1)=12\).
Set
\[
\mathcal F_3=\mathcal F\otimes\mathcal F\otimes\mathcal F,
\qquad
\frH=L^2((\T^3)^3).
\]
The transform \(\mathcal F_3\) maps
\(\widehat{\frH}^{\mathrm a}\) unitarily onto
\begin{equation}
\label{eq:nonfiber-momentum-antisymmetric-space}
\frH^{\mathrm a}
=
\bigl\{
F\in\frH:
F(p,q,r)=-F(q,p,r)
\text{ for almost every }(p,q,r)
\bigr\}.
\end{equation}
Using this restriction, define
\[
H_{\lambda,\mu,\gamma}
=
\mathcal F_3\widehat H_{\lambda,\mu,\gamma}
\mathcal F_3^{-1}
\quad\text{on }\frH^{\mathrm a}.
\]
Set \(\mathcal F_2=\mathcal F\otimes\mathcal F\).
For \(G\in L^2((\T^3)^2,\eta\otimes\eta)\), the mixed-pair
Hamiltonian in this representation acts as
\[
\begin{aligned}
(\mathcal F_2\widehat h_{\mu,\gamma}\mathcal F_2^{-1}G)(p,q)
={}&[\varepsilon(p)+\gamma\varepsilon(q)]G(p,q)\\
&+\mu\int_{\T^3}G(t,p+q-t)\,\eta(dt).
\end{aligned}
\]
\edit{The momentum-space Hamiltonian preserves the total
quasi-momentum. Its direct-integral decomposition and the action
of the fibre interactions are given in
Subsection~\ref{subsec:direct-integral-decompositions}, in
particular \eqref{eq:fiber-V13}--\eqref{eq:fiber-fermion-interaction}.}

\subsection{Von Neumann direct-integral decompositions}
\label{subsec:direct-integral-decompositions}
\edit{The Fourier transform \(\mathcal F_2\) identifies
\(\widehat{\frh}_2\) with \(L^2((\T^3)^2,\eta\otimes\eta)\)
and maps \(\widehat{\frh}_2^{\mathrm a}\) onto its antisymmetric
subspace. We reserve \(\cH^{\mathrm a}\) for the three-particle
fibre space defined below.}

For \(n=2,3\) and \(a\in\Z^3\), let
\[
(\widehat T_a^{(n)}\widehat f)(x_1,\ldots,x_n)
=\widehat f(x_1+a,\ldots,x_n+a).
\]
The kinetic energies commute with these simultaneous translations,
and the interaction potentials depend only on coordinate differences.
Thus the corresponding Hamiltonians commute with every
\(\widehat T_a^{(n)}\). The translations also preserve the
antisymmetric subspaces. With \(\mathcal F_n=\mathcal F^{\otimes n}\),
\[
(\mathcal F_n\widehat T_a^{(n)}\mathcal F_n^{-1}F)(p_1,\ldots,p_n)
=e^{ia\cdot(p_1+\cdots+p_n)}F(p_1,\ldots,p_n).
\]
The characters \(e^{ia\cdot Q}\), \(a\in\Z^3\), generate the
von Neumann algebra of multiplication by functions of
\(Q=p_1+\cdots+p_n\) in \(L^\infty(\T^3)\).
Commutation with this algebra implies decomposability with
respect to \(Q\); see also \cite[Section~3]{ALM2007}.

\subsubsection{Two-particle fibres}
Define
\[
(\mathcal U_2G)(k;p)=G(p,k-p).
\]
The map \((p,q)\mapsto(p+q,p)\) preserves product Haar measure.
Consequently,
\[
\mathcal U_2:L^2((\T^3)^2,\eta\otimes\eta)
\longrightarrow\int_{\T^3}^{\oplus}L^2(\T^3,\eta)\,\eta(dk)
\]
is unitary, with
\((\mathcal U_2^{-1}g)(p,q)=g(p+q;p)\).
In these variables, the translation characters depend only on \(k\), and
\[
\mathcal U_2\mathcal F_2\widehat h_{\mu,\gamma}
\mathcal F_2^{-1}\mathcal U_2^{-1}
=\int_{\T^3}^{\oplus}h_{\mu,\gamma}(k)\,\eta(dk).
\]
The mixed-pair fibre on \(\frh=L^2(\T^3,\eta)\) is
\begin{equation}
\label{eq:mixed-pair-fibre-definition}
h_{\mu,\gamma}(k)=h_{0,\gamma}(k)+\mu v,
\qquad
(h_{0,\gamma}(k)u)(p)=\E_\gamma(k;p)u(p),
\end{equation}
where
\[
\E_\gamma(k;p)=\varepsilon(p)+\gamma\varepsilon(k-p),
\qquad
(vu)(p)=\int_{\T^3}u(t)\,\eta(dt).
\]
Thus \(v\) is the rank-one orthogonal projection onto the constants.

For the fermion pair, define
\begin{equation}
\label{eq:odd-pair-space}
\frh^{\mathrm a}(k)
=
\bigl\{
u\in L^2(\T^3):
u(p)=-u(k-p)\text{ for almost every }p
\bigr\}.
\end{equation}
The restriction of \(\mathcal U_2\mathcal F_2\) gives
\[
\mathcal U_2\mathcal F_2
\widehat h_{12,\lambda}
\mathcal F_2^{-1}\mathcal U_2^{-1}
=
\int_{\T^3}^{\oplus}h_{12,\lambda}(k)\,\eta(dk)
\]
on \(\int_{\T^3}^{\oplus}\frh^{\mathrm a}(k)\,\eta(dk)\).

The field of antisymmetric pair spaces is measurable: its orthogonal
projection is
\[
(\Pi^{\mathrm a}(k)u)(p)=\tfrac12[u(p)-u(k-p)],
\]
and \(k\mapsto\Pi^{\mathrm a}(k)\) is strongly continuous.
To give the fermion-pair fibres explicitly, choose a measurable
representative of \(k\) in \(\R^3\), use its half, and set
\begin{equation}
\label{eq:odd-form-factors}
\phi_{j,k}(p)=\sqrt2\sin(p_j-k_j/2),\qquad j=1,2,3.
\end{equation}
These functions form an orthonormal triple in \(\frh^{\mathrm a}(k)\).
Changing the lift changes only their signs and leaves the following
operator unchanged:
\begin{equation}
\label{eq:odd-rank-three-interaction}
v_{12}(k)=\sum_{j=1}^3|\phi_{j,k}\rangle\langle\phi_{j,k}|.
\end{equation}
\revblock{Indeed, in the centred coordinates \(p=k/2+x\),
\(t=k/2+y\), antisymmetry becomes oddness under \(x\mapsto-x\),
and the kernel reduces to
\(2\sum_{j=1}^3\sin x_j\sin y_j\).}
The fermion-pair Hamiltonian is therefore
\begin{equation}
\label{eq:fermion-pair-fiber}
\begin{aligned}
h_{12,\lambda}(k)&=h_{12,0}(k)+\lambda v_{12}(k)
\quad\text{on }\frh^{\mathrm a}(k),\\
(h_{12,0}(k)u)(p)&=\mathcal E_{12}(k;p)u(p),
\qquad\mathcal E_{12}(k;p)=\varepsilon(p)+\varepsilon(k-p).
\end{aligned}
\end{equation}

\subsubsection{Three-particle fibres}
For the three-particle decomposition, let
\[
\cH=L^2((\T^3)^2,\eta\otimes\eta),
\]
and
\begin{equation}
\label{eq:antisym-space}
\cH^{\mathrm a}
=
\bigl\{
f\in\cH:
f(p,q)=-f(q,p)
\text{ for almost every }(p,q)
\bigr\}.
\end{equation}
Define
\begin{equation}
\label{eq:three-particle-fiber-map}
(\mathcal U_3F)(K;p,q)=F(p,q,K-p-q)
\quad\text{for almost every }(K,p,q).
\end{equation}
The measure-preserving change of variables
\((p,q,r)\mapsto(p+q+r,p,q)\) shows that
\[
\mathcal U_3:\frH^{\mathrm a}
\longrightarrow
\int_{\T^3}^{\oplus}\cH^{\mathrm a}\,\eta(dK)
\]
is unitary, with inverse
\[
(\mathcal U_3^{-1}g)(p,q,r)
=
g(p+q+r;p,q)
\quad\text{for almost every }(p,q,r).
\]
Consequently,
\begin{equation}
\label{eq:three-particle-direct-integral}
\mathcal U_3H_{\lambda,\mu,\gamma}\mathcal U_3^{-1}
=
\int_{\T^3}^{\oplus}
H_{\lambda,\mu,\gamma}(K)\,\eta(dK).
\end{equation}
Equivalently, with \(\mathcal U=\mathcal U_3\mathcal F_3\),
\[
\mathcal U\widehat H_{\lambda,\mu,\gamma}\mathcal U^{-1}
=\int_{\T^3}^{\oplus}H_{\lambda,\mu,\gamma}(K)\,\eta(dK).
\]

On the full space \(\cH\), let \(H_\gamma^{\mathrm{free}}(K)\)
be multiplication by
\begin{equation}
\label{eq:free-three}
\E_\gamma(K;p,q)
=
\varepsilon(p)+\varepsilon(q)
+\gamma\varepsilon(K-p-q).
\end{equation}
On the full space \(\cH\), the interaction operators are
\begin{align}
(V_{13}f)(p,q)
&=\int_{\T^3}f(t,q)\,\eta(dt),
\label{eq:fiber-V13}\\
(V_{23}f)(p,q)
&=\int_{\T^3}f(p,t)\,\eta(dt),
\label{eq:fiber-V23}\\
(V_{12}f)(p,q)
&=
2\sum_{j=1}^3\int_{\T^3}
\cos(p_j-t_j)f(t,p+q-t)\,\eta(dt),
\label{eq:fiber-fermion-interaction}
\end{align}
with equalities understood almost everywhere.
The free multiplier is symmetric in \(p,q\). Moreover,
for the exchange operator \(P_{12}f(p,q)=f(q,p)\),
\[
P_{12}V_{12}P_{12}=V_{12},
\qquad
P_{12}V_{13}P_{12}=V_{23}.
\]
Hence \(\cH^{\mathrm a}\) reduces the free operator,
\(V_{12}\), and \(V_{13}+V_{23}\). The fibre Hamiltonian is
\begin{equation}
\label{eq:full-fiber-model-with-lambda}
H_{\lambda,\mu,\gamma}(K)
=
H_\gamma^{\mathrm{free}}(K)|_{\cH^{\mathrm a}}
+\lambda V_{12}|_{\cH^{\mathrm a}}
+\mu(V_{13}+V_{23})|_{\cH^{\mathrm a}}.
\end{equation}
Below, these restrictions are implicit when the ambient
space is \(\cH^{\mathrm a}\). The individual operators
\(V_{13}\) and \(V_{23}\) are defined on \(\cH\) and do
not generally preserve \(\cH^{\mathrm a}\).
The ranks one and three refer to the interactions in the
two-particle fibres. The corresponding three-particle interactions
have infinite rank because the spectator momentum remains a free
variable.

\compactedit{All pair fibres and the physical three-particle fibres are bounded
and self-adjoint for real couplings. For fixed couplings, the
three-particle family is continuous in \(K\) in operator norm.}

\section{Two-particle spectra and the essential spectrum}
\label{sec:mixed}
\compactedit{We study the mixed pair and its channels
(\S\ref{subsec:mixed-spectral-properties} and
\S\ref{subsec:critical-mixed-channels}), followed by the fermion
pair and its channel (\S\ref{subsec:fermion-interaction}).
For \(\gamma>1\), we classify the fermion-channel spectrum
relative to the free upper edge for every nonnegative coupling
and establish the quadratic estimates at the critical coupling.
Subsection~\ref{subsec:essential-threshold-and-gap} then identifies
the essential spectrum and internal gap. These results provide
the spectral and threshold estimates used in
Section~\ref{sec:upper-efimov}. Throughout this section,
\(\gamma>0\) is fixed.}

\subsection{Spectral properties of the mixed pair and its channels}
\label{subsec:mixed-spectral-properties}
We use the mixed-pair fibre \eqref{eq:mixed-pair-fibre-definition}.
We first describe its band geometry and spectrum and then transfer
these results to the mixed-channel fibres.

\subsubsection{Band geometry}
\label{subsec:mixed-band-geometry}
Set
\begin{equation}
\label{eq:r-gamma}
r_{\gamma,j}(k)
=
\sqrt{1+\gamma^2+2\gamma\cos k_j},
\qquad j=1,2,3.
\end{equation}
We call \(k\) regular if \(r_{\gamma,j}(k)>0\) for every \(j\).
For regular \(k\), define \(p_*(k)\in\T^3\) by
\[
\cos p_{*,j}(k)
=
\frac{1+\gamma\cos k_j}{r_{\gamma,j}(k)},
\qquad
\sin p_{*,j}(k)
=
\frac{\gamma\sin k_j}{r_{\gamma,j}(k)}.
\]
Then
\begin{equation}
\label{eq:coordinate-phase-representation}
\cos p_j+\gamma\cos(k_j-p_j)
=
r_{\gamma,j}(k)\cos(p_j-p_{*,j}(k)).
\end{equation}
Denote the band edges by
\begin{align}
m_\gamma(k)&=\min_{p\in\T^3}\E_\gamma(k;p),
\label{eq:mixed-band-lower-edge}\\
M_\gamma(k)&=\max_{p\in\T^3}\E_\gamma(k;p).
\label{eq:mixed-band-upper-edge}
\end{align}
\begin{lemma}[Band extrema and normal form]
\label{lem:mixed-band}
For regular \(k\),
\begin{align}
m_\gamma(k)
&=2\sum_{j=1}^3(1+\gamma-r_{\gamma,j}(k)),
\label{eq:mixed-min}\\
M_\gamma(k)
&=2\sum_{j=1}^3(1+\gamma+r_{\gamma,j}(k)).
\label{eq:mixed-max}
\end{align}
The extrema are nondegenerate and attained uniquely on the
torus at
\[
p^-(k)=p_*(k),
\qquad
p^+(k)=p_*(k)+\pi\mathbf1,
\]
respectively, with addition understood modulo \(2\pi\). Moreover,
\begin{equation}
\label{eq:mixed-upper-normal-form}
\begin{aligned}
\E_\gamma(k;p^-(k)+u)-m_\gamma(k)
&=M_\gamma(k)-\E_\gamma(k;p^+(k)+u)\\
&=2\sum_{j=1}^3r_{\gamma,j}(k)(1-\cos u_j).
\end{aligned}
\end{equation}
and
\begin{equation}
\label{eq:mixed-upper-Hessian}
-\frac12D_p^2\E_\gamma(k;p^+(k))
=
\operatorname{diag}
\bigl(r_{\gamma,1}(k),r_{\gamma,2}(k),r_{\gamma,3}(k)\bigr)>0.
\end{equation}
In local coordinates, the common energy difference is
\[
\sum_{j=1}^3r_{\gamma,j}(k)u_j^2+O(|u|^4),\qquad u\to0.
\]
\end{lemma}
\begin{proof}
By \eqref{eq:coordinate-phase-representation},
\[
\E_\gamma(k;p)
=
2\sum_{j=1}^3
\bigl[1+\gamma-r_{\gamma,j}(k)
\cos(p_j-p_{*,j}(k))\bigr].
\]
Since all \(r_{\gamma,j}(k)\) are positive, the extrema
and their uniqueness follow by setting each cosine equal
to \(1\) or \(-1\). Translation by \(p^+(k)\) gives
\eqref{eq:mixed-upper-normal-form}; differentiation gives
\eqref{eq:mixed-upper-Hessian}. The Hessian at the minimum
is \(2\operatorname{diag}(r_{\gamma,1},r_{\gamma,2},
r_{\gamma,3})>0\).
\end{proof}
The free operator is multiplication by the continuous
function \(\E_\gamma(k;\cdot)\), whose range is
\([m_\gamma(k),M_\gamma(k)]\). Since \(v\) has rank one,
Weyl's theorem yields
\[
\sigma_{\mathrm{ess}}(h_{\mu,\gamma}(k))
=
[m_\gamma(k),M_\gamma(k)].
\]
For \(\mu>0\), the inequality
\(h_{\mu,\gamma}(k)\ge m_\gamma(k)I\) excludes spectrum
below \(m_\gamma(k)\).
For regular \(k\), there are no eigenvalues in the band. Indeed,
the eigenvalue equation is
\((\E_\gamma(k;p)-e)u(p)=-\mu c\), where
\(c=\int u\,d\eta\). Each level set of the dispersion has measure
zero, since its equation has at most two solutions in \(p_1\) for
fixed \(p_2,p_3\). Thus \(c=0\) implies \(u=0\); otherwise
\(u=-\mu c/(\E_\gamma(k;\cdot)-e)\).
At a band edge this behaves as \(|p-p^\pm(k)|^{-2}\), which is
not locally square integrable. For an interior energy, choose
\(t\in(0,\pi)\) with \(\E_\gamma(k;p_*(k)+t\mathbf1)=e\).
The derivative \(\partial_{p_1}\E_\gamma=2r_{\gamma,1}(k)\sin t\)
is nonzero there. Using \(s=\E_\gamma(k;p)-e\) as one local
coordinate, the squared reciprocal has the divergent integral
\(\int_{-a}^a|s|^{-2}\,ds\). Consequently,
\begin{equation}
\label{eq:no-lower-mixed-eigenvalues}
\sigma_{\mathrm p}(h_{\mu,\gamma}(k))\cap(-\infty,M_\gamma(k)]
=\varnothing.
\end{equation}

\subsubsection{Discrete spectrum and upper-threshold resonance}
\label{subsec:upper-critical-coupling}
For regular \(k\), define
\begin{equation}
\label{eq:J-gamma}
J_\gamma(k)
=
\int_{\T^3}
\frac{\eta(dp)}{M_\gamma(k)-\E_\gamma(k;p)}.
\end{equation}
By the normal form in Lemma~\ref{lem:mixed-band}, the integrand
is comparable to \(|u|^{-2}\) near the unique maximum and is
bounded away from it. Hence \(0<J_\gamma(k)<\infty\).
The upper critical coupling is
\begin{equation}
\label{eq:pair-critical}
\mu_{2,\mathrm c}^+(k,\gamma)=J_\gamma(k)^{-1}.
\end{equation}
For \(z>M_\gamma(k)\), set
\[
A_\gamma^+(k;z)
=
\int_{\T^3}\frac{\eta(dp)}{z-\E_\gamma(k;p)}.
\]
The eigenvalue equation gives the rank-one
Birman--Schwinger criterion
\begin{equation}
\label{eq:pair-eigenvalue-determinant-equivalence}
z\in\sigma_{\mathrm{disc}}(h_{\mu,\gamma}(k))
\quad\Longleftrightarrow\quad
\mu A_\gamma^+(k;z)=1,
\qquad z>M_\gamma(k).
\end{equation}
\begin{proposition}[Repulsive mixed-pair critical coupling]
\label{prop:mixed-critical}
Let \(k\) be regular and \(\mu>0\).
\begin{enumerate}
\item
If \(0<\mu<\mu_{2,\mathrm c}^+(k,\gamma)\), then
\(h_{\mu,\gamma}(k)\) has no discrete eigenvalues.
\item
If \(\mu=\mu_{2,\mathrm c}^+(k,\gamma)\), then
\(h_{\mu,\gamma}(k)\) has no discrete eigenvalues and
\(M_\gamma(k)\) is an upper-threshold resonance, with
generalised eigenfunction
\begin{equation}
\label{eq:upper-virtual-function}
\psi_k^+(p)
=
\frac{1}{M_\gamma(k)-\E_\gamma(k;p)}
\in L^1(\T^3)\setminus L^2(\T^3),
\end{equation}
defined almost everywhere.
\item
If \(\mu>\mu_{2,\mathrm c}^+(k,\gamma)\), then
\(h_{\mu,\gamma}(k)\) has exactly one discrete eigenvalue
\(e_{\mu,\gamma}^+(k)>M_\gamma(k)\), and it is simple.
\end{enumerate}
\end{proposition}
\begin{proof}
The function \(A_\gamma^+(k;z)\) decreases continuously and strictly
from \(J_\gamma(k)\) to zero as \(z\) runs from \(M_\gamma(k)\)
to infinity. Hence the scalar equation
\(\mu A_\gamma^+(k;z)=1\) has one solution exactly when
\(\mu J_\gamma(k)>1\). At a solution \(z_*\), every eigenfunction
is a multiple of \((z_*-\E_\gamma(k;\cdot))^{-1}\in L^2(\T^3)\):
the eigenvalue equation gives this formula, and its integral cannot
vanish for a nonzero eigenfunction. The eigenvalue is therefore
simple. Together with \eqref{eq:no-lower-mixed-eigenvalues}, this
proves the discrete-spectrum assertions.
At the critical coupling, \eqref{eq:mixed-upper-normal-form} gives
\(\psi_k^+(p^+(k)+u)\asymp|u|^{-2}\), so
\(\psi_k^+\in L^1\setminus L^2\). The defining formulas extend
\(h_{0,\gamma}(k)\) and \(v\) to bounded operators on \(L^1\), where
\[
\bigl(h_{\mu_{2,\mathrm c}^+(k,\gamma),\gamma}(k)-M_\gamma(k)\bigr)
\psi_k^+
=\bigl[-1+\mu_{2,\mathrm c}^+(k,\gamma)J_\gamma(k)\bigr]\mathbf1=0.
\]
Thus the resonance equation holds almost everywhere and in the
distributional sense.
\end{proof}

\subsubsection{Mixed channels}
For fixed total quasi-momentum \(K\), define on the full space \(\cH\)
\[
H_{13,\mu}^{\mathrm{ch}}(K)=H_\gamma^{\mathrm{free}}(K)+\mu V_{13},
\qquad
H_{23,\mu}^{\mathrm{ch}}(K)=H_\gamma^{\mathrm{free}}(K)+\mu V_{23}.
\]
The exchange operator intertwines these channels:
\(P_{12}H_{13,\mu}^{\mathrm{ch}}(K)P_{12}=H_{23,\mu}^{\mathrm{ch}}(K)\).
Neither channel separately preserves \(\cH^{\mathrm a}\).
Under the unitary identification
\[
(\mathcal W_{13}f)(q;p)=f(p,q),
\qquad
\mathcal W_{13}:\cH\longrightarrow
\int_{\T^3}^{\oplus}L^2(\T^3,\eta)\,\eta(dq),
\]
the first channel has the decomposition
\begin{equation}
\label{eq:mixed-channel-direct-integral}
\mathcal W_{13}H_{13,\mu}^{\mathrm{ch}}(K)\mathcal W_{13}^{-1}
=\int_{\T^3}^{\oplus}H_{\mathrm{mix},\mu,\gamma}(K;q)\,\eta(dq),
\qquad
H_{\mathrm{mix},\mu,\gamma}(K;q)=\varepsilon(q)I+h_{\mu,\gamma}(K-q).
\end{equation}
The second channel has the same representation with the roles of
\(p\) and \(q\) interchanged. Each fibre spectrum is therefore
\[
\sigma(H_{\mathrm{mix},\mu,\gamma}(K;q))
=\varepsilon(q)+\sigma(h_{\mu,\gamma}(K-q)),
\]
and its essential spectrum is
\[
[\varepsilon(q)+m_\gamma(K-q),\,
  \varepsilon(q)+M_\gamma(K-q)].
\]
For \(\mu>0\) and regular \(K-q\), Proposition~\ref{prop:mixed-critical} gives
the complete off-band classification. The critical coupling is
\(\mu_{2,\mathrm c}^+(K-q,\gamma)\); at equality the upper
channel edge is a resonance, and above this coupling the unique
upper eigenvalue is \(\varepsilon(q)+e_{\mu,\gamma}^+(K-q)\).
Since the fibres depend continuously on \(q\) in operator norm,
their spectra determine the channel spectrum:
\[
\sigma(H_{13,\mu}^{\mathrm{ch}}(K))
=\sigma(H_{23,\mu}^{\mathrm{ch}}(K))
=\bigcup_{q\in\T^3}
\sigma(H_{\mathrm{mix},\mu,\gamma}(K;q)).
\]
The union is closed by norm continuity and compactness of the
spectator torus; see \cite[Theorem~XIII.85]{ReedSimon1978}.

\subsection{Critical tuning and threshold properties of the mixed channels}
\label{subsec:critical-mixed-channels}
\reviewedit{Fix \(\gamma>0\). For each \(K\) sufficiently close to
\(\pi\mathbf1\), we choose a mixed coupling independent of the
spectator momentum. At the spectator momentum corresponding to the
free maximum, both mixed-channel fibres have an upper-threshold
resonance at \(E_{+,\gamma}(K)\). We prove that this energy is also
the upper edge of each mixed-channel spectrum.}

\subsubsection{Free spectrum}
Define
\begin{align}
E_{-,\gamma}(K)
&=\min_{p,q\in\T^3}\E_\gamma(K;p,q),
\label{eq:free-three-lower-edge}\\
E_{+,\gamma}(K)
&=\max_{p,q\in\T^3}\E_\gamma(K;p,q).
\label{eq:free-three-upper-edge}
\end{align}
\begin{lemma}[Free spectrum on the antisymmetric spaces]
\label{lem:free-essential-spectrum}
For every \(K\in\T^3\), the restriction of
\(H_\gamma^{\mathrm{free}}(K)\) to \(\cH^{\mathrm a}\) satisfies
\[
\sigma(H_\gamma^{\mathrm{free}}(K)|_{\cH^{\mathrm a}})
=
\sigma_{\mathrm{ess}}(H_\gamma^{\mathrm{free}}(K)|_{\cH^{\mathrm a}})
=
[E_{-,\gamma}(K),E_{+,\gamma}(K)].
\]
On \(\widehat{\frH}^{\mathrm a}\),
\[
\sigma(\widehat H_\gamma^{\mathrm{free}})
=
\sigma_{\mathrm{ess}}(\widehat H_\gamma^{\mathrm{free}})
=
[0,12(2+\gamma)].
\]
\end{lemma}
\begin{proof}
Continuity and connectedness imply that the range of each
free multiplier is an interval. Every open neighbourhood
of a value in that range has a nonempty open preimage.
For the fibre multiplier, this preimage is invariant under
exchange of \(p,q\) and contains a point with \(p\ne q\).
Choose an open neighbourhood \(U\) of that point such that
\(U\) and its exchange image are disjoint and contained
in the preimage. Antisymmetric extensions of functions
in \(L^2(U)\) form an infinite-dimensional subspace of
the corresponding spectral subspace. Thus every value
in the range belongs to the essential spectrum of the
antisymmetric restriction. No spectrum lies outside
the range.
The same argument applies to
\(\varepsilon(p)+\varepsilon(q)+\gamma\varepsilon(r)\),
whose range is \([0,12(2+\gamma)]\).
Fourier equivalence gives the coordinate-space assertion.
\end{proof}

\subsubsection{Geometry of the free upper edge}
\label{subsec:upper-geometry}
Near \(\pi\mathbf1\), write
\(K=\pi\mathbf1+\widetilde K\) in local torus coordinates.
\begin{lemma}[Analytic upper-threshold point]
\label{lem:analytic-upper-threshold}
There exists an open neighbourhood \(\mathcal O_0^+\) of
\(\pi\mathbf1\) such that, for every \(K\in\mathcal O_0^+\),
the function \(\E_\gamma(K;\cdot,\cdot)\) has a
unique nondegenerate global maximum at \((k_K^+,k_K^+)\).
The map \(K\mapsto k_K^+\) is real analytic,
\(k_{\pi\mathbf1}^+=\pi\mathbf1\), and
\begin{equation}
\label{eq:upper-stationarity}
\sin k_{K,j}^+
=
\gamma\sin(K_j-2k_{K,j}^+),
\qquad j=1,2,3.
\end{equation}
Set
\[
\kappa_K^+=K-k_K^+,
\qquad
\rho_K^+=K-2k_K^+.
\]
\reviewedit{At the free maximum, \(k_K^+\) is the momentum of each
fermion, \(\rho_K^+\) is that of the third particle, and
\(\kappa_K^+\) is the total momentum of either mixed pair.
The free upper edge is therefore attained at spectator momentum
\(q=k_K^+\) in the mixed channels and \(r=\rho_K^+\) in the
fermion-pair channel.}
In local coordinates centred at \(K=\pi\mathbf1\), the following
expansions hold as \(\widetilde K\to0\):
\begin{align}
k_K^+
&=
\pi\mathbf1
+\frac{\gamma}{1+2\gamma}\widetilde K
+O(|\widetilde K|^3),
\label{eq:upper-k-expansion}
\\
\kappa_K^+
&=
\frac{1+\gamma}{1+2\gamma}\widetilde K
+O(|\widetilde K|^3),
\label{eq:upper-kappa-expansion}
\\
\rho_K^+
&\equiv
\pi\mathbf1
+\frac{1}{1+2\gamma}\widetilde K
+O(|\widetilde K|^3)
\pmod{2\pi\mathbb Z^3}.
\label{eq:upper-rho-expansion}
\end{align}
\end{lemma}
\begin{proof}
At \(K=\pi\mathbf1\),
\[
\E_\gamma(\pi\mathbf1;p,q)\le12(2+\gamma),
\]
with equality only at \(p=q=\pi\mathbf1\). Writing
\(p=\pi\mathbf1+x\), \(q=\pi\mathbf1+y\), we obtain
\[
E_{+,\gamma}(\pi\mathbf1)
-\E_\gamma(\pi\mathbf1;p,q)
=
|x|^2+|y|^2+\gamma|x+y|^2
+O((|x|+|y|)^4).
\]
The quadratic form is positive definite. The analytic
implicit function theorem therefore gives a unique nearby
critical point, which remains a nondegenerate maximum.
Exchange symmetry forces its components to coincide.
By compactness, the energy remains uniformly below the
maximum outside a fixed neighbourhood of
\((\pi\mathbf1,\pi\mathbf1)\). After shrinking
\(\mathcal O_0^+\), the continued maximum is consequently
the unique global maximum.
The critical-point equations give
\eqref{eq:upper-stationarity}. Write
\(k_{\pi\mathbf1+\widetilde K}^+
=\pi\mathbf1+a(\widetilde K)\). Then
\[
\sin a_j(\widetilde K)
=
\gamma\sin(\widetilde K_j-2a_j(\widetilde K)),
\qquad
Da(0)=\frac{\gamma}{1+2\gamma}I.
\]
Uniqueness implies \(a(-\widetilde K)=-a(\widetilde K)\);
hence the quadratic Taylor term vanishes.
The three expansions follow.
\end{proof}
For the spectator momentum \(q\), define
\[
\tau_K^+(q)
=
\varepsilon(q)+M_\gamma(K-q)
=
\max_{p\in\T^3}\E_\gamma(K;p,q).
\]
Thus \(\tau_K^+(q)\) is the upper edge of the free
mixed-channel fibre with spectator momentum \(q\).
\begin{proposition}[Free upper edge and mixed-pair maximum]
\label{prop:upper-edge-coincidence}
For \(K\in\mathcal O_0^+\),
\begin{equation}
\label{eq:upper-edge-coincidence}
E_{+,\gamma}(K)
=
\varepsilon(k_K^+)+M_\gamma(\kappa_K^+),
\end{equation}
and the unique maximiser of
\(\E_\gamma(\kappa_K^+;\cdot)\) satisfies
\begin{equation}
\label{eq:pair-maximizer-coincidence}
p^+(\kappa_K^+)=k_K^+.
\end{equation}
\end{proposition}
\begin{proof}
The unique maximizer \((k_K^+,k_K^+)\) of the full dispersion
maximizes first over \(p\) at fixed \(q=k_K^+\), and then over
\(q\). Hence \(\tau_K^+(k_K^+)=E_{+,\gamma}(K)\) and the unique
mixed-pair maximizer at \(\kappa_K^+\) is \(k_K^+\).
\end{proof}
Define
\begin{align}
B_K^+
&=
-\frac12\bigl[
D^2\varepsilon(k_K^+)
+\gamma D^2\varepsilon(\rho_K^+)
\bigr],
\label{eq:upper-B}\\
C_K^+
&=
-\frac{\gamma}{2}D^2\varepsilon(\rho_K^+).
\label{eq:upper-C}
\end{align}
Nondegeneracy of the maximum gives
\[
-\frac12D^2_{(p,q)}
\E_\gamma(K;k_K^+,k_K^+)
=
\begin{pmatrix}
B_K^+ & C_K^+\\
C_K^+ & B_K^+
\end{pmatrix}>0.
\]
The corresponding Schur complement
\begin{equation}
\label{eq:upper-R}
R_K^+
=
B_K^+-C_K^+(B_K^+)^{-1}C_K^+
\end{equation}
satisfies
\begin{equation}
\label{eq:upper-BR-positive}
B_K^+>0,
\qquad R_K^+>0.
\end{equation}
At the symmetric point,
\begin{equation}
\label{eq:upper-matrices-at-pi}
B_{\pi\mathbf1}^+=(1+\gamma)I,
\qquad
C_{\pi\mathbf1}^+=\gamma I,
\qquad
R_{\pi\mathbf1}^+
=\frac{1+2\gamma}{1+\gamma}I.
\end{equation}
Taylor expansion yields
\begin{equation}
\label{eq:upper-energy-expansion}
\begin{aligned}
&E_{+,\gamma}(K)
-\E_\gamma(K;k_K^++x,k_K^++y)\\
&\quad=
\langle B_K^+x,x\rangle
+2\langle C_K^+y,x\rangle
+\langle B_K^+y,y\rangle
+O((|x|+|y|)^3).
\end{aligned}
\end{equation}
The remainder is uniform for \(K\) in compact subsets
of \(\mathcal O_0^+\).
For \(q=k_K^++y\), the maximising momentum
is \(p=k_K^++x_K(y)\), where
\[
x_K(y)=-(B_K^+)^{-1}C_K^+y+O(|y|^2).
\]
Substitution into \eqref{eq:upper-energy-expansion} gives
\begin{equation}
\label{eq:upper-moving-threshold}
E_{+,\gamma}(K)-\tau_K^+(k_K^++y)
=
\langle R_K^+y,y\rangle+O(|y|^3),
\end{equation}
again uniformly on compact subsets of \(\mathcal O_0^+\).

\subsubsection{Critical coupling and simultaneous resonance}
\label{subsec:upper-critical}
Shrink \(\mathcal O_0^+\) so that \(\kappa_K^+\) remains
in a regular neighbourhood of \(0\), and define
\begin{equation}
\label{eq:three-body-upper-critical}
\mu_{\mathrm c}^+(K,\gamma)
=
\mu_{2,\mathrm c}^+(\kappa_K^+,\gamma)
=
\frac{1}{J_\gamma(\kappa_K^+)}.
\end{equation}
\begin{lemma}[Regularity of the critical coupling]
\label{lem:upper-critical-regularity}
The function \(K\mapsto\mu_{\mathrm c}^+(K,\gamma)\) is
positive and real analytic near \(\pi\mathbf1\), and
\begin{equation}
\label{eq:upper-critical-expansion}
\mu_{\mathrm c}^+(K,\gamma)
=
\mu_{\mathrm c}^+(\pi\mathbf1,\gamma)
+O(|\widetilde K|^2).
\end{equation}
\end{lemma}
\begin{proof}
By Lemma~\ref{lem:J-finite} (see
Appendix~\ref{app:technical-lemmas}), \(J_\gamma\) is positive
and real analytic near \(0\), with \(\nabla J_\gamma(0)=0\).
The assertions follow by composition with the analytic
map \(K\mapsto\kappa_K^+\) and
\eqref{eq:upper-kappa-expansion}.
\end{proof}
\begin{proposition}[Simultaneous mixed-channel resonance]
\label{prop:simultaneous-upper-resonance}
For \(K\in\mathcal O_0^+\), write
\[
k_+=k_K^+,\qquad
\kappa_+=\kappa_K^+,\qquad
\mu_{\mathrm c}^+=\mu_{\mathrm c}^+(K,\gamma).
\]
Then \(M_\gamma(\kappa_+)\) is an upper-threshold resonance
of \(h_{\mu_{\mathrm c}^+,\gamma}(\kappa_+)\), with generalised
eigenfunction
\[
\psi_{\kappa_+}^+(p)
=
\frac{1}{M_\gamma(\kappa_+)-\E_\gamma(\kappa_+;p)}
\in L^1(\T^3)\setminus L^2(\T^3),
\]
and
\[
\sigma(h_{\mu_{\mathrm c}^+,\gamma}(\kappa_+))
\cap(M_\gamma(\kappa_+),\infty)=\varnothing.
\]
Consequently,
\begin{equation}
\label{eq:simultaneous-edge-coincidence}
E_{+,\gamma}(K)
=
\varepsilon(k_+)+M_\gamma(\kappa_+)
\end{equation}
is an upper-threshold resonance energy of both mixed-channel
fibres with spectator momentum \(k_+\).
\end{proposition}
\begin{proof}
Apply Proposition~\ref{prop:mixed-critical} and
\eqref{eq:upper-edge-coincidence}. The two mixed-channel
fibres are unitarily equivalent under exchange of particles
\(1\) and \(2\).
\end{proof}

\subsubsection{Channel determinant asymptotics and transversality}
\label{subsec:upper-transversality}
The leading term of the channel determinant determines the singular
part of the Birman--Schwinger operator. The expansion below will later
be used to separate this part from a remainder that remains compact
at the threshold itself.
For \(z\ge E_{+,\gamma}(K)\), define
\begin{equation}
\label{eq:upper-channel-determinant}
\Delta_K^+(q,z)
=
1-\mu_{\mathrm c}^+(K,\gamma)
\int_{\T^3}
\frac{\eta(dp)}{z-\E_\gamma(K;p,q)}.
\end{equation}
At \(z=E_{+,\gamma}(K)\), the denominator can vanish only
when \(p=q=k_K^+\). The resulting singularity in the
\(p\)-integral is integrable in dimension three.
For regular \(k\), let
\[
B_\gamma^+(k)
=
-\frac12D_p^2\E_\gamma(k;p^+(k)).
\]
By \eqref{eq:pair-maximizer-coincidence},
\[
\begin{aligned}
B_\gamma^+(\kappa_K^+)
&=-\frac12D_p^2\E_\gamma(\kappa_K^+;k_K^+)\\
&=-\frac12\bigl[D^2\varepsilon(k_K^+)
+\gamma D^2\varepsilon(\rho_K^+)\bigr]
=B_K^+.
\end{aligned}
\]
Set
\begin{align}
\mathfrak c_K^+
&=\frac{1}{4\pi\sqrt{\det B_K^+}},
\label{eq:upper-cK}\\
\ell_K^+
&=-\nabla J_\gamma(\kappa_K^+).
\label{eq:upper-ellK}
\end{align}
\begin{theorem}[Critical upper-channel determinant asymptotics]
\label{thm:upper-determinant-asymptotics}
For \(K\in\mathcal O_0^+\), \(\delta\ge0\), and
\(\delta+|y|^2\to0\),
\begin{equation}
\label{eq:upper-determinant-asymptotic}
\begin{aligned}
&\Delta_K^+
\bigl(k_K^++y,E_{+,\gamma}(K)+\delta\bigr)\\
&\quad=
\mu_{\mathrm c}^+(K,\gamma)
\left[
\mathfrak c_K^+
\sqrt{\delta+\langle R_K^+y,y\rangle}
-\langle\ell_K^+,y\rangle
\right]
+O(\delta+|y|^2).
\end{aligned}
\end{equation}
The remainder estimate holds uniformly for \(K\) in
compact subsets of \(\mathcal O_0^+\), with a common
sufficiently small range of \(\delta+|y|^2\).
\end{theorem}
\begin{proof}
Put \(q=k_K^++y\), \(k=\kappa_K^+-y\), and
\[
s=\delta+E_{+,\gamma}(K)-\tau_K^+(k_K^++y).
\]
Then
\[
E_{+,\gamma}(K)+\delta-\E_\gamma(K;p,q)
=
M_\gamma(k)-\E_\gamma(k;p)+s,
\]
and \eqref{eq:upper-moving-threshold} gives
\[
s=\delta+\langle R_K^+y,y\rangle+O(|y|^3).
\]
Positive definiteness of \(R_K^+\) implies
\[
s\asymp\delta+|y|^2.
\]
Put \(a=\delta+\langle R_K^+y,y\rangle\).
For \(\delta+|y|^2>0\), rationalising the difference gives
\[
\bigl|\sqrt{s}-\sqrt a\bigr|
=\frac{|s-a|}{\sqrt{s}+\sqrt a}
\le\frac{C|y|^3}{\sqrt{\delta+|y|^2}}
\le C|y|^2.
\]
Hence \(\sqrt{s}=\sqrt a+O(|y|^2)\), also at
\(\delta=0\), \(y=0\), where both square roots vanish.
Moreover,
\[
\frac{1}{4\pi\sqrt{\det B_\gamma^+(k)}}
=
\mathfrak c_K^++O(|y|),
\]
and
\[
J_\gamma(k)
=
J_\gamma(\kappa_K^+)
+\langle\ell_K^+,y\rangle+O(|y|^2).
\]
Apply Lemma~\ref{lem:upper-square-root-law} (see
Appendix~\ref{app:technical-lemmas}) and use
\[
\mu_{\mathrm c}^+(K,\gamma)J_\gamma(\kappa_K^+)=1,
\qquad
|y|\sqrt{\delta+|y|^2}\le\delta+|y|^2.
\]
This proves the expansion. All estimates are uniform
on the stated compact parameter sets.
\end{proof}
\begin{definition}[Strict upper-threshold transversality]
\label{def:upper-transversality}
A point \(K\in\mathcal O_0^+\) satisfies strict
upper-threshold transversality if
\begin{equation}
\label{eq:upper-transversality}
\langle(R_K^+)^{-1}\ell_K^+,\ell_K^+\rangle
<
(\mathfrak c_K^+)^2.
\end{equation}
Set
\[
d_K^+(y)
=
\mathfrak c_K^+\sqrt{\langle R_K^+y,y\rangle}
-\langle\ell_K^+,y\rangle.
\]
\end{definition}
\begin{proposition}[Positivity criterion]
\label{prop:upper-positivity-criterion}
For fixed \(K\in\mathcal O_0^+\), one has
\(d_K^+(y)>0\) for every \(y\ne0\) if and only if
\eqref{eq:upper-transversality} holds.
Under this condition, there exist \(\rho_0>0\) and
\(0<C_1\le C_2<\infty\) such that
\begin{equation}
\label{eq:upper-determinant-two-sided}
C_1\sqrt{\delta+|y|^2}
\le
\Delta_K^+
\bigl(k_K^++y,E_{+,\gamma}(K)+\delta\bigr)
\le
C_2\sqrt{\delta+|y|^2}
\end{equation}
whenever
\begin{equation}
\label{eq:upper-determinant-small-region}
\delta\ge0,
\qquad
\delta+|y|^2\le\rho_0^2.
\end{equation}
The constants can be chosen uniformly on compact subsets
of the strict transversality region.
\end{proposition}
\begin{proof}
Set
\begin{equation}
\label{eq:upper-aK}
a_K=\sqrt{\langle(R_K^+)^{-1}\ell_K^+,\ell_K^+\rangle}.
\end{equation}
Cauchy--Schwarz gives
\begin{equation}
\label{eq:upper-R-Cauchy-Schwarz}
|\langle\ell_K^+,y\rangle|
\le a_K\sqrt{\langle R_K^+y,y\rangle}.
\end{equation}
Thus \(a_K<\mathfrak c_K^+\) implies strict positivity.
Conversely, if \(\ell_K^+\ne0\), evaluation at
\(y=(R_K^+)^{-1}\ell_K^+\) gives
\[
d_K^+(y)=a_K(\mathfrak c_K^+-a_K),
\]
so positivity requires \(a_K<\mathfrak c_K^+\).
For \(\ell_K^+=0\), this condition holds automatically.
For \(\delta\ge0\), the same estimate yields
\[
\begin{aligned}
(\mathfrak c_K^+-a_K)
\sqrt{\delta+\langle R_K^+y,y\rangle}
&\le
\mathfrak c_K^+\sqrt{\delta+\langle R_K^+y,y\rangle}
-\langle\ell_K^+,y\rangle\\
&\le
(\mathfrak c_K^++a_K)
\sqrt{\delta+\langle R_K^+y,y\rangle}.
\end{aligned}
\]
The leading term in \eqref{eq:upper-determinant-asymptotic}
is between \(bt\) and \(Bt\), where
\(t=\sqrt{\delta+|y|^2}\) and \(b,B>0\).
Its remainder satisfies \(|r|\leq Lt^2\).
Choose \(\rho_0\) so that \(L\rho_0\leq b/2\), and take
\(C_1=b/2\), \(C_2=B+b/2\).
For fixed \(\gamma\), these choices are uniform for \(K\)
in compact subsets of the strict transversality region.
\end{proof}
Since \(\kappa_{\pi\mathbf1}^+=0\) and
\(\nabla J_\gamma(0)=0\), one has
\(\ell_{\pi\mathbf1}^+=0\). Continuity therefore gives
the following consequence.
\begin{corollary}[Persistence of upper transversality]
\label{cor:upper-transversality-persistence}
There exists an open neighbourhood
\(\mathcal O_1^+\subset\mathcal O_0^+\) of \(\pi\mathbf1\)
on which \eqref{eq:upper-transversality} holds.
\end{corollary}

\subsubsection{Positivity away from the critical spectator}
\label{subsec:upper-isolation}
\reviewedit{At \(K=\pi\mathbf1\), comparison with
\(J_\gamma(0)\) gives positivity of the threshold determinant away
from the free maximum. The local lower bound near \(k_K^+\) and
compactness away from it extend this positivity to nearby \(K\).}
\begin{proposition}[Uniqueness of the critical spectator]
\label{prop:upper-unique-spectator}
After shrinking \(\mathcal O_1^+\), for every
\(K\in\mathcal O_1^+\),
\begin{equation}
\label{eq:upper-unique-zero}
\Delta_K^+(k_K^+,E_{+,\gamma}(K))=0,
\qquad
\Delta_K^+(q,E_{+,\gamma}(K))>0
\quad(q\ne k_K^+).
\end{equation}
\end{proposition}
\begin{proof}
First take \(K=\pi\mathbf1\), and write
\(p=\pi\mathbf1+P\), \(q=\pi\mathbf1+Q\). Then
\[
E_{+,\gamma}(\pi\mathbf1)
-\E_\gamma(\pi\mathbf1;p,q)
=
\varepsilon(P)+\varepsilon(Q)+\gamma\varepsilon(P+Q).
\]
Set
\[
\mathcal A(Q)
=
\int_{\T^3}
\frac{\eta(dP)}
{\varepsilon(P)+\varepsilon(Q)+\gamma\varepsilon(P+Q)}.
\]
\notationedit{Since \(M_\gamma(0)=12(1+\gamma)\) and
\(\varepsilon(P+\pi\mathbf1)=12-\varepsilon(P)\), the
measure-preserving substitution \(p=P+\pi\mathbf1\) in
\eqref{eq:J-gamma} gives
\[
J_\gamma(0)=\frac{1}{1+\gamma}
\int_{\T^3}\frac{\eta(dP)}{\varepsilon(P)}=\mathcal A(0).
\]}
Let
\[
x_j=1-\cos Q_j,\qquad
r_j=\sqrt{1+\gamma^2+2\gamma\cos Q_j},
\qquad
a_j=1+\gamma-r_j.
\]
A coordinatewise translation of \(P\) transforms the
denominator into
\[
2\sum_{j=1}^3
\bigl[x_j+a_j+r_j(1-\cos u_j)\bigr].
\]
If \(r_j=0\), the corresponding summand is independent
of \(u_j\). Since \(r_j\ge|1-\gamma|\),
\[
0\le a_j
=\frac{2\gamma x_j}{1+\gamma+r_j}
\le x_j.
\]
Consequently,
\[
x_j+a_j+r_j(1-\cos u_j)
-(1+\gamma)(1-\cos u_j)
=
x_j+a_j\cos u_j\ge0.
\]
For \(Q\ne0\), at least one \(x_j\) is positive, and
the inequality is strict on a set of positive measure.
For fixed \(Q\ne0\), the original denominator is bounded
below by \(\varepsilon(Q)>0\); hence
\(\mathcal A(Q)\le\varepsilon(Q)^{-1}<\infty\).
Taking reciprocals and integrating gives
\[
\mathcal A(Q)<\mathcal A(0)=J_\gamma(0),
\qquad Q\ne0.
\]
Since
\(\mu_{\mathrm c}^+(\pi\mathbf1,\gamma)J_\gamma(0)=1\),
this proves the assertion at \(K=\pi\mathbf1\).
Choose a neighbourhood \(\mathcal V\) of \(\pi\mathbf1\)
whose closure is contained in \(\mathcal O_1^+\).
The uniform form of
Proposition~\ref{prop:upper-positivity-criterion} yields
constants \(c,r>0\), independent of \(K\in\mathcal V\),
such that, in the local torus chart centred at \(k_K^+\),
\[
\Delta_K^+(k_K^++y,E_{+,\gamma}(K))
\ge c|y|,
\qquad 0<|y|<r.
\]
At \(K=\pi\mathbf1\), continuity and the strict inequality
already proved imply
\[
m_r:=
\min_{\operatorname{dist}_{\T^3}(q,\pi\mathbf1)\ge r/2}
\Delta_{\pi\mathbf1}^+
 (q,E_{+,\gamma}(\pi\mathbf1))>0.
\]
On this compact set the denominator in the defining integral
is bounded away from zero. After decreasing \(\mathcal V\),
the same is true uniformly for \(K\in\mathcal V\), and joint
continuity in \((K,q)\) gives
\[
\Delta_K^+(q,E_{+,\gamma}(K))\ge \frac{m_r}{2}
\]
whenever
\(\operatorname{dist}_{\T^3}(q,\pi\mathbf1)\ge r/2\).
Decrease \(\mathcal V\) once more so that
\(\operatorname{dist}_{\T^3}(k_K^+,\pi\mathbf1)<r/2\).
The local estimate and the last compact-set estimate then
cover all \(q\ne k_K^+\), proving the required positivity.
The equality at \(q=k_K^+\) follows from
\eqref{eq:three-body-upper-critical} and
\eqref{eq:upper-edge-coincidence}.
\end{proof}
\notationedit{For the remainder of this section, fix an open
neighbourhood \(\cO_+\) of \(\pi\mathbf1\) such that
\[
\overline{\cO_+}\subset\cO_1^+\subset\cO_0^+,
\]
where \(\cO_1^+\) is the neighbourhood retained in
Proposition~\ref{prop:upper-unique-spectator}. Write
\(\mu_{\mathrm c}^+=\mu_{\mathrm c}^+(K,\gamma)\).}
\begin{lemma}[Uniform positivity on compact energy intervals]
\label{lem:Delta-uniform-positivity}
For every \(K\in\cO_+\) and every compact interval
\([a,b]\subset(E_{+,\gamma}(K),\infty)\),
\[
\inf_{q\in\T^3,\ z\in[a,b]}\Delta_K^+(q,z)>0.
\]
\end{lemma}
\begin{proof}
Strict monotonicity in the energy and
Proposition~\ref{prop:upper-unique-spectator} give
\[
\Delta_K^+(q,z)\geq\Delta_K^+(q,a)>0,
\qquad q\in\T^3,\quad z\in[a,b].
\]
The continuous function \(\Delta_K^+(\cdot,a)\) has a positive
minimum on the compact torus.
\end{proof}
\begin{corollary}[Upper edge of the mixed-channel spectra]
\label{cor:mixed-channel-upper-edge}
For every \(K\in\cO_+\) and \(\alpha\in\{13,23\}\),
\[
\sigma\bigl(H_{\alpha,\mu_{\mathrm c}^+}^{\mathrm{ch}}(K)\bigr)
\cap(E_{+,\gamma}(K),\infty)=\varnothing,
\qquad
\sup\sigma\bigl(H_{\alpha,\mu_{\mathrm c}^+}^{\mathrm{ch}}(K)\bigr)
=E_{+,\gamma}(K).
\]
\end{corollary}
\begin{proof}
For each spectator \(q\), put
\(b_q(p)=E_{+,\gamma}(K)-\E_\gamma(K;p,q)\). By
Proposition~\ref{prop:upper-unique-spectator},
\(\mu_{\mathrm c}^+\int b_q^{-1}\,d\eta=1-\Delta_K^+(q,E_{+,\gamma}(K))\le1\).
Weighted Cauchy--Schwarz gives
\[
\mu_{\mathrm c}^+\left|\int u\,d\eta\right|^2
\le\mu_{\mathrm c}^+\left(\int b_q^{-1}\,d\eta\right)
       \left(\int b_q|u|^2\,d\eta\right)
\le\int b_q|u|^2\,d\eta.
\]
Thus \(H_{\mathrm{mix},\mu_{\mathrm c}^+,\gamma}(K;q)\le E_{+,\gamma}(K)I\).
The direct integral \eqref{eq:mixed-channel-direct-integral} gives
the same bound for \(H_{13,\mu_{\mathrm c}^+}^{\mathrm{ch}}(K)\), and exchange
symmetry gives it for the other mixed channel. Finally,
\(H_{\alpha,\mu_{\mathrm c}^+}^{\mathrm{ch}}(K)\ge H_\gamma^{\mathrm{free}}(K)\)
on \(\cH\), whose free upper edge is \(E_{+,\gamma}(K)\).
This proves equality of the upper edges.
\end{proof}

\subsection{Spectral properties of the fermion pair and its channel}
\label{subsec:fermion-interaction}
\compactedit{We study the fermion-pair upper eigenvalues and identify the
couplings at which its channel spectrum first extends above the
free three-particle edge and then separates from the free band.}

Recall that
\[
h_{12,\lambda}(k)=h_{12,0}(k)+\lambda v_{12}(k),
\qquad \lambda\geq0,
\]
acts in
\[
\frh^{\mathrm a}(k)
=\{u\in L^2(\T^3,\eta):u(p)=-u(k-p)\text{ for a.e. }p\},
\]
\reviewedit{Here \(h_{12,0}(k)\) is multiplication by
\(\mathcal E_{12}(k;p)=\varepsilon(p)+\varepsilon(k-p)\).}
The interaction is the orthogonal
projection onto the span of the orthonormal functions
\[
\phi_{j,k}(p)=\sqrt2\sin(p_j-k_j/2),\qquad j=1,2,3.
\]
A measurable choice of representatives of \(k\) is fixed in this
formula. Changing a representative changes only the signs of the
corresponding form factors and leaves the projection unchanged.

\subsubsection{Essential spectrum and upper eigenvalues}
For each \(k\in\T^3\), denote the free pair band edges by
\[
m_{12}(k):=\min_{p\in\T^3}\mathcal E_{12}(k;p),
\qquad
M_{12}(k):=\max_{p\in\T^3}\mathcal E_{12}(k;p).
\]
In the centred coordinates \(x=p-k/2\),
\[
\mathcal E_{12}(k;p)
=12-4\sum_{\ell=1}^3\cos(k_\ell/2)\cos x_\ell,
\]
and hence
\[
m_{12}(k)=12-4\sum_{\ell=1}^3|\cos(k_\ell/2)|,
\qquad
M_{12}(k)=12+4\sum_{\ell=1}^3|\cos(k_\ell/2)|.
\]
The spectrum of the free operator on \(\frh^{\mathrm a}(k)\)
is the entire interval \([m_{12}(k),M_{12}(k)]\), and every point
of this interval belongs to its essential spectrum. Indeed, a
localisation near a nonfixed point of \(p\mapsto k-p\), minus
its reflected localisation, gives an antisymmetric approximate
eigenfunction. The supports may be chosen to have measure tending
to zero, so that the normalised functions converge weakly to zero.
Approximating a fixed point by nonfixed points gives the same
conclusion at its dispersion value. If the dispersion is constant,
the operator is a scalar multiple of the identity on an
infinite-dimensional space.
Since \(v_{12}(k)\) has rank three, invariance of the essential
spectrum under finite-rank perturbations yields
\begin{equation}
\label{eq:fermion-pair-essential-spectrum}
\sigma_{\mathrm{ess}}(h_{12,\lambda}(k))
=\sigma_{\mathrm{ess}}(h_{12,0}(k))
=[m_{12}(k),M_{12}(k)]\subset[0,24].
\end{equation}
Positivity of the interaction excludes spectrum below \(m_{12}(k)\),
and its rank bounds the number of eigenvalues above \(M_{12}(k)\)
by three, counting multiplicities.
\begin{proposition}[Fermion-pair eigenvalues at strong coupling]
\label{prop:strong-fermion-pair}
Let \(\lambda>24\) and \(k\in\T^3\). Then
\(h_{12,\lambda}(k)\) has exactly three eigenvalues above
\(M_{12}(k)\), counted with multiplicities. Ordered as
\(e_{12,1}^+(k;\lambda)\leq e_{12,2}^+(k;\lambda)
\leq e_{12,3}^+(k;\lambda)\), they satisfy
\begin{equation}
\label{eq:strong-pair-eigenvalue-bounds}
\lambda+12\leq e_{12,j}^+(k;\lambda)\leq\lambda+24,
\qquad j=1,2,3,
\end{equation}
uniformly in \(k\in\T^3\).
\end{proposition}
\begin{proof}
Set \(P_k=v_{12}(k)\). Parity and orthogonality in centred
coordinates give
\[
P_kh_{12,0}(k)P_k=12P_k.
\]
Thus
\[
\langle h_{12,\lambda}(k)u,u\rangle
=(\lambda+12)\|u\|^2,\qquad u\in\operatorname{Ran}P_k,
\]
whereas
\[
\langle h_{12,\lambda}(k)u,u\rangle
\leq M_{12}(k)\|u\|^2,\qquad u\in\ker P_k.
\]
Since \(\dim\operatorname{Ran}P_k=3\) and
\(\lambda+12>24\geq M_{12}(k)\), the min--max principle gives
exactly three upper eigenvalues and the stated lower bound.
The upper bound follows from
\[
h_{12,\lambda}(k)\leq(\lambda+M_{12}(k))I
\leq(\lambda+24)I.
\]
\end{proof}
\reviewedit{Define the scalar functions}
\begin{equation}
\label{eq:fermion-scalar-BS-functions}
\mathfrak s_j(k;w)
=\int_{\T^3}
\frac{\sin^2x_j}
{w-12+4\sum_{\ell=1}^3\cos(k_\ell/2)\cos x_\ell}
\,\eta(dx),\qquad w>M_{12}(k).
\end{equation}
The free denominator is even in every centred coordinate.
Consequently, the Birman--Schwinger matrix in the basis
\(\{\phi_{j,k}\}_{j=1}^3\) is diagonal, with diagonal entries
\(2\lambda\mathfrak s_j(k;w)\). Each function
\(w\mapsto\mathfrak s_j(k;w)\) is strictly decreasing.
For \(\lambda>24\), the preceding proposition and the
Birman--Schwinger principle therefore imply that each equation
\begin{equation}
\label{eq:fermion-pair-scalar-roots}
1=2\lambda\mathfrak s_j(k;\nu_{12,j}^+(k;\lambda)),
\qquad j=1,2,3,
\end{equation}
has exactly one solution above \(M_{12}(k)\). The ordered
eigenvalues \(e_{12,j}^+\) are the increasing rearrangement of
these labelled roots.
The implicit function theorem gives continuity of the labelled
roots in each local momentum chart. Changing a lift of \(k_\ell\)
by \(2\pi\) changes the sign of \(\cos(k_\ell/2)\); the change
of variable \(x_\ell\mapsto x_\ell+\pi\) leaves the defining
integrals unchanged. Thus the roots are well defined and continuous
on \(\T^3\). Their increasing rearrangement is also continuous,
including at eigenvalue crossings.

\subsubsection{The fermion-pair channel}
Define the channel operator on \(\cH^{\mathrm a}\) by
\[
H_{12,\lambda}^{\mathrm{ch}}(K)
=H_{\lambda,0,\gamma}(K)
=H_\gamma^{\mathrm{free}}(K)+\lambda V_{12}.
\]
The spectator change of variables is implemented by the unitary map
\[
(\mathcal W_{12,K}f)(r;p)=f(p,K-r-p),
\qquad
\mathcal W_{12,K}:\cH^{\mathrm a}\longrightarrow
\int_{\T^3}^{\oplus}\frh^{\mathrm a}(K-r)\,\eta(dr).
\]
Its inverse is \((\mathcal W_{12,K}^{-1}g)(p,q)=g(K-p-q;p)\).
The fibre corresponding to the third-particle momentum \(r\) is
\begin{equation}
\label{eq:fermion-channel-fiber}
\mathcal H_{12}^{\mathrm{ch}}(K;r;\lambda,\gamma)
=\gamma\varepsilon(r)I+h_{12,\lambda}(K-r),
\end{equation}
and
\begin{equation}
\label{eq:fermion-channel-direct-integral}
\mathcal W_{12,K}H_{12,\lambda}^{\mathrm{ch}}(K)
\mathcal W_{12,K}^{-1}
=\int_{\T^3}^{\oplus}
\mathcal H_{12}^{\mathrm{ch}}(K;r;\lambda,\gamma)\,\eta(dr).
\end{equation}
After local centring, the fibres act on the fixed space of odd
functions and depend continuously on \(r\) in operator norm.
Consequently, the channel spectrum is the union of the fibre
spectra; see also \cite{ReedSimon1978}. To justify the use of the
ordinary union, a spectral point of one fibre can be approximated
by vectors supported in arbitrarily small spectator neighbourhoods
of positive measure. Conversely, if a real number belongs to no
fibre spectrum, local norm continuity and compactness of \(\T^3\)
give a uniform bound on the fibre resolvents.
The shifted free pair bands fill precisely
\([E_{-,\gamma}(K),E_{+,\gamma}(K)]\): their union is the range
of the continuous function \((r,p)\mapsto
\E_\gamma(K;p,K-r-p)\) on the connected torus \((\T^3)^2\).
Each such band remains in its interacting fibre spectrum.
Positivity excludes spectrum below the corresponding lower edge.
For \(\lambda>24\), define the continuous upper branches by
\[
Z_{12,j}^+(K;r;\lambda)
=\gamma\varepsilon(r)+e_{12,j}^+(K-r;\lambda),
\qquad j=1,2,3,
\]
and set
\[
\Sigma_{12}^+(K,\lambda)
=\bigcup_{j=1}^3
\{Z_{12,j}^+(K;r;\lambda):r\in\T^3\}.
\]
This is a compact set with lower endpoint
\begin{equation}
\label{eq:fermion-upper-channel-bottom}
\tau_{12}^+(K,\lambda)
=\min_{\substack{r\in\T^3\\1\leq j\leq3}}
Z_{12,j}^+(K;r;\lambda)\geq\lambda+12.
\end{equation}
The fibre spectral description therefore gives
\begin{equation}
\label{eq:fermion-channel-spectrum}
\sigma(H_{12,\lambda}^{\mathrm{ch}}(K))
=[E_{-,\gamma}(K),E_{+,\gamma}(K)]\cup\Sigma_{12}^+(K,\lambda),
\qquad\lambda>24.
\end{equation}

\subsubsection{Critical fermion coupling at the free upper edge}
Fix \(\gamma>1\). In the remainder of this subsection,
\(K\) belongs to a sufficiently small neighbourhood \(\cO_+\)
of \(K_0=\pi\mathbf1\). By Lemma~\ref{lem:analytic-upper-threshold},
the unique free maximum is \((k_K^+,k_K^+)\) and is nondegenerate.
Write
\[
E_+=E_{+,\gamma}(K),\qquad k_+=k_K^+,
\qquad \rho_+=K-2k_+.
\]
All reductions of \(\cO_+\) below are made for this fixed
\(\gamma\); constants in uniform estimates may depend on
\(\gamma\) and the chosen neighbourhood.
For \(r\in\T^3\) and \(\delta\geq0\), define
\begin{equation}
\label{eq:fermion-critical-b}
b_j(K,r,\delta)
=\int_{\T^3}
\frac{|\phi_{j,K-r}(p)|^2}
{E_++\delta-\E_\gamma(K;p,K-r-p)}\,\eta(dp),
\qquad j=1,2,3.
\end{equation}
The integrals at \(\delta=0\) are finite, as shown below.
For \(\delta>0\), they equal
\(2\mathfrak s_j(K-r;E_++\delta-\gamma\varepsilon(r))\).
Thus the channel Birman--Schwinger matrix is diagonal, with
associated scalar factors
\begin{equation}
\label{eq:fermion-critical-d}
d_{\lambda,j}(K,r,\delta)=1-\lambda b_j(K,r,\delta).
\end{equation}
Put
\begin{equation}
\label{eq:fermion-critical-coupling}
\beta(K,\gamma)=\max_{\substack{r\in\T^3\\1\leq j\leq3}}
b_j(K,r,0),
\qquad
\lambda_{\mathrm{crit}}(K,\gamma)=\beta(K,\gamma)^{-1}.
\end{equation}
This coupling describes the onset of channel spectrum above the
fixed three-particle edge \(E_+\). The existence of a bound state
in an individual pair fibre is a different threshold question.
\begin{lemma}[Regularity of the threshold integrals]
\label{lem:fermion-critical-regularity}
After reducing \(\cO_+\), the functions \(b_j(K,r,\delta)\)
are jointly continuous for \(\delta\geq0\). At \(\delta=0\)
they are twice continuously differentiable in \(r\), and these
derivatives are jointly continuous in \((K,r)\).
\end{lemma}
\begin{proof}
We first establish regularity at the only potentially singular
spectator momentum \(\rho_+\). Represent \(k=K-r\) near zero
and write \(p=k/2+\pi\mathbf1+u\). The denominator at
\(\delta=0\) becomes
\[
a(K,r,u)=s(K,r)+4\sum_{\ell=1}^3
c_\ell(K-r)(1-\cos u_\ell),
\]
where
\[
s(K,r)=E_+-\gamma\varepsilon(r)-M_{12}(K-r),
\qquad c_\ell(k)=\cos(k_\ell/2)>0.
\]
The numerator is \(2\sin^2u_j\). The function \(s\) is
nonnegative and vanishes only at \(\rho_+\). Its Hessian there
is positive definite: it is the Schur complement of the pair
block in the positive definite Hessian of the free energy
deficit. Thus, with \(t=r-\rho_+\), uniformly near \((t,u)=(0,0)\),
\[
a(K,r,u)\asymp |t|^2+|u|^2,
\qquad |\nabla_r a(K,r,u)|\leq C(|t|+|u|^2),
\qquad |D_r^2a(K,r,u)|\leq C.
\]
These estimates prove finiteness of the threshold integrals.
The second \(r\)-derivatives of the integrand are bounded by
\[
C\left\{
\frac{|u|^2}{(|t|^2+|u|^2)^2}
+\frac{|u|^2(|t|^2+|u|^4)}{(|t|^2+|u|^2)^3}
\right\}
\leq C(1+|u|^{-2}).
\]
The same integrable majorant controls the integrand and its
first \(r\)-derivatives. Dominated differentiation therefore
gives the claimed \(C^2\) regularity and joint continuity of
these derivatives. Outside this neighbourhood the denominator
is uniformly positive. The same domination, with \(a+\delta\)
in place of \(a\), proves joint continuity down to \(\delta=0\).
\end{proof}

\reviewedit{Lemmas~\ref{lem:fermion-critical-symmetric}
and~\ref{lem:fermion-critical-maxima} in
Appendix~\ref{app:technical-lemmas} show that each function
\(r\mapsto b_j(K,r,0)\) has a unique nondegenerate maximum
at \(r_j(K)\), with \(r_j(K_0)=K_0\). They also give the
uniform quadratic bounds \eqref{eq:fermion-critical-quadratic}
and~\eqref{eq:fermion-critical-bound} for the threshold integrals
and the critical multipliers.}

In particular, multiplication by
\(\operatorname{diag}(d_{\lambda_{\mathrm{crit}},j}(K,r,0))_{j=1}^3\)
on \(L^2(\T^3,\eta;\mathbb C^3)\) has no bounded inverse.
Although its zeros form a finite set, continuity makes its
smallest diagonal entry arbitrarily small on sets of positive
measure. For \(\delta>0\), the corresponding multiplication
operator is boundedly invertible.
\begin{proposition}[The channel spectrum for weak coupling]
\label{prop:fermion-weak-spectrum}
For \(K\in\cO_+\) and \(\lambda\geq0\),
\[
\sup\sigma(H_{12,\lambda}^{\mathrm{ch}}(K))=E_+
\quad\Longleftrightarrow\quad
0\leq\lambda\leq\lambda_{\mathrm{crit}}(K,\gamma).
\]
In this closed interval,
\[
\sigma(H_{12,\lambda}^{\mathrm{ch}}(K))
=[E_{-,\gamma}(K),E_+].
\]
At the symmetric point the critical coupling is independent of
\(\gamma>1\) and equals
\begin{equation}
\label{eq:fermion-critical-symmetric}
\lambda_{\mathrm{crit}}(K_0,\gamma)
=\left[2\int_{\T^3}
\frac{\sin^2u_1}{4\sum_{\ell=1}^3(1-\cos u_\ell)}
\,\eta(du)\right]^{-1}.
\end{equation}
\end{proposition}
\begin{proof}
Let \(0\leq\lambda\leq\lambda_{\mathrm{crit}}\) and
\(z=E_++\delta\), with \(\delta>0\). Then
\[
d_{\lambda,j}(K,r,\delta)
\geq d_{\lambda_{\mathrm{crit}},j}(K,r,\delta)>0,
\]
uniformly in \(r\). The free fibre resolvents have norm at
most \(\delta^{-1}\), and the finite-rank resolvent formula
therefore gives a uniform bound on the interacting fibre
resolvents. Hence \(z\) belongs to the channel resolvent set.
The channel contains the free band, and positivity excludes
spectrum below it, proving the asserted equality.
If \(\lambda>\lambda_{\mathrm{crit}}\), choose \(j\) and
\(r_j(K)\) for which \(b_j(K,r_j(K),0)=\beta\). Then
\(\lambda b_j(K,r_j(K),0)>1\). The function
\(\delta\mapsto b_j(K,r_j(K),\delta)\) is continuous,
strictly decreasing, and tends to zero as \(\delta\to\infty\).
It therefore takes the value \(\lambda^{-1}\) at a positive
\(\delta\). By the fibre Birman--Schwinger principle,
\(E_++\delta\) is a fibre eigenvalue, and local norm continuity
places it in the channel spectrum. Finally, at \(K_0\) the
maxima occur at \(r=K_0\); coordinate symmetry gives the
displayed formula for \(\lambda_{\mathrm{crit}}\).
\end{proof}

\subsubsection{Separation of the upper two-cluster branches}
The same fibre criterion gives, for every \(\lambda\geq0\),
\begin{equation}
\label{eq:fermion-channel-all-couplings}
\begin{split}
\sigma(H_{12,\lambda}^{\mathrm{ch}}(K))\cap(E_+,\infty)
=\{E_++\delta:\ &\delta>0,\\[-2pt]
&\lambda b_j(K,r,\delta)=1
\text{ for some }j,r\}.
\end{split}
\end{equation}
Indeed, a zero of a scalar factor gives a fibre eigenvalue.
If no factor vanishes at a fixed \(\delta>0\), compactness of
the spectator torus bounds their absolute values away from
zero, and the fibre resolvents are uniformly bounded.
Put
\begin{equation}
\label{eq:fermion-separation-coupling}
\begin{aligned}
\underline b_j(K,\gamma)&=\min_{r\in\T^3}b_j(K,r,0),
&\overline b_j(K,\gamma)&=\max_{r\in\T^3}b_j(K,r,0),\\
\lambda_{\mathrm{sep}}(K,\gamma)
&=\left[\min_{1\leq j\leq3}\underline b_j(K,\gamma)\right]^{-1}.
\end{aligned}
\end{equation}
All these extrema are positive and continuous in \(K\).
\begin{proposition}[Channel classification near the free upper edge]
\label{prop:fermion-all-couplings}
For each fixed \(\gamma>1\), after reducing \(\cO_+\), the
following holds for every \(K\in\cO_+\):
\[
0<\lambda_{\mathrm{crit}}(K,\gamma)
<\lambda_{\mathrm{sep}}(K,\gamma)
\le 12(1+\gamma),
\]
and moreover \(24<\lambda_{\mathrm{sep}}(K,\gamma)\).
\begin{enumerate}[(i)]
\item If \(0\leq\lambda\leq\lambda_{\mathrm{crit}}\), then
\[
\sigma(H_{12,\lambda}^{\mathrm{ch}}(K))
=[E_{-,\gamma}(K),E_+].
\]
\item If \(\lambda_{\mathrm{crit}}<\lambda\leq\lambda_{\mathrm{sep}}\),
there exists \(a>0\) such that
\[
[E_+,E_++a]\subset\sigma(H_{12,\lambda}^{\mathrm{ch}}(K)).
\]
\item If \(\lambda>\lambda_{\mathrm{sep}}\), all three upper
two-cluster branches lie strictly above \(E_+\), and
\[
E_+<\tau_{12}^+(K,\lambda),
\qquad
(E_+,\tau_{12}^+(K,\lambda))
\subset\rho(H_{12,\lambda}^{\mathrm{ch}}(K)).
\]
\end{enumerate}
At the symmetric point,
\begin{equation}
\label{eq:fermion-separation-symmetric}
\lambda_{\mathrm{sep}}(\pi\mathbf1,\gamma)=12(1+\gamma).
\end{equation}
\end{proposition}
\begin{proof}
We first estimate the separation coupling. Reduce \(\cO_+\)
so that \(E_+>12\gamma+12\), which is possible because
\(E_{+,\gamma}(K_0)=12\gamma+24\). With respect to the
probability measure \(|\phi_{j,K-r}(p)|^2\eta(dp)\),
the compression identity \(P_kh_{12,0}(k)P_k=12P_k\) gives
\[
\int_{\T^3}|\phi_{j,K-r}(p)|^2
\bigl[E_+-\E_\gamma(K;p,K-r-p)\bigr]\,\eta(dp)
=E_+-\gamma\varepsilon(r)-12>0.
\]
The Cauchy--Schwarz inequality, first for \(\delta>0\) and
then by passage to the limit, yields
\[
b_j(K,r,0)
\geq\frac{1}{E_+-\gamma\varepsilon(r)-12}
\geq\frac{1}{12(1+\gamma)}.
\]
Here we used \(E_+\leq12(2+\gamma)\) and
\(\varepsilon(r)\geq0\). Therefore
\(\lambda_{\mathrm{sep}}(K,\gamma)\leq12(1+\gamma)\).
At \(K=K_0\) and \(r=0\), the pair dispersion is identically
\(12\), so
\[
b_j(K_0,0,0)=\frac{1}{12(1+\gamma)},\qquad j=1,2,3.
\]
\reviewedit{Thus the lower bound is attained at \(r=0\), and}
\[
\lambda_{\mathrm{sep}}(K_0,\gamma)=12(1+\gamma).
\]
\revblock{Thus \(\lambda_{\mathrm{sep}}\) is determined by the
minimum of \(b_j(K_0,\cdot,0)\) at \(r=0\), whereas
\(\lambda_{\mathrm{crit}}\) is determined by its maximum at
\(r=K_0\).}
Since \(\gamma>1\), continuity gives
\(\lambda_{\mathrm{sep}}(K,\gamma)>24\) after a further
reduction of \(\cO_+\).
We next show that the transition intervals of the three modes
overlap. Coordinate permutations at \(K_0\) imply
\[
\overline b_j(K_0,\gamma)=\beta_0,
\qquad
\underline b_j(K_0,\gamma)=\alpha_0,
\qquad j=1,2,3,
\]
where \(0<\alpha_0<\beta_0\), because each threshold integral
has a unique maximum. By continuity, we may arrange that
\[
\max_j\overline b_j(K,\gamma)^{-1}
<\min_j\underline b_j(K,\gamma)^{-1},
\qquad K\in\cO_+.
\]
Hence the three intervals
\((\overline b_j^{-1},\underline b_j^{-1}]\)
have a common nonempty intersection. \rev{Consequently,}
\[
\bigcup_{j=1}^3(\overline b_j^{-1},\underline b_j^{-1}]
=\left(\min_j\overline b_j^{-1},
        \max_j\underline b_j^{-1}\right]
=(\lambda_{\mathrm{crit}},\lambda_{\mathrm{sep}}].
\]
This also proves \(0<\lambda_{\mathrm{crit}}<\lambda_{\mathrm{sep}}\).
Assertion \textup{(i)} is the weak-coupling result proved above.
For \textup{(ii)}, fix
\(\lambda\in(\lambda_{\mathrm{crit}},\lambda_{\mathrm{sep}}]\).
The interval identity supplies an index \(j\) such that
\[
\lambda\overline b_j>1,\qquad
\lambda\underline b_j\leq1.
\]
Let \(r_{\max}\) and \(r_{\min}\) attain these extrema.
Continuity at \(\delta=0\) gives \(a>0\) for which
\[
\lambda b_j(K,r_{\max},\delta)>1,
\qquad 0\leq\delta\leq a.
\]
Strict decrease in \(\delta\) gives
\[
\lambda b_j(K,r_{\min},\delta)
<\lambda b_j(K,r_{\min},0)\leq1,
\qquad \delta>0.
\]
For each \(\delta\in(0,a]\), continuity in \(r\) and
connectedness of \(\T^3\) therefore provide a momentum
\(r_\delta\) with \(\lambda b_j(K,r_\delta,\delta)=1\).
The fibre spectral criterion places \(E_++\delta\) in the
channel spectrum. Since \(E_+\) already belongs to the free
band, assertion \textup{(ii)} follows, including
\(\lambda=\lambda_{\mathrm{sep}}\).
Finally, let \(\lambda>\lambda_{\mathrm{sep}}\). Then
\[
\lambda b_j(K,r,0)\geq\lambda\underline b_j>1,
\qquad r\in\T^3,\quad j=1,2,3.
\]
For every \(r,j\), strict decrease from a value greater than
\(\lambda^{-1}\) to zero gives a unique positive solution
\(\delta_j(K,r,\lambda)\) of
\[
\lambda b_j(K,r,\delta_j(K,r,\lambda))=1.
\]
Since \(\partial_\delta b_j(K,r,\delta)<0\) for \(\delta>0\),
the implicit function theorem gives continuity of these roots
in \(r\). Compactness then implies
\[
\min_{\substack{r\in\T^3\\1\leq j\leq3}}
\delta_j(K,r,\lambda)>0.
\]
The energies \(E_++\delta_j(K,r,\lambda)\) are exactly the
three upper fibre eigenvalues, counting multiplicities.
Indeed, the three scalar equations give three eigenvalues,
and the rank-three interaction permits no additional upper
eigenvalues. Since \(\lambda>\lambda_{\mathrm{sep}}>24\),
the branches and their lower endpoint defined above apply, and
\[
\tau_{12}^+(K,\lambda)
=E_++\min_{\substack{r\in\T^3\\1\leq j\leq3}}
\delta_j(K,r,\lambda)>E_+.
\]
The channel spectrum is the free band together with these
branches, proving \textup{(iii)}.
\end{proof}

\subsubsection{Factorisation of the fermion interaction}
For use in the three-particle Birman--Schwinger reduction, define
\(\Phi_K:\cH^{\mathrm a}\to L^2(\T^3,\eta;\mathbb C^3)\) by
\begin{equation}
\label{eq:Phi-factor-map-recalled}
(\Phi_Kf)_j(r)
=\int_{\T^3}\phi_{j,K-r}(p)
f(p,K-r-p)\,\eta(dp),\qquad j=1,2,3.
\end{equation}
All integral identities are understood almost everywhere, with
the same measurable representatives in every occurrence of a
form factor. Bessel's inequality in each spectator fibre gives
\(\|\Phi_K\|\leq1\). A change of variables shows that
\[
(\Phi_K^*v)(p,q)
=\sum_{j=1}^3\phi_{j,p+q}(p)v_j(K-p-q).
\]
The right-hand side is antisymmetric in \(p,q\).
Orthonormality of the form factors yields \(\Phi_K\Phi_K^*=I\),
and hence \(\|\Phi_K\|=1\). Moreover,
\begin{equation}
\label{eq:V12-factorization}
\begin{split}
V_{12}&=\Phi_K^*\Phi_K,\\
(V_{12}f)(p,q)
&=\sum_{j=1}^3\phi_{j,p+q}(p)
\int_{\T^3}\phi_{j,p+q}(p')
f(p',p+q-p')\,\eta(dp').
\end{split}
\end{equation}
Thus \(V_{12}\) is an orthogonal projection of infinite rank
on \(\cH^{\mathrm a}\); the rank-three property holds for
each pair fibre \(v_{12}(k)\).

\subsection{Essential spectrum and opening of an internal gap}
\label{subsec:essential-threshold-and-gap}
\reviewedit{At the critical mixed coupling and for \(K\) near
\(\pi\mathbf1\), both mixed-channel spectra coincide with the
free band. The fermion-pair channel therefore determines the
essential spectrum of the full operator and the opening of its
internal gap.}

\begin{proposition}[Essential spectrum in terms of the channel spectra]
\label{prop:channel-essential-spectrum}
Let \(\gamma>0\) and \(\lambda,\mu\ge0\). For every \(K\in\T^3\),
\begin{equation}
\label{eq:essential-spectrum-channel-union}
\begin{aligned}
\sigma_{\mathrm{ess}}\bigl(H_{\lambda,\mu,\gamma}(K)\bigr)
={}&\sigma\bigl(H_{12,\lambda}^{\mathrm{ch}}(K)\bigr)\\
&\cup\sigma\bigl(H_{13,\mu}^{\mathrm{ch}}(K)\bigr)
\cup\sigma\bigl(H_{23,\mu}^{\mathrm{ch}}(K)\bigr).
\end{aligned}
\end{equation}
Here \(H_{12,\lambda}^{\mathrm{ch}}(K)\) acts on
\(\cH^{\mathrm a}\), whereas the mixed channels act on \(\cH\).
The two mixed-channel spectra coincide by exchange symmetry.
\end{proposition}
\begin{proof}
The formula follows from the two inclusions proved in
Lemmas~\ref{lem:ess-inclusion} and~\ref{lem:ess-reverse} below.
For the general lattice channel description, see \cite{ALM2007}.
\end{proof}
For the next two lemmas, fix \(\gamma>0\), \(\lambda,\mu\geq0\)
and \(K\in\T^3\), and set
\[
H=H_{\lambda,\mu,\gamma}(K),\qquad
H_0=H_\gamma^{\mathrm{free}}(K)|_{\cH^{\mathrm a}},\qquad
I_0=[E_{-,\gamma}(K),E_{+,\gamma}(K)].
\]
Let \(\Sigma\) denote the channel union on the right-hand side of
\eqref{eq:essential-spectrum-channel-union}. In each pair fibre,
the interaction has finite rank and preserves the free essential
spectrum. The shifted free pair bands fill \(I_0\). Thus the
direct-integral representations and local norm continuity imply
that \(I_0\) is contained in each channel spectrum; see
\cite[Theorem~XIII.85]{ReedSimon1978}.
\begin{lemma}[Essential spectrum: inclusion]
\label{lem:ess-inclusion}
\(\sigma_{\mathrm{ess}}(H)\subset\Sigma\).
\end{lemma}
\begin{proof}
Recall the factorisation \(V_{12}=\Phi_K^*\Phi_K\) on
\(\cH^{\mathrm a}\). Define
\[
\frh=L^2(\T^3,\eta),\qquad
\Psi:\cH^{\mathrm a}\longrightarrow\frh,\qquad
(\Psi f)(q)=\int_{\T^3}f(p,q)\,\eta(dp).
\]
Its adjoint and the mixed-interaction factorisation are
\begin{equation}
\label{eq:mixed-interaction-factorization}
(\Psi^*g)(p,q)=\frac12[g(q)-g(p)],
\qquad
V_{13}+V_{23}=2\Psi^*\Psi
\quad\text{on }\cH^{\mathrm a}.
\end{equation}
Set
\[
\mathcal Lf=
\bigl(\sqrt\lambda\,\Phi_Kf,\sqrt{2\mu}\,\Psi f\bigr),
\qquad
H=H_0+\mathcal L^*\mathcal L.
\]
For real \(z\notin\Sigma\), put \(R_0(z)=(H_0-z)^{-1}\).
This resolvent exists because \(I_0\subset\Sigma\).
On \(L^2(\T^3;\C^3)\oplus\frh\),
\[
I+\mathcal L R_0(z)\mathcal L^*
=
\begin{pmatrix}
D_{12}(z)&0\\
0&M_{d_z}
\end{pmatrix}
+\mathcal K_z,
\]
where
\[
D_{12}(z)=I+\lambda\Phi_KR_0(z)\Phi_K^*,
\qquad
d_z(q)=1+\mu\int_{\T^3}
\frac{\eta(dp)}{\E_\gamma(K;p,q)-z},
\]
and \(M_{d_z}\) denotes multiplication by \(d_z\).
The operator \(\mathcal K_z\) is compact. Indeed, the integral part
of the mixed diagonal block has kernel
\[
-\frac{\mu}{\E_\gamma(K;p,q)-z},
\]
and the off-diagonal blocks have kernels of the same type with
bounded pair form factors in the numerator. Since
\(\operatorname{dist}(z,I_0)>0\), these kernels are bounded and
therefore Hilbert--Schmidt.
The operator \(D_{12}(z)\) is invertible because
\(z\notin\sigma(H_{12,\lambda}^{\mathrm{ch}}(K))\) and \(z\notin I_0\).
The rank-one resolvent formula for the mixed fibres gives
\(d_z(q)\ne0\) for every \(q\), since
\(z\notin\sigma(H_{13,\mu}^{\mathrm{ch}}(K))\).
Continuity on the compact torus yields
\(\min_q|d_z(q)|>0\). Thus \(M_{d_z}\) is invertible.
Since \(D_{12}(z)\) and \(M_{d_z}\) are invertible and
\(\mathcal K_z\) is compact, the operator
\(I+\mathcal L R_0(z)\mathcal L^*\) is Fredholm.
The Fredholm equivalence of \(I+XY\) and \(I+YX\), together with
\[
H-z=(H_0-z)\bigl[I+R_0(z)\mathcal L^*\mathcal L\bigr],
\]
shows that \(H-z\) is Fredholm. Hence
\(z\notin\sigma_{\mathrm{ess}}(H)\).
\end{proof}
\begin{lemma}[Essential spectrum: reverse inclusion]
\label{lem:ess-reverse}
\(\Sigma\subset\sigma_{\mathrm{ess}}(H)\).
\end{lemma}
\begin{proof}
We construct singular Weyl sequences.
First let \(z_0\in I_0\), and choose \((p_0,q_0)\) with
\(\E_\gamma(K;p_0,q_0)=z_0\).
Choose normalized functions \(a_m,b_m\) with disjoint supports
shrinking to \(p_0,q_0\), respectively. When \(p_0=q_0\), take both
supports in a shrinking neighbourhood of this point, still disjoint.
Then
\[
f_m(p,q)=\frac{a_m(p)b_m(q)-b_m(p)a_m(q)}{\sqrt2}
\]
belongs to \(\cH^{\mathrm a}\), has norm one, and converges weakly
to zero. Continuity of the dispersion gives
\(\|(H_0-z_0)f_m\|\to0\). If \(A_m\) and \(B_m\) are the
supports of \(a_m\) and \(b_m\), respectively, Cauchy--Schwarz
in the factor maps and boundedness of the form factors yield
\[
\|\Psi f_m\|^2+\|\Phi_Kf_m\|^2
\leq C\bigl[\eta(A_m)+\eta(B_m)\bigr]\longrightarrow0.
\]
Consequently, \(\|(H-z_0)f_m\|\to0\), proving
\(I_0\subset\sigma_{\mathrm{ess}}(H)\).
Next let \(z_0\in\sigma(H_{12,\lambda}^{\mathrm{ch}}(K))\setminus I_0\).
By the direct-integral representation, there is \(r_0\in\T^3\)
such that \(z_0\) is an eigenvalue of
\(\gamma\varepsilon(r_0)I+h_{12,\lambda}(K-r_0)\).
Let \(u_0\) be a normalized eigenfunction. Using a local lift of
\(K-r\), transport \(u_0\) by the centred-coordinate translation to
normalized functions \(u_r\in\frh^{\mathrm a}(K-r)\).
Norm continuity of the centred fibres implies
\[
\|[\gamma\varepsilon(r)I+h_{12,\lambda}(K-r)-z_0]u_r\|
\longrightarrow0
\quad\text{as }r\to r_0.
\]
Choose \(a_m\) of norm one supported in balls \(B_m\) shrinking to
\(r_0\), and set
\[
f_m(p,q)=a_m(K-p-q)\,u_{K-p-q}(p),
\]
extending the product by zero outside \(K-p-q\in B_m\).
These functions are antisymmetric, normalized, and weakly null.
The preceding continuity gives
\[
\|(H_{12,\lambda}^{\mathrm{ch}}(K)-z_0)f_m\|
\longrightarrow0.
\]
For each fixed \(q\), the function \(p\mapsto f_m(p,q)\)
vanishes outside \(K-q-B_m\). By Cauchy--Schwarz and
translation invariance of Haar measure,
\[
\begin{aligned}
|\Psi f_m(q)|^2
&=
\left|\int_{K-q-B_m}f_m(p,q)\,\eta(dp)\right|^2\\
&\le \eta(B_m)\int_{\T^3}|f_m(p,q)|^2\,\eta(dp).
\end{aligned}
\]
Integrating in \(q\) and using \(\|f_m\|=1\), we obtain
\[
\|\Psi f_m\|^2\le \eta(B_m)\|f_m\|^2
=\eta(B_m)\longrightarrow0.
\]
Since
\[
H=H_{12,\lambda}^{\mathrm{ch}}(K)+2\mu\Psi^*\Psi
\quad\text{on }\cH^{\mathrm a},
\]
it follows that
\[
\|(H-z_0)f_m\|
\le \|(H_{12,\lambda}^{\mathrm{ch}}(K)-z_0)f_m\|
+2\mu\|\Psi^*\|\,\|\Psi f_m\|
\longrightarrow0.
\]
Thus \((f_m)\) is a singular Weyl sequence for \(H\) at \(z_0\).
Finally, let
\(z_0\in\sigma(H_{13,\mu}^{\mathrm{ch}}(K))\setminus I_0\).
There are \(q_0\in\T^3\) and a normalized \(u_0\in\frh\) such that
\[
[\varepsilon(q_0)I+h_{\mu,\gamma}(K-q_0)]u_0=z_0u_0.
\]
For normalized \(a_m\) supported in balls \(B_m\) shrinking to
\(q_0\), put \(g_m(p,q)=u_0(p)a_m(q)\).
Norm continuity of the mixed fibres gives
\[
\|(H_{13,\mu}^{\mathrm{ch}}(K)-z_0)g_m\|\longrightarrow0.
\]
Let \((P_{12}g)(p,q)=g(q,p)\) and
\(F_m=g_m-P_{12}g_m\in\cH^{\mathrm a}\). The explicit kernels,
Cauchy--Schwarz and boundedness of the form factors give
\[
\|V_{23}g_m\|^2+\|\Phi_KF_m\|^2
\leq C\eta(B_m)\longrightarrow0.
\]
Indeed,
\[
\|V_{23}g_m\|^2
\le C\left|\int a_m\,d\eta\right|^2
\le C\eta(B_m),
\]
and similarly for \(\|\Phi_KF_m\|^2\). For the second term,
the spectator-fibre integral is supported in \(K-r-B_m\) or
\(B_m\), each of measure \(\eta(B_m)\).
The operator $H_\gamma^{\mathrm{free}}(K)+\mu(V_{13}+V_{23})$
on $\cH$ commutes with $P_{12}$. Using
$F_m=(I-P_{12})g_m$ and $V_{12}=\Phi_K^*\Phi_K$ on
$\cH^{\mathrm a}$, we obtain
\[
\begin{split}
(H-z_0)F_m={}&(I-P_{12})
\bigl[(H_{13,\mu}^{\mathrm{ch}}(K)-z_0)g_m+\mu V_{23}g_m\bigr]
\\&+\lambda\Phi_K^*\Phi_KF_m.
\end{split}
\]
Thus \(\|(H-z_0)F_m\|\to0\). Moreover,
\[
\|g_m-P_{12}g_m\|^2
=2-2|\langle u_0,a_m\rangle|^2\longrightarrow2,
\]
because \(|\langle u_0,a_m\rangle|^2
\le\int_{B_m}|u_0|^2\,d\eta\to0\).
Therefore
\[
f_m=\frac{g_m-P_{12}g_m}{\|g_m-P_{12}g_m\|}
\]
is a normalized antisymmetric singular Weyl sequence for \(H\)
at \(z_0\). Its weak convergence to zero follows from the shrinking
supports in at least one momentum variable.
The \((23)+1\) channel has the same spectrum by exchange symmetry.
Thus \(\Sigma\subset\sigma_{\mathrm{ess}}(H)\).
\end{proof}
We now apply this formula at the critical mixed coupling.
For fixed \(\gamma>0\), retain a neighbourhood \(\cO_+\) on
which the mixed-channel results hold. When \(\gamma>1\), also
impose all reductions made in Proposition~\ref{prop:fermion-all-couplings}.
Write
\[
\mu_{\mathrm c}^+=\mu_{\mathrm c}^+(K,\gamma),
\qquad E_\pm=E_{\pm,\gamma}(K).
\]
\begin{definition}[Upper internal gap]
\label{def:upper-gap}
For \(\gamma>1\), \(K\in\cO_+\) and
\(\lambda>\lambda_{\mathrm{sep}}(K,\gamma)\), set
\[
\cG_\lambda^+(K)=(E_{+,\gamma}(K),\tau_{12}^+(K,\lambda)).
\]
This notation refers to a gap in the essential spectrum;
the full operator may have discrete eigenvalues in this interval.
\end{definition}
\begin{lemma}[Critical essential spectrum at the tuned coupling]
\label{lem:critical-essential-spectrum}
Fix \(\gamma>0\), \(K\in\cO_+\), \(\mu=\mu_{\mathrm c}^+\). Then for every \(\lambda\ge0\),
\[
\sigma_{\mathrm{ess}}\bigl(H_{\lambda,\mu_{\mathrm c}^+,\gamma}(K)\bigr)
=
\sigma\bigl(H_{12,\lambda}^{\mathrm{ch}}(K)\bigr).
\]
\end{lemma}
\begin{proof}
By Corollary~\ref{cor:mixed-channel-upper-edge}, neither mixed
channel has spectrum above \(E_+\). Since \(\mu_{\mathrm c}^+>0\),
\[
H_{\alpha,\mu_{\mathrm c}^+}^{\mathrm{ch}}(K)
\ge H_\gamma^{\mathrm{free}}(K)\ge E_-I,
\qquad\alpha\in\{13,23\},
\]
on \(\cH\). \reviewedit{Each channel spectrum contains the free
band \([E_-,E_+]\). Consequently,}
\[
\sigma\bigl(H_{13,\mu_{\mathrm c}^+}^{\mathrm{ch}}(K)\bigr)
=\sigma\bigl(H_{23,\mu_{\mathrm c}^+}^{\mathrm{ch}}(K)\bigr)
=[E_-,E_+].
\]
Since this interval is also contained in
\(\sigma(H_{12,\lambda}^{\mathrm{ch}}(K))\), the claim follows from
Proposition~\ref{prop:channel-essential-spectrum}.
\end{proof}
\begin{corollary}[Weak-coupling regime]
\label{cor:weak-regime}
If \(\lambda=0\), or if \(\gamma>1\) and
\(0\le\lambda\le\lambda_{\mathrm{crit}}(K,\gamma)\), then
\[
\sigma_{\mathrm{ess}}\bigl(H_{\lambda,\mu_{\mathrm c}^+,\gamma}(K)\bigr)
=[E_-,E_+],
\]
and hence
\[
\sup\sigma_{\mathrm{ess}}\bigl(H_{\lambda,\mu_{\mathrm c}^+,\gamma}(K)\bigr)=E_+.
\]
\end{corollary}
\begin{proof}
By Lemma~\ref{lem:critical-essential-spectrum},
\(\sigma_{\mathrm{ess}}(H_{\lambda,\mu_{\mathrm c}^+,\gamma}(K))
=\sigma(H_{12,\lambda}^{\mathrm{ch}}(K))\).
For \(\lambda=0\), this is the free spectrum, so the claim follows
from Lemma~\ref{lem:free-essential-spectrum}.
For \(\gamma>1\), Proposition~\ref{prop:fermion-weak-spectrum}
gives \(\sigma(H_{12,\lambda}^{\mathrm{ch}}(K))=[E_-,E_+]\) on the
entire closed weak-coupling interval.
\end{proof}
\begin{corollary}[Separated regime and internal gap]
\label{cor:separated-regime}
If \(\gamma>1\) and \(\lambda>\lambda_{\mathrm{sep}}(K,\gamma)\), then
\(\cG_\lambda^+(K)\) is a nonempty internal gap in the essential
spectrum of \(H_{\lambda,\mu_{\mathrm c}^+,\gamma}(K)\), and both of its
endpoints belong to the essential spectrum. At \(K=\pi\mathbf1\),
\[
\cG_\lambda^+(\pi\mathbf1)=(24+12\gamma,\lambda+12).
\]
\end{corollary}
\begin{proof}
By Lemma~\ref{lem:critical-essential-spectrum} and
Proposition~\ref{prop:fermion-all-couplings}(iii),
\[
(E_+,\tau_{12}^+(K,\lambda))
\subset\rho(H_{12,\lambda}^{\mathrm{ch}}(K)).
\]
The lower endpoint \(E_+\) belongs to the free band, and the upper
endpoint \(\tau_{12}^+(K,\lambda)\) belongs to the compact set
\(\Sigma_{12}^+(K,\lambda)\), as its minimum. Both endpoints
therefore belong to the essential spectrum.
At \(K=\pi\mathbf1\), the lower bound
\(\tau_{12}^+(K,\lambda)\ge\lambda+12\) is attained at \(r=0\),
because \(\mathcal E_{12}(\pi\mathbf1;p)\equiv12\) and the three
pair eigenvalues equal \(\lambda+12\).
Also \(E_{+,\gamma}(\pi\mathbf1)=24+12\gamma\).
This proves the stated form of the symmetric gap.
\end{proof}
\begin{corollary}[Intervening regime]
\label{cor:intervening-regime}
If \(\gamma>1\) and
\(\lambda_{\mathrm{crit}}(K,\gamma)<\lambda\le\lambda_{\mathrm{sep}}(K,\gamma)\),
then \([E_+,E_++a]\) is contained in the essential spectrum for
some \(a>0\). In particular, discrete eigenvalues cannot accumulate
at \(E_+\) from above in this regime.
\end{corollary}
\begin{proof}
By Lemma~\ref{lem:critical-essential-spectrum} and
Proposition~\ref{prop:fermion-all-couplings}(ii), there is \(a>0\)
such that \([E_+,E_++a]\subset\sigma(H_{12,\lambda}^{\mathrm{ch}}(K))
=\sigma_{\mathrm{ess}}(H_{\lambda,\mu_{\mathrm c}^+,\gamma}(K))\).
Since the essential spectrum contains an interval immediately to
the right of \(E_+\), no discrete eigenvalues can accumulate there.
\end{proof}
\begin{theorem}[Critical essential spectrum and its internal gap]
\label{thm:upper-essential-edge}
Fix \(\gamma>0\). There is a neighbourhood \(\cO_+\) of
\(\pi\mathbf1\) such that, for every \(K\in\cO_+\), at
\(\mu_{\mathrm c}^+=\mu_{\mathrm c}^+(K,\gamma)\),
\[
\sigma_{\mathrm{ess}}\bigl(H_{\lambda,\mu_{\mathrm c}^+,\gamma}(K)\bigr)
=\sigma(H_{12,\lambda}^{\mathrm{ch}}(K)),\qquad\lambda\ge0.
\]
For \(\lambda=0\), this spectrum is the free band.
If \(\gamma>1\), the neighbourhood can be chosen so that,
writing \(E_\pm=E_{\pm,\gamma}(K)\), the following alternatives hold:
\begin{enumerate}[(i)]
\item For \(0\leq\lambda\leq\lambda_{\mathrm{crit}}(K,\gamma)\),
the essential spectrum is \([E_-,E_+]\).
\item For \(\lambda_{\mathrm{crit}}(K,\gamma)<\lambda
\leq\lambda_{\mathrm{sep}}(K,\gamma)\), it contains
\([E_+,E_++a]\) for some \(a>0\); consequently, discrete
eigenvalues cannot accumulate at \(E_+\) from above.
\item For \(\lambda>\lambda_{\mathrm{sep}}(K,\gamma)\),
\(\cG_\lambda^+(K)=(E_+,\tau_{12}^+(K,\lambda))\) is a nonempty
gap in the essential spectrum, and both endpoints belong to the
essential spectrum.
At \(K=\pi\mathbf1\), this gap is
\((24+12\gamma,\lambda+12)\).
\end{enumerate}
\end{theorem}
\begin{proof}
Immediate from Lemma~\ref{lem:critical-essential-spectrum} and
Corollaries~\ref{cor:weak-regime}--\ref{cor:intervening-regime}.
\end{proof}
\section{Upper-threshold Efimov asymptotics and stability}
\label{sec:upper-efimov}
Under the critical mixed-channel tuning and the assumptions on the
mass ratio and total quasi-momentum stated below, we establish
logarithmic eigenvalue-counting asymptotics in the weak and separated
regimes. These asymptotics determine, in the logarithmic sense, the
exponential rate at which the discrete eigenvalues approach the free
upper edge from above. In the weak regime, the eigenvalues lie above
the essential spectrum; in the separated regime, they lie in the
internal gap whose lower endpoint is the free upper edge.

We use Sobolev's method~\cite{Sobolev1993}, adapted to the lattice
setting as in~\cite{ALM2004,ALD2007,AlbeverioLakaevKhalkhujaev2012}.
For the related continuum fermionic problem, see
Basti and Teta~\cite{BastiTeta2017}.

\subsection{Local threshold geometry and two-sided estimates}
\label{subsec:efimov-local-estimates}
Fix \(\gamma>0\) and \(K\in\cO_+\), and write
\(E_+=E_{+,\gamma}(K)\), \(k_+=k_K^+\).
\reviewedit{The free-energy expansion
\eqref{eq:upper-energy-expansion} and the determinant estimate
\eqref{eq:upper-determinant-two-sided} give constants
\(c,C,\rho>0\) such that}
\[
\begin{aligned}
c(\delta+|x|^2+|y|^2)
&\leq E_++\delta-\E_\gamma(K;k_++x,k_++y)
\leq C(\delta+|x|^2+|y|^2),\\
c\sqrt{\delta+|x|^2}
&\leq\Delta_K^+(k_++x,E_++\delta)
\leq C\sqrt{\delta+|x|^2}
\end{aligned}
\]
for all \(|x|,|y|\leq\rho\) and \(0\leq\delta\leq\rho^2\).
\reviewedit{Proposition~\ref{prop:upper-unique-spectator}
gives \(\Delta_K^+(q,E_+)>0\) for every \(q\ne k_+\).}

\subsection{Statement of the upper-threshold Efimov theorem}
\label{subsec:main-upper-statement}
Let \(q_*\in(0,1)\) be the unique solution of
\begin{equation}
\label{eq:upper-qstar}
\frac{2\bigl[q_*-\sqrt{1-q_*^2}\arcsin q_*\bigr]}
{\pi q_*^2\sqrt{1-q_*^2}}=1,
\end{equation}
and set
\begin{equation}
\label{eq:upper-gammastar}
\notationedit{\gamma_*=\frac{q_*}{1-q_*}\approx13.6069657.}
\end{equation}
Existence and uniqueness of \(q_*\), as well as \(\gamma_*>1\),
are proved in Subsection~\ref{subsec:upper-symmetric-point}.
\edit{This is the critical mass ratio for the homogeneous
threshold operator at \(K=\pi\mathbf1\). For each fixed
\(\gamma>\gamma_*\), its supercriticality persists in a
sufficiently small neighbourhood of that point.}
Fix \(\gamma>\gamma_*\). \reviewedit{Choose an open neighbourhood
\(\cO_+\) of \(\pi\mathbf 1\) sufficiently small that
Theorem~\ref{thm:upper-essential-edge} and the local threshold
estimates hold.} Throughout the proof, \(K\in\cO_+\) is
fixed, and dependence on \(\gamma\) is suppressed from the notation.
Unless stated otherwise, constants may depend on \(K\), \(\gamma\),
and the fixed cutoff, but are independent of sufficiently small
\(\delta>0\). Set
\[
E_+=E_{+,\gamma}(K),
\qquad
\mu_{\mathrm c}^+=\mu_{\mathrm c}^+(K,\gamma).
\]
For \(0\leq\lambda\leq\lambda_{\mathrm{crit}}(K,\gamma)\) and
\(z>E_+\), define
\[
N_+^{(\lambda)}(K,z)
=\dim\operatorname{Ran}\mathbf1_{(z,\infty)}
\bigl(H_{\lambda,\mu_{\mathrm c}^+,\gamma}(K)\bigr),\qquad
N_+(K,z)=N_+^{(0)}(K,z).
\]
\reviewedit{These counts are finite by
Theorem~\ref{thm:upper-essential-edge}(i).}
For a bounded self-adjoint operator \(A\) and \(\alpha\in\mathbb R\),
write
\[
n(\alpha,A)
=\dim\operatorname{Ran}\mathbf 1_{(\alpha,\infty)}(A),
\]
allowing the value \(+\infty\). If \(\alpha>\sup\sess A\), this
quantity is finite and equals the number of eigenvalues of \(A\)
strictly greater than \(\alpha\), counted with multiplicities.
In particular, \(n(\alpha,A)<\infty\) for compact \(A\) and
\(\alpha>0\).

\begin{definition}[Upper Efimov coefficient]
\label{def:upper-U}
For \(\alpha>0\), define
\begin{equation}
\label{eq:upper-U-alpha}
\cU_K^+(\alpha)
=\frac{1}{4\pi}\int_{\mathbb R}
n\bigl(\alpha,\widehat{\cS}_K^+(s)\bigr)\,ds,
\end{equation}
where \(\widehat{\cS}_K^+(s)\) is the compact self-adjoint
operator-valued Fourier symbol of the logarithmic model constructed
in Subsection~\ref{subsec:upper-model}. Set
\[
\cU_+(K)=\cU_K^+(1).
\]
\end{definition}

Proposition~\ref{prop:upper-model-asymptotic} shows that the integral
in \eqref{eq:upper-U-alpha} is finite for every \(\alpha>0\) and
that \(\cU_K^+\) is continuous on \((0,\infty)\).
Proposition~\ref{prop:upper-supercriticality} establishes
\(\cU_+(K)>0\) after a possible reduction of \(\cO_+\).
\reviewedit{For \(\lambda>\lambda_{\mathrm{sep}}(K,\gamma)\) and
\(z,\zeta\in\cG_\lambda^+(K)\) with \(z<\zeta\) and}
\(\zeta\notin\spec(H_{\lambda,\mu_{\mathrm c}^+,\gamma}(K))\), define
\begin{equation}
\label{eq:gap-counting-function-statement}
N_{\mathrm{gap}}^{(\lambda)}(K,z;\zeta)
=\dim\operatorname{Ran}\mathbf1_{(z,\zeta)}
\bigl(H_{\lambda,\mu_{\mathrm c}^+,\gamma}(K)\bigr).
\end{equation}
\reviewedit{This count is finite by
Corollary~\ref{cor:separated-regime}.}

\begin{theorem}[Critical upper-threshold accumulation and stability in a gap]
\label{thm:main-upper}
Assume \eqref{eq:cosine-dispersion-model} and \(\gamma>\gamma_*\).
There exists an open neighbourhood \(\cO_+\) of \(\pi\mathbf1\)
such that, for every \(K\in\cO_+\), with
\(\mu_{\mathrm c}^+=\mu_{\mathrm c}^+(K,\gamma)\), the coefficient in
Definition~\ref{def:upper-U} satisfies \(0<\cU_+(K)<\infty\),
and the following assertions hold.
\begin{enumerate}[(i)]
\item For every fixed \(0\leq\lambda\leq\lambda_{\mathrm{crit}}(K,\gamma)\),
\[
\sup\sess H_{\lambda,\mu_{\mathrm c}^+,\gamma}(K)=E_{+,\gamma}(K).
\]
The operator \(H_{\lambda,\mu_{\mathrm c}^+,\gamma}(K)\) has infinitely many discrete
eigenvalues accumulating at \(E_{+,\gamma}(K)\) from above, and
\begin{equation}
\label{eq:main-upper-asymptotic}
N_+^{(\lambda)}\bigl(K,E_{+,\gamma}(K)+\delta\bigr)
=\cU_+(K)|\log\delta|+o(|\log\delta|),
\qquad \delta\downarrow0.
\end{equation}
\item If \(\lambda>\lambda_{\mathrm{sep}}(K,\gamma)\), then
\(\cG_\lambda^+(K)\) is a nonempty open gap in the essential
spectrum of \(H_{\lambda,\mu_{\mathrm c}^+,\gamma}(K)\), and
\[
E_{+,\gamma}(K),\ \tau_{12}^+(K,\lambda)
\in\sess H_{\lambda,\mu_{\mathrm c}^+,\gamma}(K).
\]
Infinitely many discrete eigenvalues in this gap accumulate
at its lower endpoint. For every fixed
\[
\zeta\in\cG_\lambda^+(K)
\setminus\spec H_{\lambda,\mu_{\mathrm c}^+,\gamma}(K),
\]
the counting function defined in
\eqref{eq:gap-counting-function-statement} satisfies
\begin{equation}
\label{eq:upper-gap-Efimov-asymptotic}
N_{\mathrm{gap}}^{(\lambda)}
\bigl(K,E_{+,\gamma}(K)+\delta;\zeta\bigr)
=\cU_+(K)|\log\delta|+o(|\log\delta|),
\qquad \delta\downarrow0.
\end{equation}
The limit is taken with \(K,\gamma,\lambda,\zeta\) fixed and
\(0<\delta<\zeta-E_{+,\gamma}(K)\). In particular, the leading
coefficient is the same as in \eqref{eq:main-upper-asymptotic}
and is independent of \(\lambda\).
\item For
\(\lambda_{\mathrm{crit}}(K,\gamma)<\lambda\leq\lambda_{\mathrm{sep}}(K,\gamma)\),
the essential spectrum contains \([E_{+,\gamma}(K),E_{+,\gamma}(K)+a]\)
for some \(a>0\); no sequence of discrete eigenvalues can approach
\(E_{+,\gamma}(K)\) from above.
\end{enumerate}
\end{theorem}

\begin{remark}[Fibrewise character of the theorem]
The critical upper Efimov family is the graph
\[
\Gamma_+^{\mathrm{Ef}}
=\{(K,\mu_{\mathrm c}^+(K,\gamma)):K\in\cO_+\}
\subset\T^3\times(0,\infty).
\]
The theorem tunes \(\mu\) separately in each fibre; it does not
assert Efimov accumulation for one fixed coupling throughout
\(\cO_+\).
By Lemma~\ref{lem:upper-critical-regularity}, the function
\(K\mapsto\mu_{\mathrm c}^+(K,\gamma)\) is real analytic.
Thus \(\Gamma_+^{\mathrm{Ef}}\) is a real-analytic submanifold
of codimension one in \(\T^3\times(0,\infty)\).
\end{remark}

\subsection{Birman--Schwinger reduction}
\label{subsec:upper-BS}
\reviewedit{In this subsection, \(\lambda=0\) and
\(\frh=L^2(\T^3,\eta)\).}
Recall that
\[
(\Psi f)(q)
=
\int_{\T^3}f(p,q)\,\eta(dp),
\qquad
(\Psi^*g)(p,q)
=
\frac12\bigl(g(q)-g(p)\bigr),
\]
and
\[
V_{13}+V_{23}=2\Psi^*\Psi.
\]
\reviewedit{Recall that, for \(z>E_+\),}
\[
\Delta_K^+(q,z)
=
1-\mu_{\mathrm c}^+
\int_{\T^3}
\frac{\eta(dp)}
     {z-\E_\gamma(K;p,q)}.
\]
\reviewedit{Lemma~\ref{lem:Delta-uniform-positivity} gives}
\[
m_z
:=
\min_{q\in\T^3}\Delta_K^+(q,z)
>
0.
\]
If
\[
\bigl(H_{0,\mu_{\mathrm c}^+,\gamma}(K)-z\bigr)f=0,
\]
then \(g=\Psi f\) satisfies
\begin{align}
f(p,q)
&=
\mu_{\mathrm c}^+\,
\frac{g(q)-g(p)}
     {z-\E_\gamma(K;p,q)},
\label{eq:upper-Faddeev-f}
\\
\Delta_K^+(q,z)g(q)
&=
-\mu_{\mathrm c}^+
\int_{\T^3}
\frac{g(p)}
     {z-\E_\gamma(K;p,q)}
\,\eta(dp).
\label{eq:upper-Faddeev-g}
\end{align}
Conversely, any \(g\in\frh\) satisfying the second
relation determines, through the first, an element
\(f\in\cH^{\mathrm a}\), since
\[
z-\E_\gamma(K;p,q)\ge z-E_+>0.
\]
The symmetry of the free dispersion gives
\(f(p,q)=-f(q,p)\) almost everywhere. Integrating the first
relation and using the second, we obtain
\[
\Psi f=(1-\Delta_K^+(\cdot,z))g+\Delta_K^+(\cdot,z)g=g.
\]
Consequently, \(f\) satisfies the eigenvalue equation, and
\(f\mapsto\Psi f\) is a linear isomorphism between the two
solution spaces.
Define
\[
(D_zg)(q)
=
\Delta_K^+(q,z)^{1/2}g(q).
\]
The operator \(D_z\) is boundedly invertible, with
\[
\|D_z^{-1}\|\le m_z^{-1/2}.
\]
The substitution \(\varphi=D_zg\) transforms the channel
equation into
\[
T_K^+(z)\varphi=\varphi.
\]

\begin{definition}[Birman--Schwinger operator]
\label{def:upper-BS-operator}
For \(z>E_+\), define \(T_K^+(z):\frh\to\frh\) by
\begin{equation}
\label{eq:upper-BS-kernel}
(T_K^+(z)\varphi)(q)
=-
\mu_{\mathrm c}^+\int_{\T^3}
\frac{\varphi(p)\,\eta(dp)}
{\Delta_K^+(q,z)^{1/2}
 \Delta_K^+(p,z)^{1/2}
 [z-\E_\gamma(K;p,q)]}.
\end{equation}
\end{definition}

\revblock{
\begin{proposition}[Properties above the upper threshold]
\label{prop:upper-BS-properties}
For every \(z>E_+\), the operator \(T_K^+(z)\) is
self-adjoint and Hilbert--Schmidt. The map
\(z\mapsto T_K^+(z)\) is continuous on \((E_+,\infty)\)
in the Hilbert--Schmidt norm.
\end{proposition}

\begin{proof}
For each fixed \(z>E_+\), positivity and continuity of
\(\Delta_K^+(\cdot,z)\), together with
\(z-\E_\gamma(K;p,q)\ge z-E_+>0\), show that the
kernel in \eqref{eq:upper-BS-kernel} is bounded.
It is real and symmetric, so \(T_K^+(z)\) is self-adjoint
and Hilbert--Schmidt. The denominator bounds are uniform
on compact subintervals of \((E_+,\infty)\); continuity
of the kernel and dominated convergence therefore give
continuity in the Hilbert--Schmidt norm.
\end{proof}
}

\begin{theorem}[Upper-threshold Birman--Schwinger principle]
\label{thm:upper-BS-principle}
For every \(K\in\mathcal O_+\) and \(z>E_{+,\gamma}(K)\),
\[
\dim\ker\bigl(H_{0,\mu_{\mathrm c}^+,\gamma}(K)-z\bigr)
=
\dim\ker\bigl(T_K^+(z)-I\bigr),
\]
and
\begin{equation}
\label{eq:upper-BS-counting}
N_+(K,z)
=
n\bigl(1,T_K^+(z)\bigr).
\end{equation}
\end{theorem}

\begin{proof}
The eigenfunction correspondence above identifies the kernels
through \(f\mapsto D_z\Psi f\). For the counting identity, put
\(F_z=z-H_\gamma^{\mathrm{free}}(K)>0\) and
\(X_z=\sqrt{2\mu_{\mathrm c}^+}\Psi F_z^{-1/2}\). Then
\[
z-H_{0,\mu_{\mathrm c}^+,\gamma}(K)=F_z^{1/2}(I-X_z^*X_z)F_z^{1/2}.
\]
The variational principle, invariance of negative inertia under
boundedly invertible congruences, and polar decomposition give
\(N_+(K,z)=n(1,X_z^*X_z)=n(1,X_zX_z^*)\).
Define \(\cB_K^+(z)=X_zX_z^*=2\mu_{\mathrm c}^+\Psi F_z^{-1}\Psi^*\).
The formula for \(\Psi^*\) yields
\[
(\cB_K^+(z)g)(q)
=[1-\Delta_K^+(q,z)]g(q)
-\mu_{\mathrm c}^+\int_{\T^3}\frac{g(p)\,\eta(dp)}{z-\E_\gamma(K;p,q)},
\]
and hence
\[
\cB_K^+(z)-I=D_z[T_K^+(z)-I]D_z.
\]
This congruence preserves positive inertia, since \(D_z\) is
boundedly invertible. \reviewedit{Thus
\(n(1,\cB_K^+(z))=n(1,T_K^+(z))\).}
\end{proof}

\subsection{Singular part and regular remainder}
\label{subsec:upper-HS}
To analyse \(T_K^+(E_++\delta)\) as \(\delta\downarrow0\),
we introduce the translated coordinates
\(p=k_K^++x\), \(q=k_K^++y\). The map
\[
(U_Kh)(x)=h(k_K^++x)
\]
is unitary on \(L^2(\T^3,\eta)\). In this subsection and the
homogeneous reduction below, we retain the notation
\(T_K^+(E_++\delta)\) for its conjugate under \(U_K\).
Thus the critical spectator momentum is represented by zero.
All integral kernels below are taken relative
to the normalised Haar measure unless another measure is specified.
Choose \(b>0\) sufficiently small that the local expansions and
estimates used below hold for \(|x|,|y|<b\), and that
\(B_b(0)\) lies in a single torus chart. Fix a real function
\(\chi\in C_c^\infty(B_b(0))\), with \(0\le\chi\le1\),
equal to one in a neighbourhood of zero. We regard \(\chi\) as a
function on the torus by extending it by zero.
The expansions \eqref{eq:upper-energy-expansion} and
\eqref{eq:upper-determinant-asymptotic} lead to the following
local models:
\begin{align}
Q_K^+(x,y)
&=
\langle B_K^+x,x\rangle
+2\langle C_K^+y,x\rangle
+\langle B_K^+y,y\rangle,
\label{eq:upper-Q}
\\
d_{K,\delta}^+(x)
&=
\mathfrak c_K^+
\sqrt{\delta+\langle R_K^+x,x\rangle}
-\langle\ell_K^+,x\rangle.
\label{eq:upper-d}
\end{align}
The block-matrix positivity preceding \eqref{eq:upper-R}
implies that \(Q_K^+\) is positive definite. Symmetry of
\(C_K^+\) gives \(Q_K^+(x,y)=Q_K^+(y,x)\).
By \eqref{eq:upper-R-Cauchy-Schwarz} and strict transversality,
\begin{equation}
\label{eq:upper-d-positive-bound}
d_{K,\delta}^+(x)\ge
(\mathfrak c_K^+-a_K)\sqrt{\delta+\langle R_K^+x,x\rangle}>0
\quad(\delta>0\text{ or }x\ne0),
\end{equation}
where \(a_K\) is defined in \eqref{eq:upper-aK}.
Define \(S_{K,\delta}^+\) on \(L^2(\T^3,\eta)\) by
\[
(S_{K,\delta}^+f)(x)
=\int_{\T^3}s_{K,\delta}^+(x,y)f(y)\,\eta(dy)
\quad\text{for almost every }x,
\]
where its kernel, relative to \(\eta(dy)=dy/(2\pi)^3\), is
\begin{equation}
\label{eq:upper-singular-kernel}
s_{K,\delta}^+(x,y)
=-\frac{\chi(x)\chi(y)}
{d_{K,\delta}^+(x)^{1/2}d_{K,\delta}^+(y)^{1/2}
 [\delta+Q_K^+(x,y)]}.
\end{equation}
Here \(\delta>0\), and the kernel is extended by zero outside
\(B_b(0)\times B_b(0)\). It is bounded, real and symmetric, so
\(S_{K,\delta}^+\) is self-adjoint and Hilbert--Schmidt for
every fixed \(\delta>0\).
The factor \(\mu_{\mathrm c}^+\) in the Birman--Schwinger kernel cancels
with the two factors \((\mu_{\mathrm c}^+)^{-1/2}\) supplied by
\eqref{eq:upper-determinant-asymptotic}. This gives precisely
\eqref{eq:upper-singular-kernel}, with no additional coupling
factor.
We write \(\mathfrak S_2(\mathcal H_1,\mathcal H_2)\) for the
Hilbert--Schmidt class and \(\mathfrak S_1(\mathcal H)\) for the
trace class, with their respective ideal norms. When the domain
and target coincide, we write \(\mathfrak S_2(\mathcal H)\).

\begin{lemma}[Kernel majorants for the Hilbert--Schmidt reduction]
\label{lem:kernel-majorants}
Let \(t_{K,\delta}^+\) be the kernel of \(T_K^+(E_++\delta)\)
in the translated coordinates, and let \(s_{K,\delta}^+\) be
the kernel defined in \eqref{eq:upper-singular-kernel}.
Put
\[
s_x=(\delta+|x|^2)^{1/2},\qquad
s_y=(\delta+|y|^2)^{1/2},\qquad
\rho=(\delta+|x|^2+|y|^2)^{1/2}.
\]
\revblock{There exist \(\delta_0>0\) and an integrable function
\(G\) on \((\T^3)^2\), independent of \(\delta\), such that}
\[
\bigl|\chi(x)\chi(y)t_{K,\delta}^+(x,y)-s_{K,\delta}^+(x,y)\bigr|^2
+\bigl|[1-\chi(x)\chi(y)]t_{K,\delta}^+(x,y)\bigr|^2
\le G(x,y)
\]
for almost every \((x,y)\) and \(0<\delta\le\delta_0\).
\end{lemma}

\begin{proof}
Theorem~\ref{thm:upper-determinant-asymptotics} and strict
transversality give
\[
\Delta_K^+(k_K^++x,E_++\delta)\asymp s_x,\qquad
\mu_{\mathrm c}^+ d_{K,\delta}^+(x)\asymp s_x,
\]
and
\[
\Delta_K^+(k_K^++x,E_++\delta)
=\mu_{\mathrm c}^+ d_{K,\delta}^+(x)+O(s_x^2).
\]
Thus the exact determinant and its local model are both
comparable to \(s_x\). The identity
\[
|a^{-1/2}-b^{-1/2}|
=
\frac{|a-b|}{\sqrt{ab}(\sqrt a+\sqrt b)},\qquad a,b>0,
\]
then implies
\[
\left|
\Delta_K^+(k_K^++x,E_++\delta)^{-1/2}
-[\mu_{\mathrm c}^+ d_{K,\delta}^+(x)]^{-1/2}
\right|
\le Cs_x^{1/2}.
\]
Write
\[
D_\delta(x,y)
=E_++\delta-\E_\gamma(K;k_K^++x,k_K^++y),
\qquad
D_\delta^0(x,y)=\delta+Q_K^+(x,y).
\]
The Taylor remainder in \eqref{eq:upper-energy-expansion} is
\(O((|x|+|y|)^3)\); nondegeneracy implies
\[
D_\delta(x,y)\asymp\rho^2,
\qquad D_\delta^0(x,y)\asymp\rho^2,
\]
so the identity for the difference of reciprocals gives
\[
\left|
D_\delta(x,y)^{-1}-D_\delta^0(x,y)^{-1}
\right|
\le C\frac{(|x|+|y|)^3}{\rho^4}
\le C\rho^{-1}.
\]
Set \(a_\delta(x)=\Delta_K^+(k_K^++x,E_++\delta)^{-1/2}\)
and \(a_\delta^0(x)=[\mu_{\mathrm c}^+ d_{K,\delta}^+(x)]^{-1/2}\).
In the difference of the localized kernels, use the exact
telescoping identity
\[
\begin{aligned}
\frac{a_\delta(x)a_\delta(y)}{D_\delta}
-\frac{a_\delta^0(x)a_\delta^0(y)}{D_\delta^0}
={}&\frac{[a_\delta(x)-a_\delta^0(x)]a_\delta(y)}{D_\delta}
\\
&+\frac{a_\delta^0(x)[a_\delta(y)-a_\delta^0(y)]}{D_\delta}
\\
&+a_\delta^0(x)a_\delta^0(y)
\left(\frac1{D_\delta}-\frac1{D_\delta^0}\right).
\end{aligned}
\]
The factor \(\mu_{\mathrm c}^+\chi(x)\chi(y)\) is bounded, and the preceding
estimates bound the absolute value of the localized difference by
\[
C\left[
\frac{s_x^{1/2}s_y^{-1/2}}{\rho^2}
+\frac{s_x^{-1/2}s_y^{1/2}}{\rho^2}
+\frac{s_x^{-1/2}s_y^{-1/2}}{\rho}
\right].
\]
Its square is bounded by
\[
C\left[
\frac{s_x}{s_y\rho^4}
+\frac{s_y}{s_x\rho^4}
+\frac{1}{s_xs_y\rho^2}
\right].
\]
For \(r=|x|>0\) and \(t=|y|>0\),
\[
\frac{s_x}{s_y}\le1+\frac rt,\qquad
\frac{s_y}{s_x}\le1+\frac tr,\qquad
s_xs_y\ge rt,\qquad \rho^2\ge r^2+t^2.
\]
Consequently, a \(\delta\)-independent majorant for the squared
difference is
\[
C\left[
\frac{1+r/t+t/r}{(r^2+t^2)^2}
+\frac{1}{rt(r^2+t^2)}
\right].
\]
This majorant is locally integrable in six dimensions. Indeed, after multiplication by the polar
Jacobian \(r^2t^2\), its four radial terms are bounded:
\[
\frac{r^2t^2}{(r^2+t^2)^2}\le\frac14,\qquad
\frac{r^3t}{(r^2+t^2)^2}\le\frac12,
\]
\[
\frac{rt^3}{(r^2+t^2)^2}\le\frac12,\qquad
\frac{rt}{r^2+t^2}\le\frac12.
\]
We next estimate the complementary kernel
\[
[1-\chi(x)\chi(y)]\,t_{K,\delta}^+(x,y).
\]
Its support is disjoint from a fixed neighbourhood of \((0,0)\).
Uniqueness of the joint maximum and compactness of the torus
therefore give a positive lower bound for its free denominator,
uniformly in \(\delta\ge0\).
Put \(w(x)=1+d_{\T^3}(x,0)^{-1}\) for \(x\ne0\), where
\(d_{\T^3}\) is the torus distance, and assign any value at zero.
Then \(w\in L^1(\T^3,\eta)\), and the following estimates hold
almost everywhere.
The local determinant estimate, positivity of the threshold
determinant away from zero, and its monotonicity in energy imply
\[
\Delta_K^+(k_K^++x,E_++\delta)^{-1}\le Cw(x)
\quad\text{for a.e. }x,\qquad 0<\delta\le\delta_0.
\]
Consequently, the squared absolute value of the complementary
kernel is at most \(Cw(x)w(y)\) for almost every \((x,y)\).
This product is integrable and independent of \(\delta\).
Combining the two estimates yields the function \(G\).

\end{proof}

\begin{proposition}[Hilbert--Schmidt reduction]
\label{prop:upper-HS-reduction}
For each fixed \(K\in\cO_+\) and cutoff \(\chi\) as above,
there is \(\delta_0>0\) such that
\begin{equation}
\label{eq:upper-HS-decomposition}
T_K^+(E_++\delta)=S_{K,\delta}^++\cR_{K,\delta}^+,
\qquad 0<\delta\le\delta_0,
\end{equation}
where \(\cR_{K,\delta}^+\) is self-adjoint and admits a
continuous extension to \(\delta=0\) in the Hilbert--Schmidt
norm on \(L^2(\T^3,\eta)\).
In particular, there exists \(\cR_{K,0}^+\) in this class
such that
\[
\lim_{\delta\downarrow0}
\|T_K^+(E_++\delta)-S_{K,\delta}^+
-\cR_{K,0}^+\|_{\mathfrak S_2}=0.
\]
\end{proposition}

\begin{proof}
The difference \(\cR_{K,\delta}^+=T_K^+(E_++\delta)-S_{K,\delta}^+\)
has real symmetric kernel
\(r_\delta=t_{K,\delta}^+-s_{K,\delta}^+\).
Lemma~\ref{lem:kernel-majorants} gives
\(|r_\delta|^2\le G\) almost everywhere, with
\(G\in L^1((\T^3)^2)\) independent of \(0<\delta\le\delta_0\).

For fixed \(q\ne k_K^+\), uniqueness of the free maximum
and compactness imply
\[
\inf_{p\in\T^3}[E_+-\E_\gamma(K;p,q)]>0.
\]
Hence the defining integral for \(\Delta_K^+(q,z)\) is continuous
at \(E_+\), with positive limit by \eqref{eq:upper-unique-zero}.
The translated kernel \eqref{eq:upper-BS-kernel} therefore has
a finite limit almost everywhere. The same holds for the model
kernel by \eqref{eq:upper-singular-kernel}, positivity of \(Q_K^+\),
and \eqref{eq:upper-d-positive-bound}. Thus \(r_\delta\to r_0\)
almost everywhere. Since \(|r_0|^2\le G\) and
\(|r_\delta-r_0|^2\le4G\), dominated convergence gives
\[
\|\cR_{K,\delta}^+-\cR_{K,0}^+\|_{\mathfrak S_2}^2
=\int_{(\T^3)^2}|r_\delta-r_0|^2\,\eta(dx)\eta(dy)
\longrightarrow0.
\]
The limit is self-adjoint because \(r_0\) is real and symmetric.

On any \([\delta_1,\delta_2]\subset(0,\delta_0]\),
\eqref{eq:upper-d-positive-bound} gives
\(d_{K,\delta}^+(x)\ge c\sqrt{\delta_1}\) on the cutoff support,
including at \(x=0\), and \(\delta+Q_K^+(x,y)\ge\delta_1\).
The exact denominators are uniformly positive by compactness
and monotonicity of \(\Delta_K^+\). The explicit kernels are
therefore pointwise continuous in \(\delta\); the same majorant
\(|r_\delta-r_{\delta_*}|^2\le4G\) proves Hilbert--Schmidt
continuity at each \(\delta_*>0\).
\end{proof}

We recall the following perturbation lemma; see
\cite[Lemma~4.7]{Sobolev1993} and
\cite[Lemma~7.2]{ALM2004}. For completeness, we include a short proof.

\begin{lemma}[Stability of eigenvalue-counting asymptotics]
\label{lem:upper-HS-stability}
Let \(A_\delta\), \(0<\delta\le\delta_0\), be a family of
compact self-adjoint operators on a Hilbert space \(\mathcal H\),
continuous in the operator norm for \(\delta>0\).
Let \(B_\delta\), \(0\le\delta\le\delta_0\), be a family
of compact self-adjoint operators on \(\mathcal H\),
continuous in the operator norm up to \(\delta=0\), and set
\[
T_\delta=A_\delta+B_\delta,
\qquad 0<\delta\le\delta_0.
\]
Suppose that \(f(\delta)>0\), \(f(\delta)\to0\) as
\(\delta\downarrow0\), and that, for every \(\alpha>0\),
\[
\lim_{\delta\downarrow0}
f(\delta)n(\alpha,A_\delta)=U(\alpha),
\]
where \(U\) is finite and continuous on \((0,\infty)\).
Then, for every \(\alpha>0\),
\[
\lim_{\delta\downarrow0}
f(\delta)n(\alpha,T_\delta)=U(\alpha).
\]
\end{lemma}

\begin{proof}
Fix \(\alpha>0\) and \(0<\epsilon<\alpha\).
Since \(B_\delta\to B_0\) in operator norm, for sufficiently
small \(\delta\) we have \(\|B_\delta-B_0\|<\epsilon/2\).
Hence \(\pm B_\delta\le\pm B_0+(\epsilon/2)I\) in the
quadratic-form sense, and the variational principle gives
\[
n(\epsilon,\pm B_\delta)
\le n(\epsilon/2,\pm B_0)<\infty.
\]
Compactness of \(B_0\) makes the right-hand side finite and
independent of \(\delta\). Since \(f(\delta)\to0\), it follows that
\[
f(\delta)n(\epsilon,\pm B_\delta)\longrightarrow0
\qquad\text{as }\delta\downarrow0.
\]
Applying Weyl's inequality to
\(T_\delta=A_\delta+B_\delta\) and
\(A_\delta=T_\delta-B_\delta\), we obtain
\[
\begin{aligned}
n(\alpha+\epsilon,A_\delta)-n(\epsilon,-B_\delta)
&\le n(\alpha,T_\delta)\\
&\le n(\alpha-\epsilon,A_\delta)+n(\epsilon,B_\delta).
\end{aligned}
\]
Multiplying by \(f(\delta)\) and letting
\(\delta\downarrow0\) at fixed \(\epsilon\), we find
\[
\begin{aligned}
U(\alpha+\epsilon)
&\le\liminf_{\delta\downarrow0}
f(\delta)n(\alpha,T_\delta)\\
&\le\limsup_{\delta\downarrow0}
f(\delta)n(\alpha,T_\delta)
\le U(\alpha-\epsilon).
\end{aligned}
\]
Letting \(\epsilon\downarrow0\) and using continuity of
\(U\) proves the assertion.
\end{proof}

\begin{remark}[Uniformly Hilbert--Schmidt perturbations]
\label{rem:upper-HS-stability}
For compact self-adjoint families \(A_\delta\) and \(B_\delta\),
\(0<\delta\le\delta_0\), the continuity assumptions in
Lemma~\ref{lem:upper-HS-stability} may be replaced by
\[
B_\delta\in\mathfrak S_2,\qquad
M:=\sup_{0<\delta\le\delta_0}
\|B_\delta\|_{\mathfrak S_2}<\infty.
\]
All other hypotheses are retained, but no limit operator
\(B_0\) is required.

Indeed, the spectral theorem gives, for each fixed \(\epsilon>0\),
\[
n(\epsilon,\pm B_\delta)
\le\epsilon^{-2}\|B_\delta\|_{\mathfrak S_2}^2
\le\epsilon^{-2}M^2,
\qquad
f(\delta)n(\epsilon,\pm B_\delta)\longrightarrow0
\quad\text{as }\delta\downarrow0.
\]
Thus the perturbation terms in the Weyl inequalities vanish
after multiplication by \(f(\delta)\). The rest of the proof is
unchanged and uses no continuity of \(A_\delta\).
\end{remark}

\subsection{Reduction to the logarithmic model}
\label{subsec:upper-model}
\reviewedit{Proposition~\ref{prop:upper-HS-reduction} isolates the local
principal operator \(S_{K,\delta}^+\). A homogeneous approximation
and a logarithmic change of the radial variable reduce the
problem to a convolution on an interval of length
\(\tfrac12|\log\delta|+O(1)\).}

\medskip
\noindent\textit{Rescaling and homogeneous approximation.}
We retain the translated coordinates, the cutoff \(\chi\),
and the local functions \(Q_K^+\), \(d_{K,\delta}^+\) introduced
in \eqref{eq:upper-Q}--\eqref{eq:upper-d}. The kernel of
\(S_{K,\delta}^+\) is given by \eqref{eq:upper-singular-kernel}.
The matrices satisfy
\[
R_K^+=B_K^+-(C_K^+)^T(B_K^+)^{-1}C_K^+,
\]
with \(B_K^+\), \(R_K^+\) positive definite and \(C_K^+\)
symmetric.
To normalize the quadratic scale, define
\begin{equation}
\label{eq:upper-rescaled-geometry}
\begin{aligned}
A_K^+&=(R_K^+)^{-1/2}B_K^+(R_K^+)^{-1/2},\\
D_K^+&=(R_K^+)^{-1/2}C_K^+(R_K^+)^{-1/2},\\
g_K^+&=(R_K^+)^{-1/2}\ell_K^+.
\end{aligned}
\end{equation}
The identity for \(R_K^+\) gives
\begin{equation}
\label{eq:upper-Schur-identity}
I=A_K^+-(D_K^+)^T(A_K^+)^{-1}D_K^+.
\end{equation}
Here \(A_K^+\) is positive definite and \(D_K^+\) is symmetric.
Strict transversality, \(|g_K^+|<\mathfrak c_K^+\), implies
\begin{equation}
\label{eq:upper-vartheta}
\vartheta_K^+(\omega)
=\mathfrak c_K^+-\langle g_K^+,\omega\rangle
\ge\mathfrak c_K^+-|g_K^+|>0,
\qquad\omega\in\Sph^2.
\end{equation}
Set
\[
L_{K,\delta}=\sqrt\delta\,(R_K^+)^{-1/2},
\qquad
J_{K,\delta}=\det L_{K,\delta}
=\frac{\delta^{3/2}}{\sqrt{\det R_K^+}}.
\]
Identify functions supported in the chosen torus chart with
functions of the local coordinate \(x\). For
\(U=\operatorname{supp}\chi\) and
\(\Omega_{K,\delta}=L_{K,\delta}^{-1}U\), the map
\[
\begin{aligned}
\cD_{K,\delta}:L^2\!\left(U,\frac{dx}{(2\pi)^3}\right)
&\longrightarrow L^2(\Omega_{K,\delta},d\xi),\\
(\cD_{K,\delta}f)(\xi)
&=\left(\frac{J_{K,\delta}}{(2\pi)^3}\right)^{1/2}
f(L_{K,\delta}\xi)
\end{aligned}
\]
is unitary; the formula holds for almost every \(\xi\).
We identify \(L^2(U,dx/(2\pi)^3)\) with the closed subspace
of functions vanishing almost everywhere outside \(U\).
Let \(\widetilde S_{K,\delta}^+\) be the extension by zero to
\(L^2(\R^3)\) of
\[
\cD_{K,\delta}(S_{K,\delta}^+|_{L^2(U)})\cD_{K,\delta}^{-1}.
\]
Since \(S_{K,\delta}^+\) vanishes on
\(L^2(U)^\perp\), nonzero eigenvalues and multiplicities are preserved.
Put
\[
\begin{aligned}
\chi_{K,\delta}(\xi)&=\chi(L_{K,\delta}\xi),\\
a_K^+(\xi)&=\mathfrak c_K^+\sqrt{1+|\xi|^2}
-\langle g_K^+,\xi\rangle,\\
\cQ_K^+(\xi,\eta)&=\langle A_K^+\xi,\xi\rangle
+2\langle D_K^+\eta,\xi\rangle
+\langle A_K^+\eta,\eta\rangle.
\end{aligned}
\]
By \eqref{eq:upper-Schur-identity},
\[
\cQ_K^+(\xi,\eta)
=\bigl\langle A_K^+(\xi+(A_K^+)^{-1}D_K^+\eta),
\xi+(A_K^+)^{-1}D_K^+\eta\bigr\rangle+|\eta|^2.
\]
The change \((\xi,\eta)\mapsto
(\xi+(A_K^+)^{-1}D_K^+\eta,\eta)\) is invertible.
Consequently, for some \(0<c_K\le1\),
\begin{equation}
\label{eq:upper-rescaled-coercivity}
c_K(|\xi|^2+|\eta|^2)\le\cQ_K^+(\xi,\eta)
\le c_K^{-1}(|\xi|^2+|\eta|^2).
\end{equation}
There are \(0<c_0<c_1\), independent of \(\delta\), such that
\(\chi_{K,\delta}=1\) on \(|\xi|\le c_0\delta^{-1/2}\)
and vanishes on \(|\xi|\ge c_1\delta^{-1/2}\).
In particular,
\[
a_K^+(\xi)\ge
(\mathfrak c_K^+-|g_K^+|)\sqrt{1+|\xi|^2}>0.
\]
The identities
\[
\begin{aligned}
d_{K,\delta}^+(L_{K,\delta}\xi)
&=\sqrt\delta\,a_K^+(\xi),\\
\delta+Q_K^+(L_{K,\delta}\xi,L_{K,\delta}\eta)
&=\delta[1+\cQ_K^+(\xi,\eta)]
\end{aligned}
\]
show that \(\widetilde S_{K,\delta}^+\) has kernel, relative
to Lebesgue measure \(d\eta\),
\begin{equation}
\label{eq:upper-scaled-principal-kernel}
\widetilde s_{K,\delta}^+(\xi,\eta)
=-\frac{1}{(2\pi)^3\sqrt{\det R_K^+}}
\frac{\chi_{K,\delta}(\xi)\chi_{K,\delta}(\eta)}
{a_K^+(\xi)^{1/2}a_K^+(\eta)^{1/2}
[1+\cQ_K^+(\xi,\eta)]}.
\end{equation}
Indeed, transformation of the kernel multiplies it by
\(J_{K,\delta}/(2\pi)^3\). The two inverse square-root factors
contribute \(\delta^{-1/2}\), and the inverse quadratic factor
contributes \(\delta^{-1}\). These powers cancel the
\(\delta^{3/2}\) in the Jacobian. Thus only the cutoff still
depends on \(\delta\).
For \(\xi,\eta\ne0\), write
\(\omega=\xi/|\xi|\), \(\nu=\eta/|\eta|\), and define
\begin{equation}
\label{eq:upper-homogeneous-kernel}
s_K^{+,\mathrm{hom}}(\xi,\eta)
=-\frac{1}{(2\pi)^3\sqrt{\det R_K^+}}
\frac{[|\xi||\eta|\vartheta_K^+(\omega)
\vartheta_K^+(\nu)]^{-1/2}}{\cQ_K^+(\xi,\eta)}.
\end{equation}
This is the homogeneous part of the rescaled kernel:
\[
s_K^{+,\mathrm{hom}}(a\xi,a\eta)
=a^{-3}s_K^{+,\mathrm{hom}}(\xi,\eta),\qquad a>0.
\]
\reviewedit{Choose \(0<\delta_0<\min\{1,c_0^2\}\) and set}
\(\rho_\delta=c_0\delta^{-1/2}>1\). For
\(0<\delta\le\delta_0\), let \(\mathsf H_{K,\delta}^+\)
be the integral operator on \(L^2(\R^3,d\xi)\) with kernel
\[
\mathbf1_{\{1<|\xi|<\rho_\delta\}}
s_K^{+,\mathrm{hom}}(\xi,\eta)
\mathbf1_{\{1<|\eta|<\rho_\delta\}}.
\]
Its kernel is real, symmetric and square integrable, so the
operator is self-adjoint and Hilbert--Schmidt.

\begin{lemma}[Homogeneous reduction]
\label{lem:upper-homogeneous-reduction}
For each fixed \(K\in\cO_+\),
\[
\sup_{0<\delta\le\delta_0}
\|\widetilde S_{K,\delta}^+-\mathsf H_{K,\delta}^+\|_{\mathfrak S_2}
<\infty.
\]
\end{lemma}

\begin{proof}
Put \(r=|\xi|\) and \(t=|\eta|\). We use
\eqref{eq:upper-rescaled-coercivity} to estimate all denominators.
Introduce the intermediate operator \(\mathsf G_{K,\delta}^+\)
whose kernel is
\[
\mathbf1_{\{r>1\}}\chi_{K,\delta}(\xi)
s_K^{+,\mathrm{hom}}(\xi,\eta)
\chi_{K,\delta}(\eta)\mathbf1_{\{t>1\}}.
\]
\emph{Large radii with the same cutoff.}
For \(r\ge1\), uniformly in \(\omega\in\Sph^2\),
\[
\begin{aligned}
a_K^+(r\omega)-r\vartheta_K^+(\omega)
&=\frac{\mathfrak c_K^+}{\sqrt{1+r^2}+r}=O(r^{-1}),\\
a_K^+(r\omega)^{-1/2}
-[r\vartheta_K^+(\omega)]^{-1/2}
&=O(r^{-5/2}).
\end{aligned}
\]
The second estimate follows from the first and the uniform
positive lower bound for \(\vartheta_K^+\).
For \(r,t\ge1\),
\[
\left|\frac1{1+\cQ_K^+}-\frac1{\cQ_K^+}\right|
=\frac1{\cQ_K^+(1+\cQ_K^+)}
\le C(r^2+t^2)^{-2}.
\]
Compare the scaled kernel with
\(\chi_{K,\delta}(\xi)s_K^{+,\mathrm{hom}}(\xi,\eta)
\chi_{K,\delta}(\eta)\) on \(r,t>1\).
Subtracting the factors one at a time bounds the absolute
kernel difference by
\[
C\frac{(rt)^{-1/2}}{r^2+t^2}
\left(r^{-2}+t^{-2}+(r^2+t^2)^{-1}\right).
\]

Using
\[
(a+b+c)^2\le3(a^2+b^2+c^2),
\qquad a,b,c\ge0,
\]
with \(a=r^{-2}\), \(b=t^{-2}\), and
\(c=(r^2+t^2)^{-1}\), we absorb all cross terms into
the sum of squares.
After integration over the angular variables and inclusion
of the radial Jacobian \(r^2t^2\), the squared
Hilbert--Schmidt norm is therefore bounded by a constant
times
\[
\int_1^\infty\!\int_1^\infty
\frac{rt}{(r^2+t^2)^2}
\left[r^{-4}+t^{-4}+(r^2+t^2)^{-2}\right]\,dt\,dr.
\]

This integral is finite. Indeed,
\[
\int_1^\infty\frac{rt}{(r^2+t^2)^2}\,dt
=\frac{r}{2(r^2+1)}\le\frac1{2r},
\]
and \((r^2+t^2)^{-2}\le r^{-4}\); symmetry treats the
term with \(t^{-4}\).
\emph{Small radii.}
The scaled kernel is bounded for \(r,t\le1\). For
\(r\le1<t\), its square is bounded by \(Ct^{-5}\), because
\(a_K^+(\xi)\ge c>0\), \(a_K^+(\eta)\ge ct\), and
\(1+\cQ_K^+(\xi,\eta)\ge ct^2\). Therefore its squared
norm on this region is bounded by
\[
C\int_0^1r^2\,dr\int_1^\infty t^{-3}\,dt<\infty.
\]
The region \(t\le1<r\) is treated symmetrically.
On these regions the kernels of \(\mathsf G_{K,\delta}^+\) and
\(\mathsf H_{K,\delta}^+\) vanish, so the kernel of
\(\widetilde S_{K,\delta}^+-\mathsf H_{K,\delta}^+\) is
the scaled kernel estimated above.
\emph{Replacement of the cutoff.}
It remains to compare the two homogeneous kernels on \(r,t>1\).
Set \(C_0=c_1/c_0>1\). Their cutoff difference is supported
where one radius belongs to \([\rho_\delta,C_0\rho_\delta]\) and
both are at most \(C_0\rho_\delta\). After angular integration,
the squared homogeneous kernel, including the radial measure,
is bounded by
\[
C\frac{rt}{(r^2+t^2)^2}\,dr\,dt.
\]
This measure is invariant under \((r,t)=\rho_\delta(u,v)\).
The transformed support is contained in
\(\{1\le u\le C_0,\ 0<v\le C_0\}\) and its transpose.
The integral over this fixed set is finite, since
\[
\int_1^{C_0}\!\int_0^{C_0}
\frac{uv}{(u^2+v^2)^2}\,dv\,du
\le\frac12\int_1^{C_0}\frac{du}{u}<\infty.
\]
All three bounds are independent of \(\delta\). The triangle
inequality in \(\mathfrak S_2\) proves the lemma.
\end{proof}

\medskip
\noindent\textit{The convolution in logarithmic coordinates.}
\reviewedit{Homogeneity gives an exact convolution representation
in logarithmic radial coordinates.}
Write \(\xi=e^t\omega\), \(\eta=e^{t'}\nu\), and use the unitary map
\[
\mathcal M:L^2(\R^3,d\xi)
\longrightarrow L^2(\R\times\Sph^2,dt\,d\omega),
\qquad (\mathcal Mf)(t,\omega)=e^{3t/2}f(e^t\omega)
\quad\text{a.e.}
\]
Here \(d\omega\) is surface measure on \(\Sph^2\), with total
area \(4\pi\); unitarity follows from
\(d\xi=e^{3t}\,dt\,d\omega\).
For the untruncated homogeneous kernel,
\[
e^{3(t+t')/2}s_K^{+,\mathrm{hom}}(e^t\omega,e^{t'}\nu)
=\cS_K^+(t-t';\omega,\nu),
\]
where
\begin{equation}
\label{eq:upper-log-kernel}
\cS_K^+(u;\omega,\nu)
=-
\frac{(2\pi)^{-3}}{\sqrt{\det R_K^+}}
\frac{[\vartheta_K^+(\omega)\vartheta_K^+(\nu)]^{-1/2}}
{e^u\langle A_K^+\omega,\omega\rangle
+2\langle D_K^+\nu,\omega\rangle
+e^{-u}\langle A_K^+\nu,\nu\rangle}.
\end{equation}
Let \(\cS_K^+(u)\) be the integral operator on
\(\frh_{\mathrm{ang}}=L^2(\Sph^2,d\omega)\) with kernel
\(\cS_K^+(u;\omega,\nu)\), relative to \(d\nu\).
Symmetry of \(D_K^+\) implies
\(\cS_K^+(-u;\omega,\nu)=\cS_K^+(u;\nu,\omega)\).
Thus
\(\cS_K^+(-u)=\cS_K^+(u)^*\), and
\begin{equation}
\label{eq:upper-cosh-bound}
\|\cS_K^+(u)\|_{\mathfrak S_2(\frh_{\mathrm{ang}})}\le\frac{C}{\cosh u}.
\end{equation}
Indeed, the denominator in \eqref{eq:upper-log-kernel} is at
least \(c_K(e^u+e^{-u})\). Define
\begin{equation}
\label{eq:upper-Fourier-family}
\widehat{\cS}_K^+(s)
=\int_{\R}e^{-isu}\cS_K^+(u)\,du
\end{equation}
in the Bochner sense in \(\mathfrak S_2(\frh_{\mathrm{ang}})\). For real \(s\),
this is compact and self-adjoint, continuous in \(s\), and tends to
zero in \(\mathfrak S_2\) as \(|s|\to\infty\).
With this Fourier convention, the Hilbert-space-valued
Plancherel identity reads
\[
\int_\R\|\widehat{\cS}_K^+(s)\|_{\mathfrak S_2}^2\,ds
=2\pi\int_\R\|\cS_K^+(u)\|_{\mathfrak S_2}^2\,du<\infty.
\]
This Fourier family is the symbol appearing in
Definition~\ref{def:upper-U}. Plancherel's identity and
\(n(\beta,T)\le\beta^{-2}\|T\|_{\mathfrak S_2}^2\) show that
the integral in \eqref{eq:upper-U-alpha} is finite for every
\(\beta>0\).
To connect this coefficient with the local operator, for
\(L>0\) define the self-adjoint Hilbert--Schmidt operator
\(\mathsf S_{K,L}^+\) on
\(L^2((0,L);\frh_{\mathrm{ang}})\) by
\[
(\mathsf S_{K,L}^+f)(t)
=\int_0^L\cS_K^+(t-t')f(t')\,dt'
\quad\text{for a.e. }t\in(0,L).
\]
The full kernel of \(\mathcal M\mathsf H_{K,\delta}^+\mathcal M^{-1}\)
is
\[
\mathbf1_{(0,L_\delta)}(t)
\cS_K^+(t-t';\omega,\nu)
\mathbf1_{(0,L_\delta)}(t').
\]
Hence this operator is the zero extension of
\(\mathsf S_{K,L_\delta}^+\) to \(L^2(\R;\frh_{\mathrm{ang}})\), where
\begin{equation}
\label{eq:upper-log-length}
L_\delta=\log \rho_\delta=\log c_0+\tfrac12|\log\delta|.
\end{equation}
\subsection{Logarithmic eigenvalue-counting asymptotics}
\label{subsec:upper-counting}

\begin{lemma}[Regularity of the counting density]
\label{lem:counting-density-regularity}
For each fixed \(K\in\cO_+\), the function \(\cU_K^+\) is
finite and continuous on \((0,\infty)\).
\end{lemma}

\begin{proof}
Write \(C(u)=\cS_K^+(u)\) and
\(\widehat C(s)=\widehat{\cS}_K^+(s)\). The estimate
\eqref{eq:upper-cosh-bound} gives
\[
\|C(u)\|_{\mathfrak S_2}\le b_K e^{-|u|}.
\]
The exponential bound implies that \(\widehat C\) extends
holomorphically, with values in \(\mathfrak S_2(\frh_{\mathrm{ang}})\),
to the strip \(|\operatorname{Im}z|<1\). Indeed, on
\(|\operatorname{Im}z|\le a<1\), the integrand and its
derivative are bounded in that norm by
\(b_K e^{-(1-a)|u|}\) and
\(b_K|u|e^{-(1-a)|u|}\), respectively.
Fix \(\beta>0\). Since \(\|\widehat C(s)\|\to0\) as
real \(|s|\to\infty\), the operator
\(I-\beta^{-1}\widehat C(z)\) is invertible at some point
of this strip. The analytic Fredholm theorem therefore implies
that its noninvertibility points are discrete. In particular,
\[
Z_\beta=\{s\in\R:\beta\in\spec\widehat C(s)\}
\]
is finite: it has no real accumulation point and is bounded
by the decay of \(\widehat C(s)\).
At any \(s\notin Z_\beta\), compactness gives a spectral gap
around the positive level \(\beta\); norm continuity of
\(\widehat C\) therefore makes the count locally constant.
Moreover,
\[
n(\beta,\widehat C(s))
\le\beta^{-2}\|\widehat C(s)\|_{\mathfrak S_2}^2
\le\beta^{-2}\left(\int_\R\|C(u)\|_{\mathfrak S_2}\,du\right)^2.
\]
Thus the counting function is bounded, vanishes for large
\(|s|\), and has at most finitely many discontinuities.
It is consequently Riemann integrable.
For every \(s\notin Z_\beta\), the spectral gap also makes
\(n(\beta',\widehat C(s))\) constant for \(\beta'\) near \(\beta\).
For \(\beta'\ge\beta/2\),
\[
n(\beta',\widehat C(s))
\le4\beta^{-2}\|\widehat C(s)\|_{\mathfrak S_2}^2.
\]
Plancherel's identity and dominated convergence therefore prove
continuity of \(\cU_K^+\) at every positive level.
\end{proof}

\begin{lemma}[Periodization and Fourier diagonalization]
\label{lem:periodization}
For each fixed \(K\in\cO_+\),
\[
\lim_{L\to\infty}\frac{n(\beta,\mathsf S_{K,L}^+)}L
=\frac1{2\pi}\int_\R n(\beta,\widehat C(s))\,ds,
\qquad\beta>0,
\]
where \(\widehat C(s)=\widehat{\cS}_K^+(s)\).
\end{lemma}

\begin{proof}
For \(L\ge1\), let \(P_L\) act on the same space as
\(\mathsf S_{K,L}^+\), with periodic kernel
\[
C_L^{\mathrm{per}}(u)=\sum_{j\in\mathbb Z}C(u+jL),
\qquad C(u)=\cS_K^+(u).
\]
The series converges in Hilbert--Schmidt norm, uniformly on
\([-L,L]\) for each fixed \(L\). It defines a self-adjoint
Hilbert--Schmidt integral operator. For \(|u|\le L\),
\[
\begin{aligned}
\|C_L^{\mathrm{per}}(u)-C(u)\|_{\mathfrak S_2}
&\le b_K\frac{e^{-(L-u)}+e^{-(L+u)}}{1-e^{-L}}\\
&\le b'_K e^{-(L-|u|)}.
\end{aligned}
\]
Using \(u=t-t'\), we obtain
\[
\begin{aligned}
\|P_L-\mathsf S_{K,L}^+\|_{\mathfrak S_2}^2
&\le (b'_K)^2\int_{-L}^L
(L-|u|)e^{-2(L-|u|)}\,du\\
&\le2(b'_K)^2\int_0^\infty v e^{-2v}\,dv<\infty,
\end{aligned}
\]
uniformly in \(L\ge1\). Thus making the interval periodic
introduces only a uniformly bounded Hilbert--Schmidt error.
For \(j\in\mathbb Z\), put \(s_j=2\pi j/L\). If
\(v\in\frh_{\mathrm{ang}}\), periodicity and \(e^{is_jL}=1\) give
\[
P_L\bigl(L^{-1/2}e^{is_jt}v\bigr)
=L^{-1/2}e^{is_jt}\widehat C(s_j)v.
\]
The Fourier series transform on \((0,L)\) therefore yields
\[
P_L\simeq\bigoplus_{j\in\mathbb Z}\widehat C(2\pi j/L),
\qquad
n(\beta,P_L)=\sum_{j\in\mathbb Z}
n(\beta,\widehat C(2\pi j/L)).
\]
\notationedit{Here \(\simeq\) denotes unitary equivalence.}
\reviewedit{Since \(\|\widehat C(s)\|\to0\) as \(|s|\to\infty\),
only finitely many summands are nonzero for each \(L\) and
\(\beta>0\). Lemma~\ref{lem:counting-density-regularity} shows that
\(s\mapsto n(\beta,\widehat C(s))\) is Riemann integrable
with compact support. The Riemann sums with mesh \(2\pi/L\)
therefore give}
\[
\lim_{L\to\infty}\frac{n(\beta,P_L)}L
=\frac1{2\pi}\int_\R n(\beta,\widehat C(s))\,ds
=2\cU_K^+(\beta).
\]
Extend both compact self-adjoint operators by zero to
\(L^2(\R;\frh_{\mathrm{ang}})\). Apply
Remark~\ref{rem:upper-HS-stability} with principal operator \(P_L\),
perturbation \(\mathsf S_{K,L}^+-P_L\), and small parameter and
normalizing factor \(L^{-1}\). The uniform Hilbert--Schmidt bound
was proved above, and the limit function \(2\cU_K^+\) is finite
and continuous by Lemma~\ref{lem:counting-density-regularity}.
Thus the same limit holds for \(\mathsf S_{K,L}^+\).
\end{proof}

\begin{proposition}[Logarithmic model asymptotics]
\label{prop:upper-model-asymptotic}
For each fixed \(K\in\cO_+\), the function
\(\cU_K^+\) is finite and continuous on \((0,\infty)\).
For every \(\alpha>0\),
\begin{equation}
\label{eq:upper-model-counting}
n(\alpha,S_{K,\delta}^+)
=\cU_K^+(\alpha)|\log\delta|
+o(|\log\delta|),
\qquad \delta\downarrow0.
\end{equation}
\end{proposition}

\begin{proof}
By Lemma~\ref{lem:periodization} and
\(L_\delta/|\log\delta|\to1/2\),
\begin{equation}
\label{eq:upper-homogeneous-counting}
\lim_{\delta\downarrow0}
\frac{n(\beta,\mathsf S_{K,L_\delta}^+)}{|\log\delta|}
=\frac1{4\pi}\int_\R n(\beta,\widehat C(s))\,ds
=\cU_K^+(\beta).
\end{equation}
\reviewedit{On the fixed space \(L^2(\R;\frh_{\mathrm{ang}})\),}
\[
\mathcal M\mathsf H_{K,\delta}^+\mathcal M^{-1}
=\mathsf S_{K,L_\delta}^+\oplus0,
\]
where \(\oplus0\) denotes extension by zero outside
\((0,L_\delta)\). Lemma~\ref{lem:upper-homogeneous-reduction}
therefore gives
\[
\sup_{0<\delta\le\delta_0}
\bigl\|\mathcal M\widetilde S_{K,\delta}^+\mathcal M^{-1}
-(\mathsf S_{K,L_\delta}^+\oplus0)\bigr\|_{\mathfrak S_2}
<\infty.
\]
Both operators in this difference are compact and self-adjoint.
These properties follow from the real symmetric square-integrable
kernels and are preserved by unitary conjugation and extension
by zero. Apply Remark~\ref{rem:upper-HS-stability}
with principal operator \(\mathsf S_{K,L_\delta}^+\oplus0\)
and perturbation equal to the displayed difference.
The uniform Hilbert--Schmidt bound has just been proved;
\(f(\delta)=|\log\delta|^{-1}>0\) tends to zero since
\(0<\delta\le\delta_0<1\). The principal counting limits are
\eqref{eq:upper-homogeneous-counting}, and their finiteness and
continuity follow from Lemma~\ref{lem:counting-density-regularity}.
Since unitary conjugation and extension by zero preserve counts
at positive levels, the conclusion is \eqref{eq:upper-model-counting}.
\end{proof}

\begin{corollary}[Birman--Schwinger counting asymptotics]
\label{cor:upper-model-HS-stability}
For every fixed \(K\in\cO_+\) and \(\alpha>0\),
\begin{equation}
\label{eq:upper-HS-stability}
\lim_{\delta\downarrow0}
\frac{n(\alpha,T_K^+(E_++\delta))}{|\log\delta|}
=\cU_K^+(\alpha).
\end{equation}
\end{corollary}

\begin{proof}
In Lemma~\ref{lem:upper-HS-stability}, take
\[
A_\delta=S_{K,\delta}^+,
\qquad B_\delta=\cR_{K,\delta}^+,
\qquad f(\delta)=|\log\delta|^{-1}.
\]
The kernel \eqref{eq:upper-singular-kernel} is real, symmetric
and square integrable, so \(A_\delta\) is compact and self-adjoint.
On every compact subinterval of \((0,\delta_0]\), its denominators
are uniformly bounded away from zero by
\eqref{eq:upper-d-positive-bound}; dominated convergence therefore
gives continuity in Hilbert--Schmidt norm, hence in operator norm.
Proposition~\ref{prop:upper-HS-reduction} gives the same continuity
for the self-adjoint family \(B_\delta\), including at
\(\delta=0\), with a Hilbert--Schmidt, hence compact, limit.
The choice \(0<\delta_0<1\) ensures that \(f\) is positive and
tends to zero. Proposition~\ref{prop:upper-model-asymptotic}
provides the counting limits for every positive level and a finite
continuous limit function. Thus all hypotheses of the lemma hold,
and \eqref{eq:upper-HS-stability} follows.
\end{proof}

\subsection{Positivity of the Efimov coefficient}
\label{subsec:upper-symmetric-point}
At \(K_0=\pi\mathbf1\), the logarithmic model is rotationally
invariant, so its Fourier family decomposes into spherical harmonics,
as in the continuum analysis of Basti and Teta~\cite{BastiTeta2017}.
For the symmetric-point calculation we allow any \(\gamma>0\).
We identify the mass threshold in the degree-one sector and then,
for \(\gamma>\gamma_*\), use continuity in \(K\).

\subsubsection{Fermionic mass threshold at the symmetric point}
At \(K_0=\pi\mathbf1\), set \(q=q_\gamma=\gamma/(1+\gamma)\).
Equations~\eqref{eq:upper-matrices-at-pi},
\eqref{eq:upper-cK}, and \(\ell_{K_0}^+=0\) give
\[
A_{K_0}^+=\frac{I}{1-q^2},\qquad
D_{K_0}^+=\frac{qI}{1-q^2},\qquad
\vartheta_{K_0}^+=\frac{1}{4\pi(1+\gamma)^{3/2}}.
\]
Substitution into \eqref{eq:upper-log-kernel}, with surface
measure of total area \(4\pi\), reduces the kernel to
\begin{equation}
\label{eq:upper-pi-kernel}
\cS_{\pi\mathbf 1}^+(u;\omega,\nu)
=-
\frac1{4\pi^2\sqrt{1-q_\gamma^2}}
\frac1{\cosh u+q_\gamma\,\omega\cdot\nu}.
\end{equation}
This is the analogue of the rotationally invariant kernel in
\cite{ALM2004} and \edit{\cite[Section~9]{ALD2007}}. The negative sign
comes from exchange antisymmetry.
For real \(s\), rotational invariance implies that
\(\widehat{\cS}_{K_0}^+(s)\) acts as a scalar
\(\lambda_l(s;q)\) on the \((2l+1)\)-dimensional space of
spherical harmonics of degree \(l\). With \(P_l(1)=1\), the
Funk--Hecke formula and our Fourier convention give
\begin{equation}
\label{eq:upper-angular-multiplier}
\lambda_l(s;q)
=-\frac{1}{2\pi\sqrt{1-q^2}}
\int_{\R}e^{-isu}\int_{-1}^1
\frac{P_l(t)}{\cosh u+qt}\,dt\,du.
\end{equation}
Both integrations are absolutely convergent. At \(s=0\),
\[
\int_{\R}\frac{du}{\cosh u+a}
=\frac{2\arccos a}{\sqrt{1-a^2}},
\qquad |a|<1,
\]
so \(P_1(t)=t\) yields
\[
\begin{aligned}
\lambda_1(0;q)
&=-\frac{1}{\pi\sqrt{1-q^2}}
\int_{-1}^1\frac{t\,\arccos(qt)}{\sqrt{1-q^2t^2}}\,dt\\
&=\frac{2}{\pi\sqrt{1-q^2}}
\int_0^1\frac{t\,\arcsin(qt)}{\sqrt{1-q^2t^2}}\,dt
=:\Lambda_1(q).
\end{aligned}
\]
Integration by parts gives the explicit expression
\begin{equation}
\label{eq:upper-Lambda-one}
\Lambda_1(q)
=\frac{2}{\pi q^2\sqrt{1-q^2}}
\bigl[q-\sqrt{1-q^2}\arcsin q\bigr],
\qquad 0<q<1.
\end{equation}
In the integral representation of \(\Lambda_1(q)\), the
integrand is positive for \(t>0\) and increases strictly with
\(q\); the positive prefactor also increases. Moreover,
\eqref{eq:upper-Lambda-one} gives
\(\Lambda_1(q)\to0\) as \(q\downarrow0\) and
\(\Lambda_1(q)\to\infty\) as \(q\uparrow1\).
Consequently, the solution \(q_*\) of \eqref{eq:upper-qstar}
is unique. Moreover,
\[
\Lambda_1(1/2)=\frac{8}{\pi\sqrt3}-\frac43<1,
\]
so \(q_*>1/2\) and \(\gamma_*>1\).
Numerically, \(\notationedit{q_*\approx0.931539512}\), giving the value of
\(\gamma_*\) stated in \eqref{eq:upper-gammastar}.
The following comparison shows that the degree-one sector is
the first angular sector to become supercritical.

\begin{lemma}[Comparison of angular sectors]
\label{lem:upper-angular-comparison}
For \(0<q<1\) and \(s\in\R\), the multipliers in
\eqref{eq:upper-angular-multiplier} satisfy
\[
\lambda_l(s;q)<0\quad\text{if }l\text{ is even},\qquad
0<\lambda_l(s;q)\le\lambda_l(0;q)\le\Lambda_1(q)
\quad\text{if }l\text{ is odd}.
\]
Consequently, at \(K_0=\pi\mathbf1\), only odd angular
sectors contribute to eigenvalue counts at positive levels.
\end{lemma}

\begin{proof}
Expand
\[
\frac1{\cosh u+qt}
=\sum_{n=0}^{\infty}(-q)^nt^n(\cosh u)^{-n-1}.
\]
For fixed \(q<1\), the sum of absolute values is bounded by
\((1-q)^{-1}(\cosh u)^{-1}\), so termwise integration in
\eqref{eq:upper-angular-multiplier} is justified. Put
\[
b_{nl}=\int_{-1}^1 t^nP_l(t)\,dt,\qquad
F_n(s)=\int_{\R}e^{-isu}(\cosh u)^{-n-1}\,du.
\]
Rodrigues' formula and integration by parts give
\[
b_{nl}=\frac{2n!}{(n-l)!!(n+l+1)!!}>0
\quad(n\ge l,\ n-l\text{ even}),
\]
and \(b_{nl}=0\) otherwise; here \(0!!=1\).
The substitution \(v=e^{2u}\) in \(F_n\), followed by the
beta integral, gives
\[
F_n(s)=\frac{2^n}{\Gamma(n+1)}
\left|\Gamma\!\left(\frac{n+1+is}{2}\right)\right|^2>0.
\]
The defining integral also gives
\[
|F_n(s)|\le\int_\R |e^{-isu}|(\cosh u)^{-n-1}\,du=F_n(0).
\]
Thus
\[
\lambda_l(s;q)
=\frac{(-1)^{l+1}}{2\pi\sqrt{1-q^2}}
\sum_{\substack{n\ge l\\n-l\ \mathrm{even}}}
q^n b_{nl}F_n(s).
\]
This proves the signs and the bound by \(\lambda_l(0;q)\)
for odd \(l\). For \(n\ge l+2\) of the same parity,
\[
\frac{b_{n,l+2}}{b_{nl}}=\frac{n-l}{n+l+3}<1.
\]
Hence \(b_{nl}\le b_{n1}\) for odd \(l\) and \(n\ge l\).
Comparison of the positive series proves
\(\lambda_l(0;q)\le\lambda_1(0;q)=\Lambda_1(q)\).
\end{proof}

It follows that the homogeneous coefficient at the symmetric
point is positive precisely when \(\gamma>\gamma_*\): if
\(\gamma\le\gamma_*\), all multipliers are at most one, so the
strict level-one count vanishes for every real \(s\);
if \(\gamma>\gamma_*\), continuity in \(s\) keeps
\(\lambda_1(s;q_\gamma)>1\) on an interval containing zero.

\begin{proposition}[Supercriticality near the symmetric point]
\label{prop:upper-supercriticality}
If \(\gamma>\gamma_*\), then \(\cU_+(\pi\mathbf 1)>0\),
and after shrinking \(\cO_+\),
\[
\cU_+(K)>0,
\qquad K\in\cO_+.
\]
\end{proposition}

\begin{proof}
At \(K_0=\pi\mathbf1\), choose a normalized spherical harmonic
\(Y\) of degree one. Then
\[
\langle\widehat{\cS}_{K_0}^+(0)Y,Y\rangle
=\Lambda_1(q_\gamma)>1.
\]
The matrix coefficients and \(\vartheta_K^+\) depend continuously
on \(K\). On a sufficiently small closed neighbourhood of \(K_0\),
positive definiteness and strict transversality give a common bound
\[
\|\cS_K^+(u)\|_{\mathfrak S_2}\le C e^{-|u|}.
\]
Pointwise continuity of the angular kernels and dominated convergence
therefore yield
\[
\int_\R\|\cS_K^+(u)-\cS_{K_0}^+(u)\|_{\mathfrak S_2}\,du
\longrightarrow0\quad(K\to K_0).
\]
In particular, the Fourier symbols converge in operator norm,
uniformly for real \(s\). Continuity in \(s\) now gives numbers
\(s_0,\epsilon_0>0\) and a neighbourhood \(\cO_+\) of \(K_0\)
such that
\[
\langle\widehat{\cS}_K^+(s)Y,Y\rangle\ge1+\epsilon_0,
\qquad K\in\cO_+,\quad |s|\le s_0.
\]
The variational principle implies
\(n(1,\widehat{\cS}_K^+(s))\ge1\) on this interval. Hence
\(\cU_+(K)\ge s_0/(2\pi)>0\).
\end{proof}

\subsection{Stability under the fermion--fermion interaction}
\label{subsec:fermion-stability}
Throughout this subsection, fix
\[
\gamma>\gamma_*,\qquad K\in\cO_+,\qquad \mu=\mu_{\mathrm c}^+,
\]
\reviewedit{and let \(\lambda\ge0\). We study the compact part of the
full channel system for every such coupling and relate its
logarithmic counting coefficient to the discrete eigenvalue counts
in the weak and separated regimes.}

\subsubsection{Fermion multipliers and the channel resolvent}
Recall the factorization \eqref{eq:V12-factorization}.
Independently of the chosen local lifts, the odd form factors satisfy
\begin{equation}
\label{eq:odd-form-factor-vanishing}
|\phi_{j,p+q}(p)|=\sqrt2\left|\sin\frac{p_j-q_j}{2}\right|,
\qquad\phi_{j,2k_K^+}(k_K^+)=0.
\end{equation}
\reviewedit{For every \(z>E_+\), put
\(R_0(K,z)=(H_\gamma^{\mathrm{free}}(K)-z)^{-1}\)} and
\begin{equation}
\label{eq:fermion-channel-D}
D_{\lambda,K}(z)=I+\lambda\Phi_KR_0(K,z)\Phi_K^*.
\end{equation}
\reviewedit{If \(z\) is in the fermion-channel resolvent set,
the factorization implies that \(D_{\lambda,K}(z)\) is invertible,
and the channel resolvent formula is}
\begin{equation}
\label{eq:fermion-channel-resolvent}
(H_{12,\lambda}^{\mathrm{ch}}(K)-z)^{-1}
=R_0(K,z)-\lambda R_0(K,z)\Phi_K^*D_{\lambda,K}(z)^{-1}
\Phi_KR_0(K,z).
\end{equation}
By \eqref{eq:fermion-critical-b}--\eqref{eq:fermion-critical-d},
\(D_{\lambda,K}(E_++\delta)\) is multiplication by
\begin{equation}
\label{eq:fermion-D-diagonal}
d_{\lambda,K}(r,E_++\delta)
=\operatorname{diag}_{1\le j\le3}\{1-\lambda b_j(K,r,\delta)\}.
\end{equation}
\revblock{Joint continuity from
Lemma~\ref{lem:fermion-critical-regularity} and compactness of
\(\T^3\) give the operator-norm limit}
\begin{equation}
\label{eq:upper-regular-lambda}
D_{\lambda,K}(E_+)=\lim_{\delta\downarrow0}D_{\lambda,K}(E_++\delta).
\end{equation}
For \(0\le\lambda<\lambda_{\mathrm{crit}}\), this limit is at least
\((1-\lambda\beta)I>0\); for \(\lambda>\lambda_{\mathrm{sep}}\),
it is uniformly negative. \revblock{For these noncritical couplings,}
Proposition~\ref{prop:fermion-all-couplings} and continuity on the
compact spectator torus therefore give
\begin{equation}
\label{eq:fermion-D-uniform-inverse}
\sup_{E_+\le z\le\zeta}\|D_{\lambda,K}(z)^{-1}\|<\infty
\end{equation}
for any fixed interior energy \(\zeta\) in the corresponding
right-hand channel resolvent interval. The inverse is norm-continuous
there. \reviewedit{Among the weak and separated couplings, only
\(\lambda=\lambda_{\mathrm{crit}}\) requires a separate threshold
estimate, namely \eqref{eq:fermion-critical-bound}.}

With \(D_z=M_{\sqrt{\Delta_K^+(\cdot,z)}}\) as in
Subsection~\ref{subsec:upper-BS}, define
\begin{equation}
\label{eq:fermion-mixed-transition}
\cW_K(z)=\Phi_KR_0(K,z)\Psi^*D_z^{-1}.
\end{equation}
The formula \eqref{eq:mixed-interaction-factorization} and the change
\(p\leftrightarrow K-r-p\) give, almost everywhere,
\begin{equation}
\label{eq:fermion-mixed-transition-kernel}
(\cW_K(z)g)_j(r)
=\int_{\T^3}\frac{\phi_{j,K-r}(p)g(p)\,\eta(dp)}
{[z-\E_\gamma(K;p,K-r-p)]\Delta_K^+(p,z)^{1/2}}.
\end{equation}
The positive sign follows from the negative free resolvent and
antisymmetry of the form factor.

\begin{lemma}[Threshold regularity of the fermion channel]
\label{lem:fermion-channel-HS}
Fix \(\gamma>\gamma_*\) and \(K\in\cO_+\). For
\(z>E_{+,\gamma}(K)\), the operator \(\cW_K(z)\) is
Hilbert--Schmidt from \(L^2(\T^3,\eta)\) to
\(L^2(\T^3,\eta;\C^3)\), and depends continuously on \(z\)
in the Hilbert--Schmidt norm. There exists a Hilbert--Schmidt
operator \(\cW_K(E_{+,\gamma}(K))\) such that
\[
\lim_{\delta\downarrow0}
\|\cW_K(E_{+,\gamma}(K)+\delta)
-\cW_K(E_{+,\gamma}(K))\|_{\mathfrak S_2}=0.
\]
\end{lemma}

\begin{proof}
On each compact subinterval of \((E_+,\infty)\), the kernel
\eqref{eq:fermion-mixed-transition-kernel} is bounded uniformly
and depends pointwise continuously on \(z\). Dominated convergence
gives Hilbert--Schmidt continuity there.
For the threshold limit, put \(q=K-r-p\), \(p=k_++x\),
\(q=k_++y\), and \(v=r-\rho_+\), where \(\rho_+=K-2k_+\).
Then \(y=-x-v\); the invertible map \((x,v)\mapsto(x,y)\)
has Jacobian of absolute value one, so \(x,y\) are independent
integration coordinates.
By \eqref{eq:odd-form-factor-vanishing} and
Subsection~\ref{subsec:efimov-local-estimates}, its absolute value
is bounded locally by
\begin{equation}
\label{eq:fermion-transition-majorant}
\frac{C(|x|+|y|)}
{(\delta+|x|^2+|y|^2)(\delta+|x|^2)^{1/4}}.
\end{equation}
The squared kernel consequently has the uniform majorant
\begin{equation}
\label{eq:fermion-transition-square-majorant}
\frac{C}{(|x|^2+|y|^2)|x|},
\end{equation}
which is locally integrable: its radial integral is bounded by
\(C\int_0^b\int_0^b rt^2/(r^2+t^2)\,dt\,dr<\infty\).
\revblock{
Fix a sufficiently small neighbourhood
\(\mathcal N\) of \((\rho_+,k_+)\) in the variables \((r,p)\),
on which the preceding local estimates hold.
The continuous function
\[
a_K(r,p)=E_+-\E_\gamma(K;p,K-r-p)
\]
is nonnegative and vanishes only at \((\rho_+,k_+)\).
Compactness therefore gives a constant \(c_{\mathcal N}>0\)
such that
\[
a_K(r,p)\ge c_{\mathcal N},\qquad
(r,p)\in(\T^3\times\T^3)\setminus\mathcal N.
\]
The same lower bound holds for \(a_K(r,p)+\delta\),
uniformly in \(\delta\ge0\).

Let \(W_j(r,p;\delta)\) denote the component kernel of
\(\cW_K(E_++\delta)\). Since the form factors are bounded,
the preceding lower bound, monotonicity of \(\Delta_K^+\)
in the energy, and its threshold estimates yield
\[
\begin{aligned}
|W_j(r,p;\delta)|^2
&\le \frac{C}{\Delta_K^+(p,E_++\delta)}\\
&\le \frac{C}{\Delta_K^+(p,E_+)}
\le C\bigl[1+d_{\T^3}(p,k_+)^{-1}\bigr]
\end{aligned}
\]
for almost every \((r,p)\notin\mathcal N\) and
\(0<\delta\le\delta_0\), with \(C\) independent of \(\delta\).
The last function is integrable with respect to
\(\eta(dr)\eta(dp)\). Together with the local majorant
\eqref{eq:fermion-transition-square-majorant}, this gives
an integrable bound for the squared kernels on the whole
product torus, independent of \(\delta\).
}
Pointwise convergence and dominated convergence prove the assertion.
\end{proof}

\reviewedit{For the following weighted estimates and channel counting
identities, assume either
\(0\le\lambda\le\lambda_{\mathrm{crit}}(K,\gamma)\) or
\(\lambda>\lambda_{\mathrm{sep}}(K,\gamma)\).}

\subsubsection{A weighted estimate including the critical coupling}
For \(z=E_++\delta\) in the right-hand channel resolvent interval,
define
\begin{equation}
\label{eq:fermion-weighted-transition}
L_{\lambda,K}(\delta)
=|D_{\lambda,K}(E_++\delta)|^{-1/2}\cW_K(E_++\delta),
\qquad
J_{\lambda,K}(\delta)=\operatorname{sgn}D_{\lambda,K}(E_++\delta).
\end{equation}
The sign is \(I\) in the closed weak-coupling interval, and \(-I\)
in the separated regime for \(0<\delta<\tau_{12}^+-E_+\).

\begin{lemma}[Hilbert--Schmidt continuity of the weighted operator]
\label{lem:fermion-weighted-HS}
For each fixed \(K\) as above and each
\(0\leq\lambda\leq\lambda_{\mathrm{crit}}(K,\gamma)\) or
\(\lambda>\lambda_{\mathrm{sep}}(K,\gamma)\), there is a Hilbert--Schmidt
operator \(L_{\lambda,K}(0)\) such that
\begin{equation}
\label{eq:fermion-weighted-limit}
\|L_{\lambda,K}(\delta)-L_{\lambda,K}(0)\|_{\mathfrak S_2}
\longrightarrow0,\qquad\delta\downarrow0.
\end{equation}
\reviewedit{At \(\lambda=\lambda_{\mathrm{crit}}\),
\(L_{\lambda,K}(0)\) is defined by the pointwise almost-everywhere
limit of the kernels of \(L_{\lambda,K}(\delta)\).}
\end{lemma}

\begin{proof}
If \(0\le\lambda<\lambda_{\mathrm{crit}}\) or
\(\lambda>\lambda_{\mathrm{sep}}\), the assertion follows directly
from \eqref{eq:fermion-D-uniform-inverse}, norm continuity of
\(|D_{\lambda,K}|^{-1/2}\), and Lemma~\ref{lem:fermion-channel-HS}.
It remains to treat \(\lambda=\lambda_{\mathrm{crit}}\).
Write \(k_+=k_K^+\), \(\rho_+=K-2k_+\). The component kernel is
\[
L_j(r,p;\delta)
=\frac{\phi_{j,K-r}(p)}
{d_{\lambda,j}(K,r,\delta)^{1/2}
 [E_++\delta-\E_\gamma(K;p,K-r-p)]
 \Delta_K^+(p,E_++\delta)^{1/2}}.
\]
\revblock{Let \(W_j(r,p;\delta)\) denote the \(j\)-th component
kernel of \(\cW_K(E_++\delta)\), given by
\eqref{eq:fermion-mixed-transition-kernel}.}

Put \(x=p-k_+\), \(v=r-\rho_+\), and \(y=-x-v\).
The volume-preserving change used in
Lemma~\ref{lem:fermion-channel-HS} and
\(|x|^2+|y|^2\asymp|x|^2+|v|^2\) turn
\eqref{eq:fermion-transition-square-majorant} into the locally
integrable bound
\[
|W_j(r,p;\delta)|^2\le\frac{C}{(|x|^2+|v|^2)|x|}.
\]
\revblock{On the complement of a fixed sufficiently small
neighbourhood of \((\rho_+,k_+)\) in the variables \((r,p)\),
the unweighted squared kernels are bounded by
\(C(1+d_{\T^3}(p,k_+)^{-1})\), uniformly for
\(0<\delta\leq\delta_0\), as in the proof of
Lemma~\ref{lem:fermion-channel-HS}.}

Fix a component \(j\). If
\(\beta>b_j(K,r_j(K),0)\), its critical multiplier is uniformly
positive at the threshold, and these unweighted estimates suffice.
Suppose now that \(j\) is active, namely
\(\beta=b_j(K,r_j(K),0)\).

If \(r_j(K)\ne\rho_+\), choose disjoint small neighbourhoods
of the two points. By Lemma~\ref{lem:fermion-critical-maxima},
the maximum at \(r_j(K)\) is nondegenerate, and
\eqref{eq:fermion-critical-bound} gives
\(d_{\lambda_{\mathrm{crit}},j}(K,r,\delta)
\ge c\,d_{\T^3}(r,r_j(K))^2\).
Near \(r_j(K)\), the free denominator is uniformly positive
in \(p\) and \(\delta\ge0\). The global determinant estimate
therefore yields, almost everywhere,
\[
|L_j(r,p;\delta)|^2
\leq
\frac{C[1+d_{\T^3}(p,k_+)^{-1}]}
{d_{\T^3}(r,r_j(K))^2},
\]
an integrable product majorant. On the complement of that
neighbourhood, the critical multiplier is bounded below by a
positive constant, so the unweighted estimates apply.

If \(r_j(K)=\rho_+\), the singular points coincide. By
\eqref{eq:fermion-critical-bound} and the estimates of Subsection~\ref{subsec:efimov-local-estimates},
\begin{equation}
\label{eq:fermion-weighted-majorant}
|L_j(r,p;\delta)|^2
\leq\frac{C}
{(\delta+|v|^2)(\delta+|x|^2+|v|^2)(\delta+|x|^2)^{1/2}}
\leq\frac{C}{|v|^2(|x|^2+|v|^2)|x|}.
\end{equation}
The last function is locally integrable. Indeed, with
\(t=|x|\) and \(s=|v|\), its integral over a product of balls
is bounded by
\begin{equation}
\label{eq:fermion-weighted-integrability}
C\int_0^b\int_0^b\frac{t}{t^2+s^2}\,ds\,dt
=C\int_0^b\arctan(b/t)\,dt<\infty.
\end{equation}
Outside the joint neighbourhood, the product
\[
C[1+d_{\T^3}(r,\rho_+)^{-2}]
 [1+d_{\T^3}(p,k_+)^{-1}]
\]
is a sufficient integrable majorant.
\revblock{
Thus every component admits an integrable squared-kernel
majorant independent of \(\delta\). For almost every \((r,p)\),
the kernel \(L_j(r,p;\delta)\) has a finite limit
\(L_j(r,p;0)\) as \(\delta\downarrow0\).
The resulting kernel defines a Hilbert--Schmidt operator
\(L_{\lambda,K}(0)\). If \(G_j\) denotes the corresponding
majorant, then
\[
|L_j(r,p;\delta)-L_j(r,p;0)|^2\le 4G_j(r,p)
\]
almost everywhere. Dominated convergence therefore gives
\[
\begin{aligned}
&\|L_{\lambda,K}(\delta)-L_{\lambda,K}(0)\|_{\mathfrak S_2}^2\\
&\qquad=\sum_{j=1}^3\int_{\T^3}\int_{\T^3}
|L_j(r,p;\delta)-L_j(r,p;0)|^2
\,\eta(dp)\eta(dr)\longrightarrow0.
\end{aligned}
\]
This proves \eqref{eq:fermion-weighted-limit}.
}
\end{proof}

Define the full fermion correction by
\begin{equation}
\label{eq:fermion-channel-correction}
\begin{aligned}
\cR_{\lambda,K}^+(E_++\delta)
&=2\mu\lambda\,\cW_K(E_++\delta)^*
D_{\lambda,K}(E_++\delta)^{-1}\cW_K(E_++\delta)\\
&=2\mu\lambda\,L_{\lambda,K}(\delta)^*
J_{\lambda,K}(\delta)L_{\lambda,K}(\delta).
\end{aligned}
\end{equation}

\begin{corollary}[Trace-norm continuity of the fermion correction]
\label{cor:fermion-channel-correction}
In either coupling regime, \(\cR_{\lambda,K}^+(E_++\delta)\)
extends continuously in trace norm to \(\delta=0\).
It is self-adjoint and nonnegative in the weak regime, and
nonpositive in the separated regime. Its threshold limit is
\[
\cR_{\lambda,K}^+(E_+)
=2\mu\lambda L_{\lambda,K}(0)^*J_0L_{\lambda,K}(0),
\qquad J_0=I\ \text{or}\ -I,
\]
respectively. The family is trace-norm continuous on every compact
interval from \(E_+\) to an interior point of the corresponding
right-hand channel resolvent interval.
\end{corollary}

\begin{proof}
Apply Lemma~\ref{lem:fermion-weighted-HS} and the ideal inequality
\[
\|L_\delta^*J_0L_\delta-L_0^*J_0L_0\|_{\mathfrak S_1}
\leq(\|L_\delta\|_{\mathfrak S_2}+\|L_0\|_{\mathfrak S_2})
\|L_\delta-L_0\|_{\mathfrak S_2}.
\]
Away from the threshold all inverse multipliers are locally bounded
and norm-continuous, and the transition kernels are continuous in
Hilbert--Schmidt norm.
\end{proof}

\subsubsection{Birman--Schwinger reduction and counting}
For \(z>E_+\) in the fermion-channel resolvent set, define
\begin{equation}
\label{eq:gap-effective-BS}
\notationedit{\widetilde T}_{\lambda,K}^+(z)=T_K^+(z)+\cR_{\lambda,K}^+(z)
\end{equation}
and the generally noncompact family
\begin{equation}
\label{eq:gap-ordinary-BS}
\mathbb B_{\lambda,K}^+(z)
=-2\mu\Psi(H_{12,\lambda}^{\mathrm{ch}}(K)-z)^{-1}\Psi^*.
\end{equation}
Substituting \eqref{eq:fermion-channel-resolvent} and using the
congruence in Theorem~\ref{thm:upper-BS-principle} yields
\begin{equation}
\label{eq:gap-BS-congruence}
\mathbb B_{\lambda,K}^+(z)-I
=D_z[\notationedit{\widetilde T}_{\lambda,K}^+(z)-I]D_z.
\end{equation}
Here \(\notationedit{\widetilde T}_{\lambda,K}^+(z)\) is compact and self-adjoint,
and \(D_z\) is boundedly invertible for each fixed \(z>E_+\).
\reviewedit{The congruence preserves positive inertia, so
\(n(1,\mathbb B_{\lambda,K}^+(z))
=n(1,\widetilde T_{\lambda,K}^+(z))\).}

For \(0\le\lambda\le\lambda_{\mathrm{crit}}\),
\(z-H_{12,\lambda}^{\mathrm{ch}}(K)>0\). The ordinary
Birman--Schwinger argument of Theorem~\ref{thm:upper-BS-principle},
with this channel as reference operator, gives
\begin{equation}
\label{eq:weak-BS-counting}
N_+^{(\lambda)}(K,z)
=n(1,\mathbb B_{\lambda,K}^+(z))
=n(1,\notationedit{\widetilde T}_{\lambda,K}^+(z)),\qquad z>E_+.
\end{equation}
In the gap we use the corresponding relative identity. For its
general index formulation, see \cite{Pushnitski2009}; the following
proof records the crossing direction and endpoint convention for
our family.

\begin{proposition}[Birman--Schwinger counting in the gap]
\label{prop:gap-BS-counting}
Let \(\lambda>\lambda_{\mathrm{sep}}(K,\gamma)\), and fix
\[
\zeta\in\cG_\lambda^+(K)
\setminus\spec(H_{\lambda,\mu,\gamma}(K)).
\]
Then, for every \(E_{+,\gamma}(K)<z<\zeta\),
\begin{equation}
\label{eq:gap-BS-counting}
N_{\mathrm{gap}}^{(\lambda)}(K,z;\zeta)
=
n(1,\notationedit{\widetilde T}_{\lambda,K}^+(z))
-n(1,\notationedit{\widetilde T}_{\lambda,K}^+(\zeta)).
\end{equation}
\end{proposition}

\begin{proof}
Write \(H=H_{\lambda,\mu,\gamma}(K)\) and
\(B(z)=\mathbb B_{\lambda,K}^+(z)\). The channel resolvent exists
throughout \(\cG_\lambda^+(K)\) by
Proposition~\ref{prop:fermion-all-couplings}(iii). By
Corollary~\ref{cor:fermion-channel-correction} and the formula for
\(\cB_K^+(z)\) in Theorem~\ref{thm:upper-BS-principle},
\[
\sup\sigma_{\mathrm{ess}}(B(z))
=1-\min_q\Delta_K^+(q,z)<1.
\]
This separation is uniform on compact subintervals of the gap.
The resolvent identity gives the multiplicity-preserving maps
\[
g=\Psi f,\qquad
f=-2\mu(H_{12,\lambda}^{\mathrm{ch}}(K)-z)^{-1}\Psi^*g
\]
between \(\ker(H-z)\) and \(\ker(I-B(z))\).
The family \(B\) is norm-analytic in the gap, and for any nonzero
\(g\in\ker(I-B(z))\),
\[
\langle g,B'(z)g\rangle
=-2\mu\|(H_{12,\lambda}^{\mathrm{ch}}(K)-z)^{-1}\Psi^*g\|^2<0.
\]
The strict inequality follows because \(B(z)g=g\ne0\) implies
\(\Psi^*g\ne0\), and the channel resolvent is injective.
Thus the compression of \(B'(z)\) to the finite-dimensional
crossing eigenspace is strictly negative, so every crossing of
level one is downward. There are finitely many crossings on each
compact subinterval, since the spectrum of \(H\) there is discrete.
Their multiplicities therefore give
\[
N_{\mathrm{gap}}^{(\lambda)}(K,z;\zeta)
=n(1,B(z))-n(1,B(\zeta)).
\]
By the convention \(n(\alpha,A)=\dim\Ran\mathbf1_{(\alpha,\infty)}(A)\),
level one is excluded, so an eigenvalue at \(z\) is not counted
in \((z,\zeta)\). Finally,
\eqref{eq:gap-BS-congruence} gives
\(n(1,B(w))=n(1,\notationedit{\widetilde T}_{\lambda,K}^+(w))\) at both endpoints.
\end{proof}

\reviewedit{
To retain the fermion channel for arbitrary \(\lambda\ge0\), define
the compact part of the full channel system on
\(L^2(\T^3,\eta;\C^3)\oplus L^2(\T^3,\eta)\) by
\begin{equation}
\label{eq:full-channel-compact-part}
\mathbb C_{\lambda,K}^+(z)
=\begin{pmatrix}
0&-\sqrt{2\mu\lambda}\,\cW_K(z)\\
-\sqrt{2\mu\lambda}\,\cW_K(z)^*&T_K^+(z)
\end{pmatrix},\qquad z>E_+.
\end{equation}

\begin{proposition}[Stability under the fermion--fermion perturbation]
\label{prop:lambda-stability}
For every fixed \(\lambda\ge0\) and every \(\alpha>0\),
\begin{equation}
\label{eq:all-coupling-compact-asymptotic}
\lim_{\delta\downarrow0}
\frac{n\bigl(\alpha,\mathbb C_{\lambda,K}^+(E_++\delta)\bigr)}
{|\log\delta|}
=\cU_K^+(\alpha).
\end{equation}
For the compact part \(\mathbb C_{\lambda,K}^+\), the level-one
coefficient is therefore \(\cU_K^+(1)=\cU_+(K)\) for every
\(\lambda\ge0\).
If, in addition, \(0\leq\lambda\leq\lambda_{\mathrm{crit}}(K,\gamma)\)
or \(\lambda>\lambda_{\mathrm{sep}}(K,\gamma)\), then
\begin{equation}
\label{eq:lambda-stability-count}
\lim_{\delta\downarrow0}
\frac{n\bigl(1,\widetilde T_{\lambda,K}^+(E_++\delta)\bigr)}
{|\log\delta|}
=\cU_+(K),
\end{equation}
where \(\delta>0\) is sufficiently small that \(E_++\delta\)
lies in the right-hand resolvent interval of the fermion channel.
\end{proposition}

\begin{proof}
We first identify the full channel system. Put
\(F_z=z-H_\gamma^{\mathrm{free}}(K)>0\) on \(\cH^{\mathrm a}\),
and set
\[
\mathcal L_\lambda f
=(\sqrt\lambda\,\Phi_Kf,\sqrt{2\mu}\,\Psi f),\qquad
S_z=\operatorname{diag}(I,D_z),\qquad
M_{b,\delta}=M_{\operatorname{diag}(b_j(K,\cdot,\delta))_{j=1}^3}.
\]
For \(z=E_++\delta\), the interaction factorizations and the
mixed-channel congruence in Theorem~\ref{thm:upper-BS-principle} give
\[
\mathcal L_\lambda F_z^{-1}\mathcal L_\lambda^*-I
=S_z\left[
\operatorname{diag}(\lambda M_{b,\delta},0)
+\mathbb C_{\lambda,K}^+(z)-I
\right]S_z.
\]
This identity uses the free resolvent and is valid for every
\(\lambda\ge0\) and \(z>E_+\).

Now write
\[
\mathbb C_{\lambda,K}^+(E_++\delta)
=\operatorname{diag}(0,T_K^+(E_++\delta))+Q_{\lambda,K}(\delta),
\]
where
\[
Q_{\lambda,K}(\delta)
=-\sqrt{2\mu\lambda}
\begin{pmatrix}
0&\cW_K(E_++\delta)\\
\cW_K(E_++\delta)^*&0
\end{pmatrix}.
\]
Proposition~\ref{prop:upper-BS-properties} shows that the principal
family \(A_\delta=\operatorname{diag}(0,T_K^+(E_++\delta))\)
is compact, self-adjoint and norm-continuous for \(\delta>0\).
The block formula defines a self-adjoint \(Q_{\lambda,K}(\delta)\)
also at \(\delta=0\), and Lemma~\ref{lem:fermion-channel-HS} yields
\[
\|Q_{\lambda,K}(\delta)-Q_{\lambda,K}(0)\|_{\mathfrak S_2}
=2\sqrt{\mu\lambda}\,
\|\cW_K(E_++\delta)-\cW_K(E_+)\|_{\mathfrak S_2}
\longrightarrow0.
\]
The same lemma gives Hilbert--Schmidt continuity for \(\delta>0\);
thus \(Q_{\lambda,K}\) is continuous in \(\mathfrak S_2\), hence
in operator norm, up to \(\delta=0\).
For every \(\alpha>0\),
\(n(\alpha,A_\delta)=n(\alpha,T_K^+(E_++\delta))\).
With \(f(\delta)=|\log\delta|^{-1}\) and \(0<\delta_0<1\),
Corollary~\ref{cor:upper-model-HS-stability} supplies the required
counting limits, whose finiteness and continuity follow from
Lemma~\ref{lem:counting-density-regularity}.
All hypotheses of Lemma~\ref{lem:upper-HS-stability} are therefore
satisfied on the direct-sum space, proving
\eqref{eq:all-coupling-compact-asymptotic}.

In the weak and separated regimes, apply the same lemma to
\[
\widetilde T_{\lambda,K}^+(E_++\delta)
=T_K^+(E_++\delta)+\cR_{\lambda,K}^+(E_++\delta).
\]
Choose \(\delta_0<1\) so that \(E_++\delta_0\) lies in the
right-hand channel resolvent interval. The hypotheses for the
principal family and its counting limits have just been verified.
Corollary~\ref{cor:fermion-channel-correction} gives self-adjointness
and trace-norm continuity of the correction on \([0,\delta_0]\),
hence compactness and operator-norm continuity there.
At \(\lambda=\lambda_{\mathrm{crit}}\), this follows from the
weighted Hilbert--Schmidt limit in
Lemma~\ref{lem:fermion-weighted-HS}; no bounded inverse of
\(D_{\lambda,K}(E_+)\) is needed. The stability lemma therefore
gives \eqref{eq:lambda-stability-count} at level one throughout
the stated regimes, including the critical endpoint.
\end{proof}

The all-coupling assertion \eqref{eq:all-coupling-compact-asymptotic}
concerns the compact part of the channel system. The full channel
system also contains the multiplication block \(\lambda M_{b,\delta}\),
which is noncompact for \(\lambda>0\).
The discrete counts are related to
\eqref{eq:lambda-stability-count} by the weak-regime and gap
identities above.
}

\subsection{Completion of the proof of Theorem~\ref{thm:main-upper}}
\label{subsec:main-upper-proof}

\begin{proof}[Proof of Theorem~\ref{thm:main-upper}]
\compactedit{The essential-spectrum assertions and part~(iii) follow from
Corollaries~\ref{cor:weak-regime}--\ref{cor:intervening-regime}.
For part~(i), \eqref{eq:weak-BS-counting} and
Proposition~\ref{prop:lambda-stability} give
\eqref{eq:main-upper-asymptotic} throughout the closed
weak-coupling interval, including
\(\lambda=\lambda_{\mathrm{crit}}(K,\gamma)\).}

For \(\lambda>\lambda_{\mathrm{sep}}\), fix \(\zeta\) as in part (ii).
Proposition~\ref{prop:gap-BS-counting} gives
\[
N_{\mathrm{gap}}^{(\lambda)}(K,E_++\delta;\zeta)
=n(1,\notationedit{\widetilde T}_{\lambda,K}^+(E_++\delta))
 -n(1,\notationedit{\widetilde T}_{\lambda,K}^+(\zeta)).
\]
\compactedit{The second term is finite by compactness of
\(\notationedit{\widetilde T}_{\lambda,K}^+(\zeta)\) and is independent of
\(\delta\).}

Division by
\(|\log\delta|\) and Proposition~\ref{prop:lambda-stability}
give \eqref{eq:upper-gap-Efimov-asymptotic} with the same coefficient.
By Proposition~\ref{prop:upper-supercriticality}, that coefficient
is positive. In either regime the relevant count tends to infinity,
while every compact interval strictly inside the resolvent interval
of the channel contains only finitely many eigenvalues of the full
Hamiltonian. The resulting sequences therefore accumulate at \(E_+\).
\end{proof}

\begin{corollary}[Exponential accumulation law]
\label{cor:upper-exponential-law}
Under the assumptions of Theorem~\ref{thm:main-upper}(i), let
\(E_1^+(K,\lambda)\ge E_2^+(K,\lambda)\ge\cdots>E_{+,\gamma}(K)\) be the
eigenvalues above \(E_{+,\gamma}(K)\), with multiplicities. Then
\begin{equation}
\label{eq:upper-exponential-law}
E_n^+(K,\lambda)-E_{+,\gamma}(K)
=\exp\left[-\frac{n}{\cU_+(K)}+o(n)\right],
\qquad n\to\infty.
\end{equation}
\end{corollary}

\begin{proof}
Apply Proposition~\ref{prop:counting-exponential-equivalence} of
Appendix~\ref{app:abstract-counting} to
\eqref{eq:main-upper-asymptotic}.
\end{proof}

\begin{corollary}[Exponential accumulation in the gap]
\label{cor:gap-exponential-law}
Under the assumptions of Theorem~\ref{thm:main-upper}(ii), fix an
admissible \(\lambda\) and a reference energy \(\zeta\). Enumerate the
eigenvalues in \((E_{+,\gamma}(K),\zeta)\) in nonincreasing order,
with multiplicities, as \(E_n^+(K,\lambda;\zeta)\). Then
\[
E_n^+(K,\lambda;\zeta)-E_{+,\gamma}(K)
=
\exp\left[-\frac{n}{\cU_+(K)}+o(n)\right],
\qquad n\to\infty.
\]
\end{corollary}

\begin{proof}
Apply Proposition~\ref{prop:counting-exponential-equivalence} to
\eqref{eq:upper-gap-Efimov-asymptotic}.
\end{proof}

\begin{remark}[Fixed parameters]
The result holds for each fixed admissible \(\lambda\) and interior
reference energy \(\zeta\). It makes no assertion about possible
accumulation at the upper endpoint \(\tau_{12}^+(K,\lambda)\).
\end{remark}

\subsection{Strong-coupling asymptotics of the upper gap endpoint}
\label{subsec:upper-edge-asymptotic}
We now let \(\lambda\to\infty\) and refine the lower bound
\(\tau_{12}^+(K,\lambda)\ge\lambda+12\) for the upper gap endpoint.

\begin{theorem}[Asymptotics of the lower edge of the upper
two-cluster essential spectrum]
\label{thm:tau12-asymptotic}
For each fixed \(K\in\cO_+\), as \(\lambda\to+\infty\),
\begin{equation}
\label{eq:tau12-asymptotic}
\tau_{12}^+(K,\lambda)
=
\lambda+12+\frac{b_*(K)}{\lambda}
+O(\lambda^{-2}),
\end{equation}
where
\begin{equation}
\label{eq:tau12-b-star}
b_*(K)
=
8\sum_{\ell=1}^3\cos^2\frac{K_\ell}{2}
-
4\max_{1\le j\le3}\cos^2\frac{K_j}{2}
\ge0.
\end{equation}
At \(K=\pi\mathbf{1}\),
\(\tau_{12}^+(\pi\mathbf{1},\lambda)=\lambda+12\) holds exactly
for all \(\lambda>24\).
\revblock{At this momentum, the interval
\((E_{+,\gamma}(\pi\mathbf1),\tau_{12}^+(\pi\mathbf1,\lambda))\)
is a nonempty internal gap precisely when
\(\lambda>12(1+\gamma)=\lambda_{\mathrm{sep}}(\pi\mathbf1,\gamma)\).}
\end{theorem}

\begin{proof}
We expand \eqref{eq:fermion-pair-scalar-roots} as
\(\lambda\to\infty\). \revblock{Put
\[
v_k(x)=+4\sum_{\ell=1}^3\cos(k_\ell/2)\cos x_\ell,
\qquad u_j=\nu_{12,j}^+(k;\lambda)-12=\lambda+O(1).
\]
The \(O(1)\) bound is uniform in \(k,j\) by
\eqref{eq:strong-pair-eigenvalue-bounds}.
At $w=\nu_{12,j}^+(k;\lambda)$, one has
$w-\mathcal E_{12}(k;p)=u_j+v_k(x)$, and the resolvent
expansion reads
\[
\frac{1}{u_j+v_k(x)}
=\frac1{u_j}-\frac{v_k(x)}{u_j^2}
+\frac{v_k(x)^2}{u_j^3}+O(u_j^{-4}).
\]}
Since \(\|v_k\|_\infty\le12\) and \(u_j\ge\lambda>24\), the
resolvent remainder is uniform in \(k,j\). \revblock{The zeroth moment is one; translation
by \(\pi\mathbf1\) preserves \(2\sin^2x_j\) and changes the sign
of \(v_k\), so the first moment vanishes.} Hence
\[
1=\frac{\lambda}{u_j}
  +\frac{\lambda \mathfrak m_{2,j}(k)}{u_j^3}+O(\lambda^{-3}),
\]
where
\[
\begin{aligned}
\mathfrak m_{2,j}(k)
&=\int_{\T^3}2\sin^2x_j\,v_k(x)^2\,\eta(dx)\\
&=8\sum_{\ell=1}^3\cos^2\frac{k_\ell}{2}
  -4\cos^2\frac{k_j}{2}.
\end{aligned}
\]
Multiplying the scalar equation by \(u_j\) gives
\[
\nu_{12,j}^+(k;\lambda)
=\lambda+12+\lambda^{-1}\mathfrak m_{2,j}(k)+O(\lambda^{-2}).
\]
The remainder is uniform in \(k\) and \(j\). Let
\(d(r)=d_{\T^3}(r,0)\). Since \(\gamma>0\), there are constants
\(c,L>0\) such that
\[
\gamma\varepsilon(r)\ge c d(r)^2,\qquad
|\mathfrak m_{2,j}(K-r)-\mathfrak m_{2,j}(K)|\le Ld(r).
\]
Hence
\[
\gamma\varepsilon(r)+\lambda^{-1}\mathfrak m_{2,j}(K-r)
\ge\lambda^{-1}\mathfrak m_{2,j}(K)-\frac{L^2}{4c\lambda^2}.
\]
An upper bound with leading term \(\lambda^{-1}\mathfrak m_{2,j}(K)\)
follows by taking \(r=0\).
Minimizing also over \(j\) gives
\(b_*(K)=\min_j \mathfrak m_{2,j}(K)\) and proves
\eqref{eq:tau12-asymptotic}. At \(K=\pi\mathbf1\) and \(r=0\), the pair dispersion is constant,
so all three pair eigenvalues equal \(\lambda+12\). Thus
\(\tau_{12}^+\le\lambda+12\), and
\eqref{eq:fermion-upper-channel-bottom} gives equality for every
\(\lambda>24\).
\end{proof}

\subsection{Spectral reflection and the attractive counterpart}
\label{subsec:spectral-reflection}

\begin{remark}[Spectral reflection for the cosine dispersion]
\label{rem:spectral-reflection}
Set \(K_0=\pi\mathbf1\) and \(C_\gamma=12(2+\gamma)\), and define
the unitary involution on \(\cH^{\mathrm a}\) by
\[
(\mathcal Jf)(p,q)=f(p+K_0,q+K_0)
\quad\text{for almost every }(p,q).
\]
Using \eqref{eq:free-three}, \(2K_0=0\) on \(\T^3\), and
\(\varepsilon(t+K_0)=12-\varepsilon(t)\), we obtain
\[
\begin{aligned}
\E_\gamma(K;p+K_0,q+K_0)
&=24-\varepsilon(p)-\varepsilon(q)+\gamma\varepsilon(K-p-q)\\
&=C_\gamma-\E_\gamma(K-K_0;p,q).
\end{aligned}
\]
The interaction operators are invariant under this simultaneous
translation, as follows directly from
\eqref{eq:fiber-V13}--\eqref{eq:fiber-fermion-interaction}. Hence,
using the same Hamiltonian formula for real couplings,
\begin{equation}
\label{eq:spectral-reflection}
\mathcal JH_{\lambda,\mu,\gamma}(K)\mathcal J^{-1}
=C_\gamma I-H_{-\lambda,-\mu,\gamma}(K-K_0).
\end{equation}
This relation reflects both the spectrum and the essential spectrum
about \(C_\gamma/2\), preserving eigenvalue multiplicities.
In particular, the upper-threshold formulation for repulsive
interactions corresponds to a lower-threshold formulation for
attractive interactions; an internal gap remains an internal gap.
\revblock{The involution \(\mathcal J\) commutes with \(P_{12}\),
since the simultaneous translation acts on both variables in the
same way. Hence it preserves \(\cH^{\mathrm a}\).}
\end{remark}

\begin{corollary}[Attractive lower-threshold counterpart]
\label{cor:lower-reflection}
Let \(\gamma>\gamma_*\), \(\cO_- = \cO_+-K_0\), and
\(K\in\cO_-\), where \(K_0\) and \(C_\gamma\) are as in
Remark~\ref{rem:spectral-reflection}. Put \(K^+=K+K_0\) and
\[
\mu_{\mathrm c}^-(K,\gamma)=\mu_{\mathrm c}^+(K^+,\gamma),
\qquad \cU_-(K)=\cU_+(K^+)>0,
\qquad
\widetilde H_\lambda(K)=H_{-\lambda,-\mu_{\mathrm c}^-(K,\gamma),\gamma}(K).
\]
Then \(E_{-,\gamma}(K)=C_\gamma-E_{+,\gamma}(K^+)\), and the
following assertions hold, with all parameters fixed as
\(\delta\downarrow0\).
\begin{enumerate}[(i)]
\item For \(0\leq\lambda\leq\lambda_{\mathrm{crit}}(K^+,\gamma)\),
\(\inf\sess\widetilde H_\lambda(K)=E_{-,\gamma}(K)\), and
\[
\dim\Ran\mathbf1_{(-\infty,E_{-,\gamma}(K)-\delta)}
\bigl(\widetilde H_\lambda(K)\bigr)
=\cU_-(K)|\log\delta|+o(|\log\delta|).
\]
\item For \(\lambda>\lambda_{\mathrm{sep}}(K^+,\gamma)\), put
\[
\tau_{12}^-(K,\lambda)=C_\gamma-\tau_{12}^+(K^+,\lambda).
\]
The interval
\(\cG_\lambda^-(K)=(\tau_{12}^-(K,\lambda),E_{-,\gamma}(K))\)
is a nonempty gap in \(\sess\widetilde H_\lambda(K)\), with both
endpoints in the essential spectrum. For every fixed
\(\zeta\in\cG_\lambda^-(K)\setminus\spec\widetilde H_\lambda(K)\),
\[
\dim\Ran\mathbf1_{(\zeta,E_{-,\gamma}(K)-\delta)}
\bigl(\widetilde H_\lambda(K)\bigr)
=\cU_-(K)|\log\delta|+o(|\log\delta|).
\]
\end{enumerate}
Thus infinitely many discrete eigenvalues accumulate at
\(E_{-,\gamma}(K)\) from below; in part~(ii), this is the upper
endpoint of the internal gap. The coefficient is the same as in
the reflected repulsive problem and is independent of \(\lambda\)
in the stated regime. All counts include multiplicities.
\end{corollary}

\begin{proof}
By \eqref{eq:spectral-reflection}, \(\widetilde H_\lambda(K)\)
is unitarily equivalent to
\(C_\gamma I-H_{\lambda,\mu_{\mathrm c}^+(K^+,\gamma),\gamma}(K^+)\).
Spectral reflection maps \((E_{+,\gamma}(K^+)+\delta,\infty)\)
onto \((-\infty,E_{-,\gamma}(K)-\delta)\), and
\((E_{+,\gamma}(K^+)+\delta,C_\gamma-\zeta)\) onto
\((\zeta,E_{-,\gamma}(K)-\delta)\), preserving spectral-projection
ranks. Both assertions therefore follow from
Theorem~\ref{thm:main-upper}, including the endpoint statements.
\end{proof}

\section{Concluding remarks}
\label{sec:conclusion}
\compactedit{At critical mixed coupling, Theorem~\ref{thm:main-upper} gives the
same logarithmic coefficient for every fixed
\(0\leq\lambda\leq\lambda_{\mathrm{crit}}(K,\gamma)\) or
\(\lambda>\lambda_{\mathrm{sep}}(K,\gamma)\). In the separated regime,
the free upper edge is the lower gap endpoint. For intervening
couplings, two-cluster spectrum excludes discrete accumulation
immediately above that edge.

Extensions to two identical bosons or three distinguishable particles
require identifying the resonant channels, estimating the remaining contributions, and
analysing the corresponding homogeneous operator. Other dispersions
require new threshold expansions; degenerate extrema may change
the asymptotic law.
Pair interactions of rank at least two with attractive and repulsive
components may allow simultaneous Efimov accumulation at both band
edges for the same parameters, or at endpoints of distinct internal
gaps. Each endpoint requires suitable spectral geometry, a resonance,
and a supercritical homogeneous operator.

For the cosine dispersion, the reflection identity in
Remark~\ref{rem:spectral-reflection} extends to other pair interactions
satisfying \(\mathcal J V_\alpha\mathcal J^{-1}=V_\alpha\).}

\appendix
\section{Abstract counting-to-accumulation equivalence}
\label{app:abstract-counting}
The following elementary equivalence gives
Corollaries~\ref{cor:upper-exponential-law} and
\ref{cor:gap-exponential-law} from the corresponding counting
asymptotics. Multiplicities are included throughout.

\begin{proposition}[Counting asymptotics and exponential accumulation]
\label{prop:counting-exponential-equivalence}
Let \(E_*\in\R\), and let
\[
E_1\geq E_2\geq\cdots>E_*,
\qquad E_n\longrightarrow E_*,
\]
where every term is repeated according to its multiplicity. Define
\[
N(E_*+\delta)=\#\{n\in\N:E_n>E_*+\delta\},
\qquad \delta>0.
\]
For every \(\cU>0\), the following assertions are equivalent:
\begin{align}
N(E_*+\delta)
&=\cU|\log\delta|+o(|\log\delta|),
&&\delta\downarrow0,
\label{eq:abstract-counting-asymptotic}\\
\log(E_n-E_*)
&=-\frac{n}{\cU}+o(n),
&&n\to\infty,
\label{eq:abstract-log-eigenvalue-law}\\
(E_n-E_*)^{1/n}
&\longrightarrow e^{-1/\cU},
&&n\to\infty.
\label{eq:abstract-root-eigenvalue-law}
\end{align}
Equivalently, \eqref{eq:abstract-log-eigenvalue-law} can be written
as
\[
E_n-E_*=\exp\left[-\frac{n}{\cU}+o(n)\right].
\]
\end{proposition}

\begin{proof}
Put \(a_n=-\log(E_n-E_*)\) and
\(A(t)=N(E_*+e^{-t})=\#\{n:a_n<t\}\). Then \(a_n\uparrow\infty\).
If \(A(t)=\cU t+o(t)\), then for every \(0<\epsilon<1\) and all
sufficiently large \(n\),
\[
A((1-\epsilon)n/\cU)<n<A((1+\epsilon)n/\cU),
\]
whence \((1-\epsilon)n/\cU\le a_n<(1+\epsilon)n/\cU\).
Thus \(a_n/n\to1/\cU\). Conversely, these bounds on \(a_n\),
with \(m=A(t)\) and \(a_m<t\le a_{m+1}\), give
\[
\frac{\cU}{1+\epsilon}\frac{m}{m+1}
<\frac{m}{t}<\frac{\cU}{1-\epsilon}
\]
for large \(t\). Since \(m\to\infty\), letting \(t\to\infty\)
and then \(\epsilon\downarrow0\) proves \(A(t)/t\to\cU\).
The remaining equivalence follows by taking logarithms.
\end{proof}

\begin{remark}
\label{rem:no-geometric-ratio}
The equivalence in
Proposition~\ref{prop:counting-exponential-equivalence} concerns
only the leading logarithmic scale. In particular,
\eqref{eq:abstract-log-eigenvalue-law} does not imply the existence
of the limit of the ratios
\((E_{n+1}-E_*)/(E_n-E_*)\).
\end{remark}

\section{Technical lemmas}
\label{app:technical-lemmas}
\edit{This appendix contains four technical lemmas: finiteness
and analyticity of the mixed-pair threshold integral
\(J_\gamma(k)\), the square-root expansion at the mixed upper
edge, and two estimates for the maxima of the fermion-channel
threshold integrals. The latter are used in
Subsection~\ref{subsec:fermion-interaction} and in the critical
weighted estimate of Lemma~\ref{lem:fermion-weighted-HS}.}
The square-root law is the lattice analogue of the threshold
behaviour used in the continuum Efimov analysis of
\cite{Sobolev1993}; for the lattice case, see also
\cite[Section~5]{ALM2004} and \cite[Section~5]{ALD2007}.

\subsection{Finiteness and analyticity of the threshold integral}
\label{subsec:app-J}

\begin{lemma}[Finiteness and analyticity of \(J_\gamma(k)\)]
\label{lem:J-finite}
Fix \(\gamma>0\). For every regular \(k\), \(0<J_\gamma(k)<\infty\).
The function \(J_\gamma\) has a holomorphic extension to a complex
neighbourhood of each regular point. It is therefore real analytic
on the regular region. If \(\gamma\ne1\), this region is the whole
torus \(\T^3\). Moreover, \(J_\gamma\) is even in each component
of \(k\), and \(\nabla J_\gamma(0)=0\).
\end{lemma}

\begin{proof}
By the exact normal form \eqref{eq:mixed-upper-normal-form},
\[
J_\gamma(k)=\int_{\T^3}\frac{\eta(du)}{a(k,u)},
\qquad a(k,u)=2\sum_{j=1}^3r_{\gamma,j}(k)(1-\cos u_j).
\]
For regular real \(k\), the denominator is positive away from zero
and comparable to \(|u|^2\) near zero. Since \(|u|^{-2}\) is locally
integrable in three dimensions, \(0<J_\gamma(k)<\infty\).
Fix a regular point \(k_0\) and choose local coordinates near it.
The positive square-root branches in \eqref{eq:r-gamma} at \(k_0\)
extend holomorphically to a complex neighbourhood of \(k_0\).
After shrinking this neighbourhood, one has, uniformly for \(u\ne0\),
\[
|a(k,u)-a(k_0,u)|\le c\,a(k_0,u),\qquad 0<c<1.
\]
Thus \(a(k,u)^{-1}\) is holomorphic in \(k\) and bounded by the
integrable function \([(1-c)a(k_0,u)]^{-1}\). The holomorphic
parameter integral theorem proves the asserted local extension.
For \(\gamma\ne1\), the bound
\(r_{\gamma,j}(k)\ge|1-\gamma|>0\) makes every point regular.
Finally, each \(r_{\gamma,j}\) is even in \(k_j\), so the integral
representation gives evenness of \(J_\gamma\) in each component.
Since zero is regular for every \(\gamma>0\), it follows that
\(\nabla J_\gamma(0)=0\).
\end{proof}

\subsection{Square-root law at the mixed upper edge}
\label{subsec:app-square-root}

\begin{lemma}[Square-root law at the mixed upper edge]
\label{lem:upper-square-root-law}
Fix \(\gamma>0\). Let \(U\) be a sufficiently small
neighbourhood of \(0\) whose closure lies in the regular region
of the mixed-pair dispersion. For \(k\in U\) and \(s\ge0\), set
\[
J_\gamma(k,s)
=
\int_{\T^3}
\frac{\eta(dp)}
{M_\gamma(k)-\E_\gamma(k;p)+s}.
\]
Then \(J_\gamma(k,0)=J_\gamma(k)\), as defined in
\eqref{eq:J-gamma}, and
\begin{equation}
\label{eq:upper-square-root-law}
J_\gamma(k,s)
=
J_\gamma(k,0)
-
\frac{\sqrt{s}}
{4\pi\sqrt{\det B_\gamma^+(k)}}
+O(s),
\qquad s\downarrow0,
\end{equation}
uniformly for \(k\in U\). Here
\[
B_\gamma^+(k)
=
-\frac12D_p^2\E_\gamma(k;p^+(k)).
\]
\end{lemma}

\begin{proof}
Write
\[
r_j(k)
=
\sqrt{1+\gamma^2+2\gamma\cos k_j},
\qquad j=1,2,3.
\]
The neighbourhood \(U\) can be chosen so that the functions
\(r_j(k)\) are bounded above and bounded away from zero on
\(\overline U\). Translating \(p\) by the unique maximiser
\(p^+(k)\), we obtain the exact identity
\[
a_k(u)
:=
M_\gamma(k)-\E_\gamma(k;p^+(k)+u)
=
2\sum_{j=1}^3 r_j(k)(1-\cos u_j).
\]
Consequently,
\[
B_\gamma^+(k)
=
\operatorname{diag}(r_1(k),r_2(k),r_3(k)).
\]
Set \(q_k(u)=\langle B_\gamma^+(k)u,u\rangle\). On a
fixed sufficiently small ball \(D\) centred at the origin,
\[
a_k(u)=q_k(u)+O(|u|^4),
\qquad
c|u|^2\le a_k(u),q_k(u)\le C|u|^2,
\]
with constants independent of \(k\in U\).
For \(s>0\), put
\[
F_s(t)=\frac1{t+s}-\frac1t
=-\frac{s}{t(t+s)}.
\]
If \(t\asymp |u|^2\), then
\[
|F_s'(t)|
=
\frac{s(2t+s)}{t^2(t+s)^2}
\le
\frac{Cs}{|u|^4(|u|^2+s)}.
\]
The mean-value theorem therefore gives
\[
|F_s(a_k(u))-F_s(q_k(u))|
\le
\frac{Cs}{|u|^2+s},
\qquad u\in D\setminus\{0\}.
\]
Its integral over \(D\) is \(O(s)\), uniformly in \(k\).
On the complement of \(D\), \(a_k\) is uniformly bounded
away from zero; hence the integral of \(F_s(a_k)\) there
is also \(O(s)\). Thus
\[
J_\gamma(k,s)-J_\gamma(k,0)
=
\int_D F_s(q_k(u))\,\eta(du)+O(s)
\]
uniformly for \(k\in U\).
In the remaining integral, let
\(v=B_\gamma^+(k)^{1/2}u\). \reviewedit{Uniform positive
definiteness of \(B_\gamma^+(k)\) implies that
\(B_\gamma^+(k)^{1/2}D\) contains a fixed ball centred at zero.} Since
\[
|F_s(|v|^2)|\le \frac{s}{|v|^4}
\]
outside that ball, extending the integral to \(\mathbb R^3\)
introduces another uniform \(O(s)\) error. Hence
\[
\begin{aligned}
J_\gamma(k,s)-J_\gamma(k,0)
&=
\frac{1}{(2\pi)^3\sqrt{\det B_\gamma^+(k)}}
\int_{\mathbb R^3}
\left(
\frac1{|v|^2+s}-\frac1{|v|^2}
\right)dv
+O(s)
\\
&=
-\frac{\sqrt{s}}
{4\pi\sqrt{\det B_\gamma^+(k)}}
+O(s),
\end{aligned}
\]
because
\[
\frac1{(2\pi)^3}
\int_{\mathbb R^3}
\left(
\frac1{|v|^2+s}-\frac1{|v|^2}
\right)dv
=
-\frac{\sqrt{s}}{4\pi}.
\]
All estimates are uniform for \(k\in U\).
\end{proof}

\revblock{\subsection{Threshold estimates for the fermion channel}
\label{subsec:app-fermion-threshold}

In this subsection we prove the two technical lemmas that were
used in Subsection~\ref{subsec:fermion-interaction} to control
the threshold integrals
\(b_j(K,r,0)\) defined in \eqref{eq:fermion-critical-b}. We keep
the notation of that subsection: \(K\in\cO_+\),
\(E_+=E_{+,\gamma}(K)\), \(k_+=k_K^+\), \(\rho_+=K-2k_+\),
\(K_0=\pi\mathbf1\), \(\beta=\beta(K,\gamma)\) is defined in
\eqref{eq:fermion-critical-coupling}, and
\(d_{\lambda_{\mathrm{crit}},j}(K,r,\delta)\) is defined in
\eqref{eq:fermion-critical-d}. Throughout, \(\gamma>1\) is fixed.

\edit{\begin{lemma}[The threshold integrals at the symmetric point]
\label{lem:fermion-critical-symmetric}
Let \(K_0=\pi\mathbf1\) and \(\gamma>1\). For each
\(j\in\{1,2,3\}\), the function \(r\mapsto b_j(K_0,r,0)\)
has its unique global maximum at \(r=K_0\), and its Hessian
there is strictly negative definite. Moreover, there are
constants \(c,C>0\) such that, for all sufficiently small
\(v\in\mathbb R^3\),
\[
c|v|^2\leq b_j(K_0,K_0,0)-b_j(K_0,K_0+v,0)
\leq C|v|^2.
\]
\end{lemma}
\begin{proof}
Write \(r=K_0+v\), with \(v\in[-\pi,\pi]^3\), and change
variables to \(p=\pi\mathbf1-v/2+u\). Then
\[
\begin{aligned}
a(v,u)&=4\sum_{\ell=1}^3
\bigl[1-\cos(v_\ell/2)\cos u_\ell\bigr]
+2\gamma\sum_{\ell=1}^3(1-\cos v_\ell),\\
a(0,u)&=4\sum_{\ell=1}^3(1-\cos u_\ell),
\end{aligned}
\]
and
\[
b_j(K_0,K_0+v,0)
=2\int_{\T^3}\frac{\sin^2u_j}{a(v,u)}\,\eta(du).
\]
Since \(0\leq\cos(v_\ell/2)\leq1\) and \(\gamma>1\),
\begin{align}
\label{eq:fermion-critical-symmetric-denominator}
a(v,u)-a(0,u)
&=4\sum_{\ell=1}^3[1-\cos(v_\ell/2)]
\{\cos u_\ell+\gamma[1+\cos(v_\ell/2)]\}\notag\\
&\geq c_\gamma|v|^2.
\end{align}
The same difference is bounded above by \(C_\gamma|v|^2\).
Consequently,
\[
b_j(K_0,K_0,0)-b_j(K_0,K_0+v,0)
=2\int_{\T^3}\sin^2u_j
\frac{a(v,u)-a(0,u)}{a(0,u)a(v,u)}\,\eta(du).
\]
Since \(a(0,u),a(v,u)\le L:=12(2+\gamma)\), the last integral
is at least \(c_\gamma|v|^2/L^2\), using
\(2\int_{\T^3}\sin^2u_j\,\eta(du)=1\). For the upper
bound, use \(a(v,u)\geq a(0,u)\) and the integrability of
\(\sin^2u_j/a(0,u)^2\). Strict positivity of the difference
for \(v\ne0\) gives uniqueness of the global maximum.
By the \(C^2\)-regularity established in
Lemma~\ref{lem:fermion-critical-regularity}, the quadratic lower
bound implies that the Hessian at \(K_0\) is strictly negative
definite.
\end{proof}}

\begin{lemma}[Nondegenerate maxima of the threshold integrals]
\label{lem:fermion-critical-maxima}
After reducing \(\cO_+\), for each \(j\), the function
\(r\mapsto b_j(K,r,0)\) has a unique, nondegenerate global
maximum at a point \(r_j(K)\) near \(K_0\), depending
continuously on \(K\).
There are \(c>0\) and \(\delta_0>0\), uniform in this
neighbourhood, such that
\begin{equation}
\label{eq:fermion-critical-quadratic}
b_j(K,r_j(K),0)-b_j(K,r,0)
\geq c\,d_{\T^3}(r,r_j(K))^2,
\qquad r\in\T^3,
\end{equation}
and
\begin{equation}
\label{eq:fermion-critical-bound}
d_{\lambda_{\mathrm{crit}},j}(K,r,\delta)
\geq c\bigl(\delta+d_{\T^3}(r,r_j(K))^2\bigr),
\qquad 0\leq\delta\leq\delta_0.
\end{equation}
For every fixed \(\delta>0\), all three factors have a strictly
positive lower bound over \(r\). At \(\delta=0\), a factor
vanishes precisely when
\[
b_j(K,r_j(K),0)=\beta(K,\gamma)
\quad\text{and}\quad r=r_j(K).
\]
At least one index satisfies this equality.
\end{lemma}
\begin{proof}
At \(K=K_0\), Lemma~\ref{lem:fermion-critical-symmetric} gives
the strict concavity of \(r\mapsto b_j(K_0,r,0)\) near \(K_0\).
By the joint continuity of the second \(r\)-derivatives
established in Lemma~\ref{lem:fermion-critical-regularity}, the
Hessians remain uniformly strictly negative definite in a fixed
small coordinate ball about \(K_0\), for \(K\) sufficiently close
to \(K_0\). Thus strict concavity persists in this ball.
The strict gap between the maximum and the values outside this
ball also persists, by joint continuity and compactness.
Thus each function has a unique global maximiser \(r_j(K)\)
in the ball. Uniqueness and compactness imply continuity of
\(r_j(K)\). Taylor's formula gives the quadratic lower bound
near the maximiser; the uniform gap outside the ball extends
it to all of \(\T^3\).
Finally, put
\(a=E_+-\E_\gamma(K;p,K-r-p)\) and
\(L=12(2+\gamma)\). Since \(0\leq a\leq L\) and
\(\|\phi_{j,K-r}\|_2=1\) by \eqref{eq:odd-form-factors},
\[
\begin{split}
b_j(K,r,0)-b_j(K,r,\delta)
&=\delta\int_{\T^3}
\frac{|\phi_{j,K-r}(p)|^2}{a(a+\delta)}\,\eta(dp)\\
&\geq\frac{\delta}{L(L+\delta_0)},
\qquad 0<\delta\leq\delta_0.
\end{split}
\]
Using \(\lambda_{\mathrm{crit}}\beta=1\), decompose
\[
\begin{split}
d_{\lambda_{\mathrm{crit}},j}(K,r,\delta)
=\lambda_{\mathrm{crit}}\bigl[&\beta-b_j(K,r_j(K),0)\\
&+b_j(K,r_j(K),0)-b_j(K,r,0)\\
&+b_j(K,r,0)-b_j(K,r,\delta)\bigr].
\end{split}
\]
All three terms in brackets are nonnegative.
Joint continuity of the threshold integrals and compactness
of the spectator torus give a uniform upper bound
\(\beta(K,\gamma)\le C_0\) in a sufficiently small neighbourhood
of \(K_0\). Hence \(\lambda_{\mathrm{crit}}=\beta^{-1}\ge C_0^{-1}>0\)
uniformly there.
The last two terms yield the stated lower bound, and the first identifies the indices
whose factors vanish at \(\delta=0\). Strict decrease in
\(\delta\) and compactness give positivity for every fixed
\(\delta>0\).
\end{proof}
}

\section*{Acknowledgment}
The author is grateful to Professor Sergio Albeverio for many years of
inspiring discussions on Efimov physics and for his continuing interest in
the upper-threshold problem studied here.

\section*{Statements and Declarations}
\subsection*{Funding}
This work was supported by Grant OT-4-66 of the Fundamental Science
Foundation of Uzbekistan.

\subsection*{Competing interests}
The author has no relevant financial or non-financial interests to
disclose.

\subsection*{Data availability}
Data sharing is not applicable to this article, as no datasets were
generated or analysed during the current study.

\subsection*{Use of generative AI}
During the preparation of this manuscript, generative AI tools were used for language editing, assistance with the organization and presentation of the text, and as an auxiliary aid for checking mathematical arguments and calculations. All mathematical analysis, computations, proofs, results, and conclusions were performed by the author. The author independently verified all mathematical content and takes full responsibility for the final manuscript.


\end{document}